\documentclass[mathpazo]{cicp}
\usepackage{algorithm}
\usepackage{algpseudocode}
\usepackage{float}
\usepackage[dvipsnames]{xcolor}
\usepackage{amsmath, bm}
\usepackage{braket}

\usepackage{indentfirst}
\usepackage{subfigure}
\usepackage{array}
\usepackage{multirow}
\usepackage{makecell}
\usepackage{booktabs}

\usepackage{amsthm}
\newtheorem{test}{Experiment}

\usepackage[
    colorlinks=false,
    linkcolor=black,
    anchorcolor=black,
    citecolor=black,
    urlcolor=black
]{hyperref}
\usepackage{cleveref}

\usepackage{geometry}
\newcommand{\replyone}[1]{{ #1}}
\newcommand{\replytwo}[1]{{ #1}}
\newcommand{\replythr}[1]{{ #1}}

\begin{document}

\title{Quantum circuits for the advection-diffusion equation with boundary conditions based on LCHS}

\author[Chen L Y et~al.]
{Leyu Chen\affil{1}, Tiegang Liu\affil{1,2}, Liang Xu\affil{3}, and Kun Wang\affil{1,2}\comma\corrauth}

\address{
    \affilnum{1}\ LMIB and School of Mathematical Sciences, Beihang University, Beijing 100191, P.R. China. \\
    \affilnum{2}\ International Research Center for Mathematics and Interdisciplinary Sciences, Hangzhou International Innovation Institute of Beihang University, Hangzhou 311115, P.R. China. \\
    \affilnum{3}\ China Academy of Aerospace Aerodynamics, Beijing 100074, P.R. China.
}

\emails{
    {\tt wangkun@buaa.edu.cn} (K.~Wang).
}

\begin{abstract}
This paper proposes a systematic and explicit quantum circuit framework for solving advection-diffusion equations with boundary conditions, based on the Linear Combination of Hamiltonian Simulations (LCHS) method.
By employing the Finite Volume Method (FVM) combined with various flux construction schemes, we elaborate the design of quantum circuits tailored explicitly for Robin boundary conditions (including Dirichlet and Neumann boundary conditions as special cases) and periodic boundary conditions.
\replytwo{The Robin boundary elimination is encoded by newly constructed endpoint-projector circuits, and all physical and boundary-induced inhomogeneous terms are incorporated through an outer linear combination of unitaries (LCU).
The error analysis adopts an LCU-weighted form, preserving the decay of the LCHS kernel in the product-formula bounds.}
\replytwo{Complementing this analysis,} a comprehensive gate complexity analysis demonstrates the quantum advantages over classical computing in high-dimensional scenarios.
We simulate the proposed circuits on a fault-tolerant emulator, and numerical results validate the effectiveness of the proposed framework across homogeneous, inhomogeneous, and high-dimensional cases.
The proposed framework is compatible with numerous spatial discretization methods and numerical schemes, extends naturally to other linear PDEs, and establishes a practical foundation for solving large-scale PDE problems on future fault-tolerant quantum computers.

\end{abstract}


\maketitle

\section{Introduction}

Partial differential equations (PDEs) constitute a foundational mathematical framework for describing physical phenomena, underpinning applications in fluid dynamics, electromagnetism, quantum mechanics, heat transfer, and beyond.
Consequently, developing efficient numerical methods for solving PDEs is of paramount importance.
Despite the immense computing capabilities of modern supercomputers, solving PDEs remains computationally demanding, particularly for large-scale, high-dimensional, and multi-scale problems, where classical methods often suffer from the curse of dimensionality.

Quantum computing~\cite{feynman1982_simulating, deutsch1985_quantum, nielsen2010_quantum}, an emerging computational paradigm, offers the potential for exponential speedups and has demonstrated quantum advantage over classical computing in representative tasks including integer factorization and cryptography~\cite{shor1997_polynomial, hallgren2007_polynomial, bennett2014_theoretical}, solving linear systems~\cite{harrow2009_quantum, gilyen2019_quantum, subasi2019_quantum, lin2020_optimal, costa2022_optimal}, and quantum system simulation~\cite{lloyd1996_universal, buluta2009_quantum, alan2005_simulated, georgescu2014_quantum}.
This potential has motivated growing interest in quantum algorithms for solving PDEs.
Nevertheless, a fundamental obstacle arises in the construction of quantum algorithms for PDEs: quantum computers support solely unitary operations, whereas the evolution of most practical PDEs is inherently non-unitary.

To resolve this critical mismatch, numerous strategies have been developed, including Carleman embedding~\cite{liu2021_efficient, sanavio2024_three, gonzalez-conde2025_quantum}, the Koopman-von Neumann approach~\cite{joseph2020_koopman, joseph2023_quantum}, the Koopman operator~\cite{giannakis2022_embedding, zhang2025_data}, the Liouville equation~\cite{jin2023_time, succi2024_ensemble}, the Schrödinger-Pauli equation~\cite{meng2023_quantum, meng2024_simulating} and the homotopy analysis method~\cite{xue2025_quantum, bharadwaj2025_quantum, choi2026_lindbladian, yuan2026_quantum} for nonlinear dynamics; as well as linear system discretization~\cite{berry2014_high, berry2017_quantum}, the quantum spectral method~\cite{childs2020_quantum, febrianto2026_a}, the Schr\"{o}dingerization method~\cite{jin2023_quantum, jin2024_quantum, jin2024_quantum_siam, jin2024_quantum_jcp, jin2024_analog, jin2024_quantum_esaim, jin2026_quantum} and the LCHS method~\cite{an2023_linear, an2023_quantum, yang2025_circuit, lu2025_infinite, huang2025_fourier, low2025_optimal, meng2026_toward} for linear equations.

\replythr{Within this linear class, the advection-diffusion equation is of particular interest for two reasons. Algorithmically, by coupling advection with diffusion, it produces dynamics that are simultaneously non-symmetric and dissipative, offering the simplest nontrivial instance of the non-unitary evolution discussed above.}
Practically, it is the prototypical model for transport phenomena in computational fluid dynamics and a stepping stone toward the incompressible Navier–Stokes equations.
\replythr{Several quantum approaches have therefore emerged for this equation in recent years.
One class of approaches combines classical spatial discretization with implicit time stepping to obtain a linear system, which is then solved by quantum algorithms for linear systems, including variational~\cite{demirdjian2022_variational} and HHL-based~\cite{ingelmann2024_two} algorithms.
A complementary line of work employs time-marching strategies, constructing a direct unitary block encoding of the explicit time-marching operator for multidimensional scalar transport~\cite{over2025_quantum} and extending it to linear transport problems with general boundary conditions via the method of images~\cite{bengoechea2026_quantum}.
Most relevant to the present work, the LCHS method has been applied to the advection-diffusion equation through a QSVT-based implementation~\cite{novikau2025_quantum}, an explicit quantum framework for general linear PDEs with Robin boundary conditions has been developed~\cite{guseynov2026_quantum}, and an efficient explicit block-encoding implementation of the LCHS method for linear dissipative equations has been presented~\cite{novikau2025_explicit}.}
\replyone{These advances motivate a complementary, systematic gate-level treatment that compiles the governing equation, its boundary conditions, and the spatial discretization schemes together into explicit quantum circuits.}

Against this backdrop, this work focuses on linear advection-diffusion equations with boundary conditions.
Crucially, real-world transport phenomena rarely occur in isolated, unbounded domains; they are governed by boundary conditions---Dirichlet, Neumann, and the more general Robin conditions, whose incorporation into a quantum circuit is far from trivial.
Specifically, when the governing PDE is semi-discretized using the FVM with diverse flux reconstruction schemes, the resulting system of ODEs exhibits a specific algebraic sparsity and source-term structure\replytwo{, whose translation into explicit quantum circuits constitutes the central challenge addressed in this work}.

\replyone{To address this challenge,} for the advection-diffusion equations considered herein, we adopt the optimal approximate LCHS method~\cite{low2025_optimal} to reformulate the time-evolution operator as a linear combination of unitaries.
In terms of circuit synthesis, our construction is inspired by the scalable framework introduced by Sato et al.~\cite{sato2024_hamiltonian} for wave and Schrödinger-type PDEs, which has been extended to heat and advection equations~\cite{hu2024_quantum}, and further elaborated for heat equations with physical boundary conditions~\cite{jin2025_quantum}.
Building upon these foundational works, \replytwo{the present study develops a structure-preserving compilation from boundary-conditioned finite-volume operators to explicit LCHS selector circuits. The elementary shift-operator primitives are adopted from~\cite{sato2024_hamiltonian}, while the new components---endpoint-projector circuits encoding the Robin boundary corrections, and an assembly handling the wrap-around couplings of periodic closure and the advection-induced noncommutativity---yield an explicit, low-depth sequence of quantum gates for both central and exponential flux schemes. The compilation is accompanied by a comprehensive error analysis combining the LCHS truncation and quadrature errors with an LCU-weighted product-formula analysis that preserves the decay of the LCHS kernel and is built on explicitly computed commutator bounds, a boundary-dependent spectral analysis of the dissipative operator $L$, and explicit gate-level resource estimates. The complete workflow is finally validated through numerical experiments on the fault-tolerant state-vector emulator QPanda2, covering $1$- and $2$-dimensional problems with Dirichlet, Neumann, Robin, and periodic boundary conditions.}

The paper is structured as follows.
Section~\ref{sec:overview of the LCHS method} reviews the LCHS method and discusses the quantum circuit design for implementing both the homogeneous and inhomogeneous terms of an ODE system.
Section~\ref{sec:the advection-diffusion equation} details the advection-diffusion model, FVM discretization, and construction of boundary-condition-dependent coefficient matrices.
Section~\ref{sec:quantum circuits for different boundary conditions} presents the quantum circuits for Robin and periodic boundaries, leveraging matrix product operator representations (adopted from~\cite{sato2024_hamiltonian}).
Section~\ref{sec:complexity analysis} analyzes gate complexity and total error.
Section~\ref{sec:numerical experiments} validates the effectiveness of the proposed quantum circuits through fault-tolerant simulations.
Finally, Section~\ref{sec:conclusion} concludes with future directions.

\section{Overview of the LCHS method}\label{sec:overview of the LCHS method}

\subsection{The LCHS method}

Consider the following system of ordinary differential equations (ODEs):
\begin{equation}\label{eq:ode}
    \frac{\mathrm{d} \mathbf{u}(t)}{\mathrm{d} t} = -A(t)\mathbf{u}(t) + \mathbf{f}(t), \quad \mathbf{u}(0) = \mathbf{u}_0,
\end{equation}
where $A(t) \in \mathbb{C}^{N \times N}$, and $\mathbf{u}(t), \mathbf{f}(t) \in \mathbb{C}^N$. By Duhamel's principle, the solution of Eq.~\eqref{eq:ode} can be represented as
\begin{equation}\label{eq:ode_solution}
    \mathbf{u}(T) = \mathcal{T} e^{-\int_0^T A(s) \, \mathrm{d}s} \mathbf{u}_0 + \int_0^T \mathcal{T} e^{-\int_s^T A(s') \, \mathrm{d}s'} \mathbf{f}(s) \, \mathrm{d}s,
\end{equation}
where $\mathcal{T}$ denotes the time-ordering operator.

Define the real and imaginary parts of $A(t)$ as $L(t) = \frac{A(t)+A(t)^\dagger}{2}, H(t) = \frac{A(t)-A(t)^\dagger}{2i}$ respectively.
If $L(s) \succeq 0$ for all $0 \leq s \leq t$, then
\begin{equation}\label{eq:inf_int}
    \mathcal{T} e^{-\int_0^T A(s) \, \mathrm{d}s} = \frac{1}{\sqrt{2\pi}} \int_{\mathbb{R}} \replythr{\hat{\kappa}}(r) \mathcal{U}(T,r) \, \mathrm{d}r, \quad
    \mathcal{U}(T,r) := \mathcal{T} e^{-i \int_0^T (rL(s) + H(s)) \, \mathrm{d}s}
\end{equation}
hold under some additional requirements imposed on the kernel function $\replythr{\hat{\kappa}}$ of the LCHS method.
Here, the inverse Fourier transform $\replythr{\kappa}$ of $\replythr{\hat{\kappa}}$ satisfies an exponential decay property for $r\geq0$, while it is arbitrary for $r<0$~\cite{an2023_linear, an2023_quantum}.

Truncate the infinite integral on the right-hand side of Eq.~\eqref{eq:inf_int} to the finite interval $[-R, R]$ and apply numerical quadrature to obtain
\begin{equation}\label{eq:integrate}
\begin{aligned}
    \frac{1}{\sqrt{2\pi}} \int_{\mathbb{R}} \replythr{\hat{\kappa}}(r)  \mathcal{U}(T,r) \, \mathrm{d} r
    \approx \frac{1}{\sqrt{2\pi}} \int_{-R}^{R} \replythr{\hat{\kappa}}(r) \mathcal{U}(T,r) \, \mathrm{d} r
    \approx \frac{1}{\sqrt{2\pi}} \sum_{j=0}^{M-1} w_j \replythr{\hat{\kappa}}(r_j) \mathcal{U}(T,r_j).
\end{aligned}
\end{equation}
Here, $M$ denotes the total number of quadrature nodes, while $r_j$ and $w_j$ represent the predefined quadrature nodes and quadrature weights on $[-R, R]$.

In this work, we adopt the kernel function
\begin{equation}
    \replythr{\hat{\kappa}}(r; \gamma, \delta) = \sqrt{\frac{2}{\pi}} \frac{\exp \left( \delta - \frac{1+r^2}{4\gamma^2} - ir\delta \right)}{1+r^2}
\end{equation}
with
\begin{equation}
    \alpha_{\replythr{\hat{\kappa}},\infty} := \frac{1}{\sqrt{2\pi}} \int_\mathbb{R} \left| \replythr{\hat{\kappa}}(r; \gamma, \delta) \right| \, \mathrm{d}r = e^\delta \mathrm{erfc}\left(\frac{1}{2\gamma}\right)
\end{equation}
from~\cite{low2025_optimal}.
Then, for any $\varepsilon_{\mathrm{lchs}} \leq 0.9027$ and $c > 0$, choosing $\gamma = \frac{1}{\delta}\sqrt{\delta + \log \frac{1 + 1/2\pi}{\varepsilon_{\mathrm{lchs}}}}$ and $R = 2\delta\gamma^2 = \mathcal{O} \left(\log \frac{1}{\varepsilon_{\mathrm{lchs}}}\right)$ yields
\begin{equation}
    \left\| \mathcal{T} e^{-\int_0^T A(s) \,\mathrm{d}s} - \frac{1}{\sqrt{2\pi}} \int_{-R}^R  \replythr{\hat{\kappa}}(r; \gamma, \delta) \mathcal{U}(T, r) \, \mathrm{d}r \right\| \leq \frac{2\pi+1}{2\pi} e^{\delta - \delta^2 \gamma^2} = \varepsilon_{\mathrm{lchs}}.
\end{equation}
This further gives
\begin{equation}
    \alpha_{\replythr{\hat{\kappa}},R}
    := \frac{1}{\sqrt{2\pi}} \int_{-R}^{R} \left| \replythr{\hat{\kappa}}(r; \gamma, \delta) \right| \, \mathrm{d}r
    \leq \frac{1}{\sqrt{2\pi}} \int_\mathbb{R} \left| \replythr{\hat{\kappa}}(r; \gamma, \delta) \right| \, \mathrm{d}r
    = e^{\delta} \mathrm{erfc}\left( \frac{1}{2\gamma} \right)
    \leq e^{\delta}.
\end{equation}

For the quadrature scheme, we employ the trapezoidal rule, which is proven to achieve exponential convergence when applied within the LCHS framework using the kernel function $\replythr{\hat{\kappa}}(r; \gamma, \delta)$ \cite{low2025_optimal}.
With the above LCHS parameters $\varepsilon_{\mathrm{lchs}}, R, \gamma, \delta$, for $\varepsilon_{\mathrm{quad}} \leq \frac{4}{15}$, step size $\Delta r \leq \frac{\pi}{\frac{1}{2}\|L\|_{L^1} + \log \frac{64\exp(3\delta/2)}{15\varepsilon_{\mathrm{quad}}}}$ and $R/\Delta r \in \mathbb{Z}$, the trapezoidal rule satisfies the following bound:
\begin{equation}\label{eq:quadrature error}
    \left\| \mathcal{T} e^{-\int_0^T A(s) \, \mathrm{d}s} - \frac{\Delta r}{\sqrt{2\pi}} \sum_{j=0}^{M-1} \replythr{\hat{\kappa}}(r_j; \gamma, \delta) \mathcal{U}(T,r_j) \right\|
    \leq \frac{2\pi+1}{2\pi} e^{\delta - \delta^2 \gamma^2} + \frac{64}{15}e^{\frac{3\delta}{2} + \frac{1}{2}\|L\|_{L^1} - \frac{\pi}{\Delta r}}
    \leq \varepsilon_{\mathrm{lchs}} + \varepsilon_{\mathrm{quad}},
\end{equation}
where \replytwo{$r_j = -\widetilde{R} + j\Delta r$ ($\widetilde{R} = R - \Delta r/2$) are the midpoint nodes; following a similar analysis in~\cite{trefethen2014_the}, the bound above remains valid by translation invariance,} $\| A \|_{L^p} := \left( \int_0^T \| A(s) \|^p \, \mathrm{d}s \right)^{1/p}$, $\varepsilon_{\mathrm{quad}} = \mathcal{O}\left( e^{\frac{1}{2}\|L\|_{L^1} - \pi/\Delta r} \right)$ and $M = 2R/\Delta r = \mathcal{O} \left( \log \frac{1}{\varepsilon_{\mathrm{lchs}}} \left( \|L\|_{L^1} + \log \frac{1}{\varepsilon_{\mathrm{quad}}} \right) \right)$.
Additionally, the normalization factor
\begin{equation}
    \alpha_{\replythr{\hat{\kappa}}, R, \Delta r} := \frac{\Delta r}{\sqrt{2\pi}} \sum_{j=0}^{M-1} \left| \replythr{\hat{\kappa}}(r_j; \gamma, \delta) \right|
\end{equation}
satisfies
\begin{equation}
    \left| \alpha_{\replythr{\hat{\kappa}}, R, \Delta r} - \alpha_{\replythr{\hat{\kappa}}, \infty} \right| \leq \frac{\varepsilon_{\mathrm{lchs}}}{1 + 2\pi} + \frac{\varepsilon_{\mathrm{quad}}}{e^{(\|L\|_{L^1} + \delta)/2}}.
\end{equation}

\begin{remark}
    For the ODE system (Eq.~\eqref{eq:ode}), if $\mathbf{f}(t)\equiv0$, only the homogeneous term needs to be implemented.
    If $\mathbf{f}(t) \neq 0$, we perform a transformation $\mathbf{v}(t)=\mathbf{u}(t)-\mathbf{u}_0$, yielding the ODE system:
    \begin{equation}\label{eq:transformed ode}
        \frac{\mathrm{d} \mathbf{v}(t)}{\mathrm{d} t} = -A(t)\mathbf{v}(t) + \mathbf{f}(t) - \replythr{A(t)} \mathbf{u}_0, \quad \mathbf{v}(0) = 0,
    \end{equation}
    meaning only the inhomogeneous term needs to be implemented.
\end{remark}

\subsection{Implementation of the homogeneous term}\label{subsec:implementation of homogeneous term}
According to Eq.~\eqref{eq:integrate}, the homogeneous term of Eq.~\eqref{eq:ode} can be approximated as
\begin{equation}\label{eq:u1(t)}
    \mathbf{u}_1(T) = \mathcal{T} e^{-\int_0^T A(s) \, \mathrm{d}s} \mathbf{u}_0
    \approx \frac{1}{\sqrt{2\pi}} \sum_{j=0}^{M-1} w_j \replythr{\hat{\kappa}}(r_j; \gamma , \delta) \mathcal{U}(T,r_j) \mathbf{u}_0.
\end{equation}
In the implementation, we utilize $n = \lceil \log_2 N \rceil$ qubits to encode the system state $\mathbf{u}_1(t)$, with the initial state normalized as $|\mathbf{u}(0)\rangle = \frac{\mathbf{u}(0)}{\|\mathbf{u}(0)\|_2}$,
and $m = \lceil \log_2 M \rceil$ ancilla qubits to encode the LCHS coefficients $\mathbf{c} = [c_0, \dots, c_{M-1}]$, where $c_j = \frac{w_j\replythr{\hat{\kappa}}(r_j;\gamma,\delta)}{\sqrt{2\pi}}$.
Additionally, the following preparation oracles are required:
\begin{equation}
    O_{\text{prep}} \ket{0}^{\otimes n} = \ket{\mathbf{u}(0)}, \quad
    O_{\text{coef},r} \ket{0}^{\otimes m} =  \sum_{j=0}^{M-1} \frac{\sqrt{c_j}}{\sqrt{\| \mathbf{c} \|_1}} |j\rangle, \quad
    O_{\text{coef},l} \ket{0}^{\otimes m} = \sum_{j=0}^{M-1} \frac{\overline{\sqrt{c_j}}}{\sqrt{\| \mathbf{c} \|_1}} |j\rangle.
\end{equation}
\replytwo{
Here, for any complex number $z = \rho e^{i\theta}$ with principal argument $\theta \in (-\pi, \pi]$, we define the principal branch of the complex square root as $\sqrt{z} = \sqrt{\rho}\,e^{i\theta/2}$, and its complex conjugate as $\overline{z} = \rho e^{-i\theta}$.
}
\replythr{The realization cost of these oracles depends on the structure of the underlying data. When the sampled distribution is sufficiently smooth that its prefix sums are efficiently computable, these oracles can be realized via the hierarchical construction~\cite{soklakov2006_efficient} with precision $\varepsilon$ at a cost of $\mathrm{poly}(n,\log(1/\varepsilon))$ or $\mathrm{poly}(m,\log(1/\varepsilon))$. This applies unconditionally to the LCHS coefficient oracles, whose kernel is analytic on an equidistant grid with a phase linear in the node index. In contrast, for unstructured (e.g., random) initial data $\mathbf{u}_0$, state preparation generally requires $\Theta(2^{n})$ gates.}

Based on these oracles, Eq.~\eqref{eq:u1(t)} can be implemented as follows:
\begin{equation}\label{eq:inner LCHS}
\begin{aligned}
    \ket{0}^{\otimes m} \ket{0}^{\otimes n} &\xrightarrow{O_{\mathrm{coef,r}} \otimes O_{\mathrm{prep}}} \left( \sum_{j=0}^{M-1} \frac{\sqrt{c_j}}{\sqrt{\|\mathbf{c}\|_1}} |j\rangle\right) \otimes|\mathbf{u}(0)\rangle \\
    & \xrightarrow{\mathrm{SEL}(T) := \sum_{j=0}^{M-1}|j\rangle\langle j|\otimes \mathcal{U}(T,r_j)}  \sum_{j=0}^{M-1} \frac{\sqrt{c_j}}{\sqrt{\|\mathbf{c}\|_1}} |j\rangle \otimes \mathcal{U}(T,r_j)|\mathbf{u}(0)\rangle \\
    & \xrightarrow{O_{\mathrm{coef,l}}^\dagger \otimes I^{\otimes n}} \left(\frac{1}{\|\mathbf{c}\|_1}\ket{0}^{\otimes m}\right) \otimes \left(\sum_{j=0}^{M-1} c_j \mathcal{U}(T,r_j) |\mathbf{u}(0) \rangle \right) + |\perp\rangle.
\end{aligned}
\end{equation}
The corresponding quantum circuit for Eq.~\eqref{eq:inner LCHS} is illustrated in Fig.~\ref{fig:homo_Select}, where $sq_i$ denotes the $i$-th qubit used for encoding the system state, $aq_j$ denotes the $j$-th ancilla qubit for encoding the LCHS coefficients, \replythr{and the $\oslash$ notation indicates control on all basis states of the ancilla register (i.e., this symbol represents the $\ket{j}\bra{j}$ factor of the select oracle $\mathrm{SEL}(T)$ defined in Eq.~\eqref{eq:inner LCHS}).}
\replytwo{Throughout this paper, within each register the qubit with subscript $1$ carries the least significant bit and is placed at the top of the register in all circuit diagrams
; accordingly, basis states are written as $\ket{j} = \ket{j_m} \otimes \cdots \otimes \ket{j_1}$ with $j = \sum_{k=1}^{m} j_k 2^{k-1}$.}

For sufficiently small $\varepsilon_{\mathrm{lchs}}$ and $\varepsilon_{\mathrm{quad}}$ and sufficiently large $M$, the approximations
\begin{equation}
    \|\mathbf{c}\|_1^{-2} = \alpha_{\replythr{\hat{\kappa}}, R, \Delta r} ^{-2} \approx \alpha_{\replythr{\hat{\kappa}},R}^{-2} \geq e^{-2\delta}, \quad
    \left\| \sum_{j=0}^{M-1} c_j \mathcal{U}(T,r_j) |\mathbf{u}(0) \rangle \right\|_2^2 \approx \frac{\left\| \mathbf{u}_1(T) \right\|_2^2}{\left\| \mathbf{u}(0) \right\|_2^2}
\end{equation}
\replythr{hold}, from which the success probability of the measurement of Eq.~\eqref{eq:inner LCHS} can be estimated as
\begin{equation}\label{eq:probability of homo}
    P_1
    = \frac{1}{\|\mathbf{c}\|_1^2} \left\| \sum_{j=0}^{M-1} c_j \mathcal{U}(T,r_j) |\mathbf{u}(0) \rangle \right\|_2^2
    \gtrsim \frac{\left\| \mathbf{u}_1(T) \right\|_2^2}{e^{2\delta}  \left\| \mathbf{u}(0) \right\|_2^2}.
\end{equation}

\begin{figure}[htbp]
\centering
\begin{minipage}[b]{0.32\textwidth}
    \centering
    \includegraphics[width=0.75\textwidth]{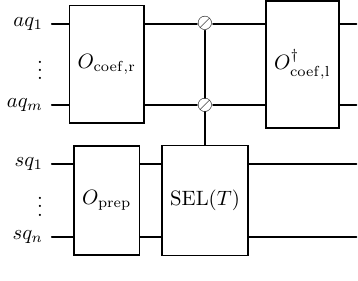}
    \caption{Quantum circuit for implementing the homogeneous term.}
    \label{fig:homo_Select}
\end{minipage}
\hfill
\begin{minipage}[b]{0.32\textwidth}
    \centering
    \includegraphics[width=0.75\textwidth]{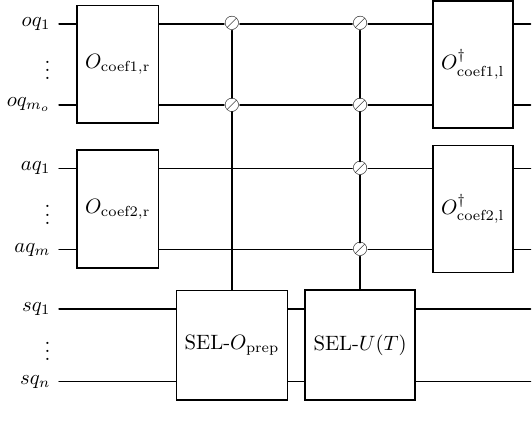}
    \caption{Quantum circuit for implementing the inhomogeneous term.}
    \label{fig:inhomo_Select}
\end{minipage}
\hfill
\begin{minipage}[b]{0.32\textwidth}
    \centering
    \includegraphics[width=\textwidth]{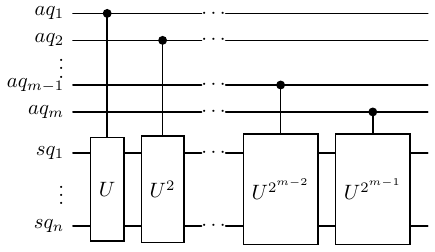}
    \caption{Quantum circuit for the select oracle in Lemma~\ref{lem:select oracle}.}
    \label{fig:circuit of Select}
\end{minipage}
\end{figure}

\subsection{Implementation of the inhomogeneous term}
By applying a suitable quadrature rule, the inhomogeneous term of Eq.~\eqref{eq:ode} can be approximated as
\begin{equation}\label{eq:inhomo_of_ode}
    \mathbf{u}_2(T)
    = \int_0^T \mathcal{T} e^{-\int_s^T A(s') \, \mathrm{d}s'} \mathbf{f}(s) \, \mathrm{d}s \\
    \approx \sum_{k = 0}^{M_o - 1} w_k \|\mathbf{f}(T_k)\| \mathcal{T} e^{-\int_{T_k}^T A(s') \, \mathrm{d}s'} \ket{\mathbf{f}(T_k)},
\end{equation}
where $M_o$ denotes the number of outer quadrature nodes, $T_k$ and $w_k$ are the preassigned quadrature nodes and weights on the interval $[0,T]$, respectively.
We define the outer LCU coefficients $\mathbf{d} = [d_0,\cdots,d_{M_o-1}]$,
where $d_k = w_
k\|\mathbf{f}(T_k)\|$ and $\ket{\mathbf{v}_k} = \ket{\mathbf{f}(T_k)}$, then by following the LCHS implementation (Eq.~\eqref{eq:inner LCHS}) for $\mathcal{T} e^{-\int_0^T A(s) \,\mathrm{d} s}$ and defining $U_{k,j}(T) := \mathcal{T} e^{-i \int_{T_k}^T (H(s) + r_jL(s)) \, \mathrm{d}s}$, we can further approximate $\mathbf{u}_2(T)$ as
\begin{equation}\label{eq:u2(t)_1}
    \mathbf{u}_2(T) \approx \sum_{k = 0}^{M_o - 1} d_k \mathcal{T} e^{-\int_{T_k}^T A(s') \, \mathrm{d}s'} \ket{\mathbf{v}_k}
    \approx \sum_{k=0}^{M_o-1} d_k \left( \sum_{j=0}^{M-1} c_j U_{k,j}(T) \right) |\mathbf{v}_k\rangle.
\end{equation}

We employ the same $n$ system qubits and $m$ ancilla qubits for the inner LCHS coefficients, together with $m_o = \lceil \log_2 M_o \rceil$ ancilla qubits for the outer LCU coefficients.
Similarly to the homogeneous case, assume the following preparation oracles can be efficiently constructed:
\begin{equation}
\begin{aligned}
    O_{\text{prep}}(k) \ket{0}^{\otimes n} = \ket{\mathbf{v}_k}, \quad
    O_{\text{coef1},r} \ket{0}^{\otimes m_o} &= \sum_{k=0}^{M_o-1} \frac{\sqrt{d_k}}{\sqrt{\| \mathbf{d} \|_1}} |k\rangle, \quad
    O_{\text{coef1},l} \ket{0}^{\otimes m_o} = \sum_{k=0}^{M_o-1} \frac{\overline{\sqrt{d_k}}}{\sqrt{\| \mathbf{d} \|_1}} |k\rangle, \\
    O_{\text{coef2},r} \ket{0}^{\otimes m} &= \sum_{j=0}^{M-1} \frac{\sqrt{c_j}}{\sqrt{\| \mathbf{c} \|_1}} |j\rangle, \quad
    O_{\text{coef2},l} \ket{0}^{\otimes m} = \sum_{j=0}^{M-1} \frac{\overline{\sqrt{c_j}}}{\sqrt{\| \mathbf{c} \|_1}} |j\rangle.
\end{aligned}
\end{equation}
With these oracles, Eq.~\eqref{eq:u2(t)_1} can be implemented as follows:
\begin{equation}\label{eq:outer LCHS}
\begin{aligned}
    &\ket{0}^{\otimes m_o} \ket{0}^{\otimes m} \ket{0}^{\otimes n}
    \xrightarrow{O_{\text{coef1},r} \otimes O_{\text{coef2},r} \otimes I^{\otimes n}} \left( \sum_{k=0}^{M_o-1} \frac{\sqrt{d_k}}{\sqrt{\| \mathbf{d} \|_1}} |k\rangle \right) \otimes \left( \sum_{j=0}^{M-1} \frac{\sqrt{c_j}}{\sqrt{\| \mathbf{c} \|_1}} |j\rangle \right) \otimes \ket{0}^{\otimes n} \\
    & \xrightarrow{\text{SEL-} O_{\text{prep}} := \sum_{k=0}^{M_o-1} |k\rangle\langle k| \otimes I^{\otimes m} \otimes O_{\text{prep}}(k)} \sum_{k=0}^{M_o-1} \frac{\sqrt{d_k}}{\sqrt{\| \mathbf{d} \|_1}} |k\rangle \otimes \left( \sum_{j=0}^{M-1} \frac{\sqrt{c_j}}{\sqrt{\| \mathbf{c} \|_1}} |j\rangle \right) \otimes |\mathbf{v}_k\rangle \\
    & \xrightarrow{\replytwo{\text{SEL-} U(T)} := \sum_{k=0}^{M_o-1}\sum_{j=0}^{M-1} |k\rangle\langle k| \otimes |j\rangle\langle j| \otimes U_{k,j}(T)} \sum_{k=0}^{M_o-1} \sum_{j=0}^{M-1} \frac{\sqrt{d_k}}{\sqrt{\| \mathbf{d} \|_1}} |k\rangle \otimes \frac{\sqrt{c_j}}{\sqrt{\| \mathbf{c} \|_1}} |j\rangle \otimes U_{k,j}(T) |\mathbf{v}_k\rangle \\
    & \xrightarrow{O_{\text{coef1},l}^{\dagger} \otimes O_{\text{coef2},l}^{\dagger} \otimes I^{\otimes n}}
    \frac{1}{\| \mathbf{d} \|_1 \| \mathbf{c} \|_1} \ket{0}^{\otimes m_o} \ket{0}^{\otimes m} \otimes \sum_{k=0}^{M_o-1} d_k \left( \sum_{j=0}^{M-1} c_j U_{k,j}(T) \right) |\mathbf{v}_k\rangle + |\perp\rangle.
\end{aligned}
\end{equation}

The corresponding quantum circuit for Eq.~\eqref{eq:outer LCHS} is illustrated in Fig.~\ref{fig:inhomo_Select}, where $oq_k$ denotes the $k$-th ancilla qubit used for encoding the outer LCU coefficients.

In this work, we adopt the trapezoidal rule
\begin{equation}\label{eq:outer trapezoidal}
    T_k = T - \frac{\Delta t}{2} - k \Delta t, \quad k=0,1,\dots,M_o-1, \quad \Delta t = \frac{T}{M_o}
\end{equation}
for the outer quadrature to simplify the subsequent quantum circuit implementation.
We first derive an upper bound for $M_o$ for the trapezoidal rule, which is a straightforward result of the standard quadrature error bound \cite{süli2003_an} and the chain rule.

\begin{lemma}\label{lem:outer quadrature}
Suppose $A(t) \in C^1([0,T]; \mathbb{C}^{N \times N})$, and assume the time-ordered exponential $\mathcal{T} e^{-\int_s^T A(s') \, \mathrm{d}s'}$ is bounded by $\left\| \mathcal{T} e^{-\int_s^T A(s') \, \mathrm{d}s'} \right\|
    \leq e^{- \int_0^T \widetilde{\alpha}(s') \, \mathrm{d}s'}
    =: C_A(T)$
with $\widetilde{\alpha}(s) = \min_{1\leq j \leq N} \{ 0, \mathrm{Re}(\lambda_j(A(s))) \}$.
Assume $\mathbf{f}(t) \in C^2([0,T]; \mathbb{R}^N)$, then applying the trapezoidal rule (Eq.~\eqref{eq:outer trapezoidal}) yields
\begin{equation}\label{eq:midpoint_error_td}
    \left\| \int_0^T \mathcal{T} e^{-\int_s^T A(s') \,\mathrm{d}s'} \mathbf{f}(s) \,\mathrm{d}s - \Delta t \sum_{k=0}^{M_o-1} \mathcal{T} e^{-\int_{T_k}^T A(s') \, \mathrm{d}s'} \mathbf{f}(T_k) \right\|
    \leq \frac{T^3 C_A(T)}{24 M_o^2} B_{A,\mathbf{f}}(T),
\end{equation}
where $B_{A,\mathbf{f}}(T) := \sup_{t \in [0,T]} \left( (\|A'(t)\| + \|A(t)\|^2) \|\mathbf{f}(t)\| + 2\|A(t)\| \|\mathbf{f}'(t)\| + \|\mathbf{f}''(t)\| \right)$.
In order to bound the quadrature error by $\varepsilon_{\mathrm{quad,o}} > 0$, it suffices to choose
\begin{equation}\label{eq:midpoint_M_td}
    M_o = \mathcal{O} \left( T^{3/2} C_A(T)^{1/2} B_{A,\mathbf{f}}(T)^{1/2} \varepsilon_{\mathrm{quad,o}}^{-1/2} \right).
\end{equation}
\end{lemma}

For subsequent analysis, we assume $A$ is time-independent, i.e.,
\begin{equation}
    U_{k,j}(T)
    = e^{-i (H + r_j L)(T-T_k)}
    = e^{-i (H + r_j L)(k \Delta t + \frac{1}{2}\Delta t)}.
\end{equation}
With $U_j(T) := e^{-i (H + r_jL)T}$ and $\mathrm{SEL}(T) = \sum_{j=0}^{M-1}|j\rangle\langle j|\otimes U_j(T)$, the select oracle $\replytwo{\text{SEL-} U(T)}$ can be rewritten as:
\begin{equation}\label{eq:sel-U_kj}
    \replytwo{\text{SEL-} U(T)
    = \left( \sum_{k=0}^{M_o-1} |k\rangle\langle k| \otimes \left( \mathrm{SEL}(\Delta t) \right)^{k} \right)
    \cdot I^{\otimes m_o} \otimes \mathrm{SEL}\left( \frac{\Delta t}{2} \right).}
\end{equation}

\begin{lemma}[\cite{Lin2022_Lecture}]\label{lem:select oracle}
Let $n$ and $m$ denote the numbers of first and second registers, respectively. For any $2^n\times2^n$ unitary operator $U$, the select oracle $\sum_{j=0}^{2^m - 1}\ket{j}\bra{j} \otimes U^j$ can be implemented as $\prod_{j=1}^m \left[ \mathrm{C}U^{2^{j-1}} \right]_{[1,n]}^{j}$, where $\left[ \mathrm{C}U^{2^{j-1}} \right]_{[1,n]}^{j}$ denotes the multi-controlled gate $U^{2^{j-1}}$ acting on the first register and controlled by the $j$-th qubit of the second register.
\end{lemma}

\begin{remark}
    We denote $\prod_{j=1}^m A_j
    = A_m \cdots A_2 A_1$ and $\prod_{j=m}^1 A_j
    = A_1 A_2 \cdots A_m$ as defined in \cite{Childs2021_Theory}.
\end{remark}

Since $\mathrm{SEL}(\Delta t)$ is a $2^{n+m} \times 2^{n+m}$ unitary gate, Lemma~\ref{lem:select oracle} allows us to further rewrite Eq.~\eqref{eq:sel-U_kj} as
\begin{equation}\label{eq:implementation of sel-U_kj}
    \replytwo{\text{SEL-} U(T)
    = \left( \prod_{k=1}^{m_o} \left[ \text{C-} \mathrm{SEL}(2^{k-1} \Delta t) \right]_{[1,n+m]}^k \right)
    \cdot I^{\otimes m_o} \otimes \mathrm{SEL}\left( \frac{\Delta t}{2} \right).}
\end{equation}

\replytwo{Throughout this paper, we write
\begin{equation}\label{eq:norm-T}
    \|\mathbf{f}\|_T := \sup_{t \in [0,T]} \|\mathbf{f}(t)\|,
\end{equation}
i.e., the standard $L^{\infty}([0,T])$-norm of the continuous function $t \mapsto \|\mathbf{f}(t)\|$.}
Combining Eq.~\eqref{eq:quadrature error}, we thus derive an
estimation for the implementation of the inhomogeneous term.

\begin{lemma}
Assume $A$ is time-independent. Applying the trapezoidal rule for outer quadrature, the following error bound holds:
\begin{equation}\label{eq:quadrature error of inhomo term}
\begin{aligned}
    &\left\| \int_0^{T} e^{-A(T-s)} \mathbf{f}(s) \, \mathrm{d}s - \sum_{k=0}^{M_o-1} \Delta t \left( \sum_{j=0}^{M-1} c_j U_{k,j}(T) \right) \mathbf{f}(T_k) \right\| \\
    \leq& \frac{T^3 e^{-\widetilde{\alpha} T}}{24 M_o^2} \cdot B_{A,\mathbf{f}}(T)
    + \Delta t \sum_{k=0}^{M_o-1} \left( \varepsilon_{\mathrm{lchs}}
    + \frac{64}{15} e^{3\delta/2 - \pi/\Delta r + \|L\| T_k/2} \right) \|\mathbf{f}\|_T \\
    \leq& \varepsilon_{\mathrm{quad,o}}
    + \varepsilon_{\mathrm{lchs}} T \|\mathbf{f}\|_T
    + \frac{64T\|\mathbf{f}\|_T}{15M_o} \frac{e^{\|L\|T/2}-1}{e^{\|L\|T/2M_o}-1} e^{3\delta/2 +\|L\|T/4M_o - \pi/\Delta r}.
\end{aligned}
\end{equation}
where $\widetilde{\alpha}$ is defined in Lemma~\ref{lem:outer quadrature}.

To bound the total error in Eq.~\eqref{eq:quadrature error of inhomo term} by $\varepsilon > 0$, it suffices to choose
\begin{equation}
    R = \mathcal{O} \left(  \log \frac{1}{\varepsilon} \replytwo{+} \log \frac{T \|\mathbf{f}\|_T}{\varepsilon} \right), \quad
    M = \mathcal{O} \left( R \left( \log  \frac{T \|\mathbf{f}\|_T(e^{\|L\|T/2}-1)}{\varepsilon M_o (e^{\|L\|T/2M_o}-1)} \replytwo{+} \frac{\|L\|T}{4M_o} \right) \right).
\end{equation}
with $M_o = \mathcal{O} \left( T^{3/2} C_A(T)^{1/2} B_{A,\mathbf{f}}(T)^{1/2} \varepsilon^{-1/2} \right)$.
\end{lemma}

For sufficiently small $\varepsilon_{\mathrm{lchs}}, \varepsilon_{\mathrm{quad}}$ and $\varepsilon_{\mathrm{quad,o}}$ and sufficiently large $M$ and $M_o$, the approximations
\begin{equation}
    \| \mathbf{d} \|_1
    = \sum_{k=0}^{M_o-1} w_k \|\mathbf{f}(T_k)\|
    \leq M_o \cdot \Delta t \cdot \|\mathbf{f}\|_T
    = T \|\mathbf{f}\|_T, \quad
    \left\| \sum_{k=0}^{M_o-1} d_k \left( \sum_{j=0}^{M-1} c_j U_{k,j}(T) \right) |\mathbf{v}_k\rangle \right\|_2^2 \approx \left\| \mathbf{u}_2(T) \right\|_2^2
\end{equation}
\replythr{hold}. Then, combining $\|\mathbf{c}\|_1^{-2} \gtrsim e^{-2\delta}$, we can estimate the success probability of the measurement of Eq.~\eqref{eq:outer LCHS} as
\begin{equation}\label{eq:probability of inhomo}
    P_2
    = \frac{1}{\|\mathbf{d}\|_1^2 \cdot \|\mathbf{c}\|_1^2} \left\| \sum_{k=0}^{M_o-1} d_k \left( \sum_{j=0}^{M-1} c_j U_{k,j}(T) \right) |\mathbf{v}_k\rangle \right\|_2^2
    \gtrsim \frac{\left\| \mathbf{u}_2(T) \right\|_2^2}{e^{2\delta} T^2 \|\mathbf{f}\|_T^2}.
\end{equation}

\section{The advection-diffusion equation}\label{sec:the advection-diffusion equation}

In this section, we consider the anisotropic advection-diffusion equation in $d$-dimensional space:
\begin{equation}\label{eq:gov_eq}
    \left\{
    \begin{aligned}
    \frac{\partial u}{\partial t} + \sum_{p=1}^d a_p \frac{\partial u}{\partial x_p} - \sum_{p=1}^d b_p  \frac{\partial^2 u}{\partial x_p^2} &= -cu + f,  \\
    u(0,\mathbf{x}) &= u_0(\mathbf{x}),
    \end{aligned}
    \right.
    \quad \mathbf{x} = (x_1, \cdots, x_d) \in \Omega \subset \mathbb{R}^d,
\end{equation}
where $\mathbf{a} = (a_1, \cdots, a_d) \in \mathbb{R}^{d}$ denotes the advection coefficients; $\mathbf{b} = (b_1, \cdots, b_d) \in \mathbb{R}_{\ge 0}^d$ denotes the anisotropic diffusion coefficients, \replyone{i.e., the direction-dependent diffusivities along the principal axes of the underlying diffusion tensor (cf.~Remark~\ref{rem:aniso});} $c \in \mathbb{R}_{\ge 0}$ denotes the attenuation coefficient, whose non-negativity guarantees the feasibility of the LCHS method.

\replyone{
\begin{remark}\label{rem:aniso}
The term ``anisotropic'' refers to direction-dependent diffusivity along the principal axes of the diffusion tensor.
Any constant symmetric positive semidefinite tensor $D$ admits the spectral decomposition $D = Q\Lambda Q^{\top}$, and the orthogonal change of variables $\mathbf{y} = Q^{\top}\mathbf{x}$ reduces the full-tensor equation---including mixed-derivative (cross-diffusion) terms---to Eq.~\eqref{eq:gov_eq} with $b_p = \lambda_p(D) \ge 0$ and rotated advection $Q^{\top}\mathbf{a}$, while leaving $c$ and every norm entering our analysis invariant. Hence Eq.~\eqref{eq:gov_eq} already covers general constant-coefficient diffusion tensors. In the whole-space or periodic setting considered in this work, the rotated domain leaves the subsequent analysis unaffected.
\end{remark}}

\subsection{Spatial discretization via the finite volume method}\label{subsec:spatial discretization}

For the spatial discretization of Eq.~\eqref{eq:gov_eq}, we first rewrite it in the conservative form
\begin{equation}
    \frac{\partial u}{\partial t} + \sum_{p=1}^d \frac{\partial F_p}{\partial x_p} = -cu+f, \quad
    F_p:= a_p u - b_p \frac{\partial u}{\partial x_p}.
\end{equation}
Then, we adopt the FVM and define the $d$-dimensional control volume for cell $\Omega_{j_1,\cdots,j_d}$ as $[x_{j_1-1/2}, x_{j_1+1/2}]\times \cdots \times [x_{j_d-1/2}, x_{j_d+1/2}]$, where $h_p = x_{j_p+1/2} - x_{j_p-1/2}$ denotes the cell width along the $p$-th dimension.
Integrating the conservative equation over the control volume $\Omega_{j_1,\cdots,j_d}$ gives the semi-discrete equation
\begin{equation}
    \frac{\mathrm{d} u_{j_1,\cdots,j_d}(t)}{\mathrm{d} t} + \sum_{p=1}^d \frac{F_{j_p+1/2} - F_{j_p-1/2}}{\replytwo{h_p}} = -c u_{j_1,\cdots,j_d}(t) + f_{j_1,\cdots,j_d}(t),
\end{equation}
where $u_{j_1,\cdots,j_d}$ and $f_{j_1,\cdots,j_d}$ are the cell-averaged values of $u$ and $f$ over $\Omega_{j_1,\cdots,j_d}$, and $F_{j_p+1/2}$ represents the flux at the interface $x_{j_p+1/2}$ (between cells $\Omega_{j_1,\cdots,j_p,\cdots,j_d}$ and $\Omega_{j_1,\cdots, j_p+1,\cdots,j_d}$).

By adopting a suitable flux construction scheme and \replythr{omitting the source term}, we derive the ODE system for the $p$-th dimension as $\frac{\mathrm{d} \mathbf{u}^p(t)}{\mathrm{d} t} + A_p \mathbf{u}^p(t) = 0$, where \replytwo{$\mathbf{u}^p(t) \in \mathbb{R}^{N_p}$ ($N_p = 2^{n_p}$) collects the cell averages along the $p$-th spatial direction (with the remaining indices $j_q$, $q \neq p$, fixed)} and $A_p \in \mathbb{R}^{N_p \times N_p}$ is the coefficient matrix.
\replytwo{Collecting the cell averages into the vector $\mathbf{u}(t) := [u_{j_1,\cdots,j_d}(t)] \in \mathbb{R}^{N}$ with $N = \prod_{p=1}^{d} N_p$, and} assembling the contributions from all dimensions, we obtain the global semi-discrete system
\replytwo{
\begin{equation}\label{eq:d-dimensional equation}
    \frac{\mathrm{d} \mathbf{u}(t)}{\mathrm{d} t} + A \mathbf{u}(t) = \mathbf{f}(t) := \mathbf{f}_{\mathrm{phys}}(t) + \mathbf{f}_{\mathrm{bc}}(t),
\end{equation}
where the source vector $\mathbf{f}(t)$ consists of the cell-averaged physical source $\mathbf{f}_{\mathrm{phys}}(t) := [f_{j_1,\cdots,j_d}(t)] \in \mathbb{R}^{N}$ and the boundary-induced contribution $\mathbf{f}_{\mathrm{bc}}(t)$ arising from the boundary fluxes,} and the global coefficient matrix $A$ is expressed via tensor products as
\begin{equation}\label{eq:d-dimensional matrix}
    A = A_1 \otimes I^{\otimes n_2} \otimes \cdots\otimes I^{\otimes n_d} + \cdots + I^{\otimes n_1} \otimes \cdots \otimes I^{\otimes n_{d-1}} \otimes A_d  + c I^{\otimes n_1} \otimes I^{\otimes n_2} \otimes \cdots\otimes I^{\otimes n_d}.
\end{equation}

\subsection{Flux construction schemes}\label{subsec:flux construction}
We use the $1$-dimensional homogeneous advection-diffusion equation (no source term, i.e., $\mathbf{f}_{\mathrm{phys}} = \mathbf{0}$) to introduce various flux construction schemes:
\begin{equation}
    \frac{\partial u}{\partial t} + \frac{\partial F}{\partial x} = 0, \quad
    F = F_{\text{adv}} + F_{\text{diff}} = a u - b \frac{\partial u}{\partial x}.
\end{equation}

\subsubsection{Central scheme}
At the interface $x_{j+1/2}$, we employ the central scheme to compute the diffusive flux and advective flux, which gives
\begin{equation}
     F_{j+1/2}^{\text{c}}
     = F_{\text{adv}, j+1/2}^{\text{c}} + F_{\text{diff}, j+1/2}^{\text{c}}
     = a \frac{u_j + u_{j+1}}{2} - b \frac{u_{j+1}-u_j}{h},
\end{equation}
from which we derive
\begin{equation}
    \frac{F_{j+1/2}^{\text{c}} - F_{j-1/2}^{\text{c}}}{h}
    = \frac{a}{2h}(u_{j+1} - u_{j-1}) + \frac{b}{h^2}(-u_{j-1} + 2u_j -u_{j+1}).
\end{equation}

\subsubsection{Exponential scheme}

We further adopt the exponential scheme \cite{allen1955_relaxation, spalding1972_a}, which is derived from the exact steady-state solution of the advection-diffusion equation (i.e., $\frac{\partial u}{\partial t} = 0$) and rigorously guarantees flux conservation.
A direct calculation gives
\begin{equation}\label{eq:F^exp}
\begin{aligned}
    F_{j+1/2}^{\text{exp}}
    &= F_{\text{adv}, j+1/2}^{\text{exp}} + F_{\text{diff}, j+1/2}^{\text{exp}}
    = a u_{j+1/2} -b \left( \frac{\partial u}{\partial x} \right)_{j+1/2} \\
    &= \frac{a\left( u_j e^{\lambda h /2} + u_{j+1} \right)}{e^{\lambda h /2} + 1} + \frac{a (u_j - u_{j+1}) e^{\lambda h/2}}{e^{\lambda h} - 1}
    = \frac{a (u_j e^{\lambda h} - u_{j+1})}{e^{\lambda h} - 1},
\end{aligned}
\end{equation}
which further gives
\begin{equation}
    \frac{F_{j+1/2}^{\text{exp}} - F_{j-1/2}^{\text{exp}}}{h}
    = \frac{a}{h}(u_j - u_{j-1}) + \frac{a}{h(e^{\lambda h} - 1)}(-u_{j-1} + 2u_j -u_{j+1}),
\end{equation}
where $\lambda = a/b$ is a parameter associated with the Peclet number.

It is straightforward to verify that
\begin{equation}
\begin{aligned}
    \lim_{a\rightarrow 0} F_{j+1/2}^{\text{exp}}
    &= -b \frac{u_{j+1}-u_j}{h}
    = F_{\text{diff}, j+1/2}^{\text{c}}, \\
    \lim_{b\rightarrow 0} F_{j+1/2}^{\text{exp}}
    &= \frac{a+|a|}{2} u_j + \frac{a-|a|}{2} u_{j+1} = F_{\text{adv}, j+1/2}^{\text{upwind}},
\end{aligned}
\end{equation}
which indicates that the exponential scheme reduces to the central diffusive flux as $a \to 0$ and to the upwind advective flux as $b \to 0$.

\subsection{Coefficient matrices under different boundary conditions}\label{subsec:matrix bc}

We mainly focus on two commonly used boundary conditions in numerical discretization.
The first type is the Robin boundary conditions (BCs), given by
\begin{equation}\label{eq:robin_bc}
    \alpha_L u(x_L) + \beta_L \left.\frac{\partial u}{\partial x}\right|_{x_L}=g_L, \quad \alpha_R u(x_R) + \beta_R \left.\frac{\partial u}{\partial x}\right|_{x_R}=g_R,
\end{equation}
where $\alpha_L^2+ \beta_L^2 \neq 0$ and $\alpha_R^2 + \beta_R^2 \neq 0$ to guarantee the physical validity of the formulation.
Obviously, we can set $\beta_L = \beta_R = 0$ to obtain the Dirichlet BCs and set $\alpha_L = \alpha_R = 0$ to obtain the Neumann BCs.
Thus, we only consider the Robin BCs in the subsequent derivation, as both the Dirichlet and Neumann BCs are special cases of the Robin BCs.

The second type is the periodic BCs, expressed as
\begin{equation}\label{eq:periodic_bc}
    u(x_L) = u(x_R), \quad
    \left.\frac{\partial u}{\partial x}\right|_{x_L} = \left.\frac{\partial u}{\partial x}\right|_{x_R},
\end{equation}
which impose the continuity of both the solution and its derivative across the domain boundary.

\subsubsection{Robin boundary conditions}

We adopt the ghost cell method to derive the boundary flux under the Robin BCs, introducing a ghost cell $[x_{-1/2}, x_{1/2}]$ at the left boundary and another $[x_{N+1/2}, x_{N+3/2}]$ at the right boundary. For simplicity, we denote $x_L = x_{1/2}$ and $x_R = x_{N+1/2}$.

For the central scheme, we substitute $u_L^{\text{c}} = \frac{u_0^{\text{c}}+u_1}{2}$ and $\left( \frac{\partial u}{\partial x} \right)_{x_L}^{\text{c}} = \frac{u_1-u_0^{\text{c}}}{h}$ into Eq.~\eqref{eq:robin_bc}, yielding
\begin{equation}\label{eq:left boundary of robin(central)}
    u_0^{\text{c}} = \frac{(2\beta_L + h\alpha_L) u_1 - 2hg_L}{2\beta_L - h\alpha_L}, \quad
    u_L^{\text{c}} = \frac{2\beta_L u_1 - hg_L}{2\beta_L - h\alpha_L}, \quad
    F_L^{\text{c}} = \frac{2(a\beta_L +b\alpha_L)u_1 - (ah+2b)g_L}{2\beta_L - h\alpha_L}.
\end{equation}
Similarly, substituting $u_R^{\text{c}} = \frac{u_N+u_{N+1}^{\text{c}}}{2}$ and $\left( \frac{\partial u}{\partial x} \right)_{x_R}^{\text{c}} = \frac{u_{N+1}^{\text{c}}-u_N}{h}$ into Eq.~\eqref{eq:robin_bc}, we obtain
\begin{equation}\label{eq:right boundary of robin(central)}
    u_{N+1}^{\text{c}} = \frac{(2\beta_R-h\alpha_R) u_N + 2hg_R}{2\beta_R + h\alpha_R}, \quad
    u_R^{\text{c}} = \frac{2\beta_R u_N + hg_R}{2\beta_R + h\alpha_R}, \quad
    F_R^{\text{c}} = \frac{2(a\beta_R +b\alpha_R)u_N + (ah-2b)g_R}{2\beta_R + h\alpha_R}.
\end{equation}
Thus, the corresponding $N \times N$ ODE system is obtained with the coefficient matrix $A$ and the boundary-induced source vector $\mathbf{f}_{\mathrm{bc}}$ given by
\begin{equation}\label{eq:robin_central_matrix}
    A =
    \begin{bmatrix}
        \frac{2b}{h^2} - \frac{(2b + ah)(2\beta_L + h\alpha_L)}{2h^2(2\beta_L - h\alpha_L)} & \frac{a}{2h} - \frac{b}{h^2} & & & \\
        -\frac{a}{2h} - \frac{b}{h^2} & \frac{2b}{h^2} & \frac{a}{2h} - \frac{b}{h^2} & & \\
        & \ddots & \ddots & \ddots & \\
        & & -\frac{a}{2h} - \frac{b}{h^2} & \frac{2b}{h^2} & \frac{a}{2h} - \frac{b}{h^2} \\
        & & & -\frac{a}{2h} - \frac{b}{h^2} & \frac{2b}{h^2} - \frac{(2b - ah)(2\beta_R - h\alpha_R)}{2h^2(2\beta_R + h\alpha_R)} \\
    \end{bmatrix}, \quad
    \mathbf{f}_{\mathrm{bc}} =
    \begin{bmatrix}
        \frac{ah+2b}{h(2\beta_L - h\alpha_L)}g_L \\
        0 \\
        \vdots \\
        0 \\
        \frac{ah-2b}{h(2\beta_R + h\alpha_R)}g_R
    \end{bmatrix}.
\end{equation}

\begin{remark}
    For the central scheme, to avoid singularities in the boundary expressions given by Eqs.~\eqref{eq:left boundary of robin(central)} and \eqref{eq:right boundary of robin(central)}, the step size $h$ must be selected to satisfy $2\beta_L - h\alpha_L \neq 0$ and $ 2\beta_R + h\alpha_R \neq 0$.
\end{remark}

For the exponential scheme, we substitute $u_L^{\text{exp}} = \frac{u_0 e^{\lambda h/2} + u_1}{e^{\lambda h/2} + 1}$ and $\left( \frac{\partial u}{\partial x} \right)_{x_L}^{\text{exp}} = \frac{\lambda (u_1 - u_0) e^{\lambda h /2}}{e^{\lambda h} - 1}$ into Eq.~\eqref{eq:robin_bc}, leading to
\begin{equation}\label{eq:left boundary of robin(exp)}
\begin{aligned}
    u_L^{\text{exp}} &= \frac{u_1 \beta_L \lambda - g_L (e^{\lambda h/2}-1)}{\beta_L \lambda - \alpha_L (e^{\lambda h/2}-1)}, \quad
    u_0^{\text{exp}} = \frac{(\beta_L \lambda + \alpha_L (1-e^{-\lambda h/2})) u_1  - g_L (e^{\lambda h}-1)e^{-\lambda h/2}}{\beta_L \lambda - \alpha_L (e^{\lambda h/2}-1)}, \\
    F_L^{\text{exp}} &= \frac{a(\beta_L \lambda + \alpha_L) u_1 -  a g_L e^{\lambda h/2}}{\beta_L\lambda - \alpha_L(e^{\lambda h/2}-1)}.
\end{aligned}
\end{equation}
Substituting $u_R^{\text{exp}} = \frac{u_N e^{\lambda h/2} + u_{N+1}}{e^{\lambda h/2} + 1}$ and $\left( \frac{\partial u}{\partial x} \right)_{x_R}^{\text{exp}} = \frac{\lambda (u_{N+1} - u_N) e^{\lambda h /2}}{e^{\lambda h} - 1}$ into Eq.~\eqref{eq:robin_bc}, we get
\begin{equation}\label{eq:right boundary of robin(exp)}
\begin{aligned}
    u_{N+1}^{\text{exp}} &= \frac{(\beta_R \lambda - \alpha_R (e^{\lambda h/2}-1)) u_N  + g_R (e^{\lambda h}-1)e^{-\lambda h/2}}{\beta_R \lambda + \alpha_R (1-e^{-\lambda h/2})}, \quad
    u_R^{\text{exp}} = \frac{u_N \beta_R \lambda + g_R (1-e^{-\lambda h/2})}{\beta_R \lambda + \alpha_R (1-e^{-\lambda h/2})}, \\
    F_R^{\text{exp}} &= \frac{a(\beta_R \lambda + \alpha_R) u_N - a g_R e^{-\lambda h/2}}{\beta_R\lambda + \alpha_R (1-e^{-\lambda h/2})}.
\end{aligned}
\end{equation}
Thus, the $N \times N$ ODE system for the exponential scheme is obtained with the coefficient matrix $A$ and the boundary-induced source vector $\mathbf{f}_{\mathrm{bc}}$ given by
\begin{equation}\label{eq:robin_exp_matrix}
\begin{aligned}
    A =&
    \begin{bmatrix}
        \frac{a}{h} + \frac{2a}{h(e^{\lambda h} - 1)} &  - \frac{a}{h(e^{\lambda h} - 1)} & & &  \\
        -\frac{a}{h} - \frac{a}{h(e^{\lambda h} - 1)} & \frac{a}{h}+\frac{2a}{h(e^{\lambda h} - 1)} & -\frac{a}{h(e^{\lambda h} - 1)} & &  \\
        & \ddots & \ddots & \ddots & \\
        & & -\frac{a}{h} - \frac{a}{h(e^{\lambda h} - 1)} & \frac{a}{h} + \frac{2a}{h(e^{\lambda h} - 1)} & -\frac{a}{h(e^{\lambda h} - 1)} \\
        & & & -\frac{a}{h} - \frac{a}{h(e^{\lambda h} - 1)} & \frac{a}{h} + \frac{2a}{h(e^{\lambda h} - 1)} \\
    \end{bmatrix} \\
    &-
    \begin{bmatrix}
         \frac{a \beta_L \lambda e^{\lambda h} + a \alpha_L (e^{\lambda h}- e^{\lambda h/2})}{h (e^{\lambda h}-1) \left[ \beta_L \lambda - \alpha_L (e^{\lambda h/2}-1) \right]} & & \\
         & \ddots &  \\
         & & \frac{a \beta_R \lambda - a \alpha_R (e^{\lambda h/2}-1)}{h (e^{\lambda h}-1) \left[ \beta_R \lambda + \alpha_R(1-e^{-\lambda h/2}) \right]}
    \end{bmatrix}, \quad
    \mathbf{f}_{\mathrm{bc}} =
    \begin{bmatrix}
        \frac{a e^{\lambda h/2}}{h\left[ \beta_L \lambda - \alpha_L(e^{\lambda h/2}-1) \right]}g_L \\
        \vdots \\
        - \frac{a e^{-\lambda h/2}}{h\left[ \beta_R \lambda + \alpha_R(1-e^{-\lambda h/2}) \right]}g_R
    \end{bmatrix}.
\end{aligned}
\end{equation}

\begin{remark}
    For the exponential scheme, to avoid singularities in the boundary expressions given by Eqs.~\eqref{eq:left boundary of robin(exp)} and~\eqref{eq:right boundary of robin(exp)}, the step size $h$ must be selected to satisfy $\beta_L \lambda - \alpha_L(e^{\lambda h/2}-1) \neq 0$ and $\beta_R \lambda + \alpha_R(1 - e^{-\lambda h/2}) \neq 0$.
\end{remark}

\subsubsection{Periodic boundary conditions}

For the periodic BCs, Eq.~\eqref{eq:periodic_bc} directly implies $u_0 = u_N$ and $u_{N+1} = u_1$. This yields $F_L = F_R$ for all flux construction schemes.

For the central scheme, the boundary fluxes are derived as
\begin{equation}
    F_L^{\text{c}} = F_R^{\text{c}}
     = a \frac{u_1 + u_N}{2} - b \frac{u_1-u_N}{h},
\end{equation}
and the corresponding $N \times N$ ODE system is given by
\begin{equation}\label{eq:periodic_central_matrix}
    A =
    \begin{bmatrix}
        \frac{2b}{h^2} & \frac{a}{2h} - \frac{b}{h^2} & & & -\frac{a}{2h} - \frac{b}{h^2} \\
        -\frac{a}{2h} - \frac{b}{h^2} & \frac{2b}{h^2} & \frac{a}{2h} - \frac{b}{h^2} & & \\
        & \ddots & \ddots & \ddots & \\
        & & -\frac{a}{2h} - \frac{b}{h^2} & \frac{2b}{h^2} & \frac{a}{2h} - \frac{b}{h^2} \\
        \frac{a}{2h} - \frac{b}{h^2} & & & -\frac{a}{2h} - \frac{b}{h^2} & \frac{2b}{h^2} \\
    \end{bmatrix}, \quad
    \mathbf{f}_{\mathrm{bc}} =
    \begin{bmatrix}
        0 \\
        \vdots \\
        0
    \end{bmatrix}.
\end{equation}

For the exponential scheme, the relevant terms are
\begin{equation}
    F_L^{\text{exp}} = F_R^{\text{exp}} = \frac{a (u_N e^{\lambda h} - u_1)}{e^{\lambda h} - 1}.
\end{equation}
It is straightforward to obtain the corresponding $N \times N$ ODE system with
\begin{equation}\label{eq:periodic_exp_matrix}
    A =
    \begin{bmatrix}
        \frac{a}{h} + \frac{2a}{h(e^{\lambda h} - 1)} & - \frac{a}{h(e^{\lambda h} - 1)} & & & -\frac{a}{h} - \frac{a}{h(e^{\lambda h} - 1)} \\
        -\frac{a}{h} - \frac{a}{h(e^{\lambda h} - 1)} & \frac{a}{h}+\frac{2a}{h(e^{\lambda h} - 1)} & -\frac{a}{h(e^{\lambda h} - 1)} & & \\
        & \ddots & \ddots & \ddots & \\
        & & -\frac{a}{h} - \frac{a}{h(e^{\lambda h} - 1)} & \frac{a}{h}+\frac{2a}{h(e^{\lambda h} - 1)} & -\frac{a}{h(e^{\lambda h} - 1)} \\
        -\frac{a}{h(e^{\lambda h} - 1)} & & & -\frac{a}{h} - \frac{a}{h(e^{\lambda h} - 1)} & \frac{a}{h} + \frac{2a}{h(e^{\lambda h} - 1)} \\
    \end{bmatrix}, \quad
    \mathbf{f}_{\mathrm{bc}} =
    \begin{bmatrix}
        0 \\
        \vdots \\
        0
    \end{bmatrix}.
\end{equation}

\section{Quantum circuits for different boundary conditions}\label{sec:quantum circuits for different boundary conditions}

In this section, we analyze the matrix structure and construct the corresponding quantum circuits via the LCHS framework, based on the coefficient matrix $A$ of the ODE system derived from the FVM with various flux schemes under different boundary conditions.

\subsection{Matrix product operator representations of coefficient matrices}\label{subsec:mpos}

We introduce several key operators and their matrix product operator (MPO) representations to characterize the coefficient matrices generated by the FVM. We also present the corresponding time-evolution operators for these basic operators, which lay the foundation for constructing quantum circuits via the LCHS framework in the subsequent analysis.

We first consider the shift operators, which capture the subdiagonal and superdiagonal structures of the FVM coefficient matrices.
Following the definition in \cite{sato2024_hamiltonian} and \replytwo{the qubit-ordering convention stated in Section~\ref{subsec:implementation of homogeneous term},} the left and right shift operators are given by

\begin{equation}\label{eq:def of S^- and S^+}
    S^- := \sum_{j=1}^{N-1} \ket{j - 1} \bra{j} =  \sum_{j=1}^{n} s_j^- =
    \begin{bmatrix}
     0 &  1&  & \\
      &  0&  \ddots& \\
      &  &  \ddots &1 \\
      &  &  &0
    \end{bmatrix}_{N\times N}, \quad
    S^+ := \sum_{j=1}^{N-1} \ket{j} \bra{j - 1} = \sum_{j=1}^{n} s_j^+ =
    \begin{bmatrix}
     0 &  &  & \\
      1&  0&  & \\
      &  \ddots&  \ddots & \\
      &  &  1&0
    \end{bmatrix}_{N\times N},
\end{equation}
where the operators $s_j^-$ and $s_j^+$ take the forms
\begin{equation}
    s_j^- := I^{\otimes (n-j)} \otimes \sigma_{01} \otimes \sigma_{10}^{\otimes (j-1)}, \quad
    s_j^+ := I^{\otimes (n-j)} \otimes \sigma_{10} \otimes \sigma_{01}^{\otimes (j-1)},
\end{equation}
and the elementary $2\times2$ matrices are defined as
\begin{equation}
    \sigma_{01} := \begin{bmatrix} 0 & 1 \\ 0 & 0 \end{bmatrix}, \quad
    \sigma_{10} := \begin{bmatrix} 0 & 0 \\ 1 & 0 \end{bmatrix}, \quad
    I := \begin{bmatrix} 1 & 0 \\ 0 & 1 \end{bmatrix}.
\end{equation}

\begin{lemma}[\cite{sato2024_hamiltonian}]\label{lem:operator1}
The operator $e^{i\lambda} s_j^- + e^{-i\lambda} s_j^+$ admits the following form
\begin{equation}
\begin{aligned}
    e^{i\lambda} s_j^- + e^{-i\lambda} s_j^+ &= e^{i\lambda} I^{\otimes (n-j)} \otimes \sigma_{01} \otimes \sigma_{10}^{\otimes (j-1)} + e^{-i\lambda} I^{\otimes (n-j)} \otimes \sigma_{10} \otimes \sigma_{01}^{\otimes (j-1)} \\
    &= I^{\otimes (n-j)} \otimes Q_j(-\lambda) Z \otimes \replytwo{(|1\rangle\langle 1|)^{\otimes (j-1)}} Q_j(-\lambda)^\dagger,
\end{aligned}
\end{equation}
where $Q_j(\lambda) := \left( \prod_{m=1}^{j-1} \mathrm{CNOT}_m^j \right) P_j(\lambda) H_j$.
Here, $H_j$ and $P_j(\lambda) := \begin{bmatrix} 1 & 0 \\ 0 & e^{i \lambda} \end{bmatrix}$ are the Hadamard and phase gates acting on the $j$-th qubit, and $\mathrm{CNOT}_m^j$ is the $\mathrm{CNOT}$ gate acting on the $m$-th qubit controlled by the $j$-th qubit.

The corresponding time-evolution operator $\exp \left( -i \gamma \tau (e^{i \lambda} s_j^- + e^{-i \lambda} s_j^+) \right)$ can be explicitly implemented by
\begin{equation}
    W_j(\gamma \tau, \lambda) :=I^{\otimes (n - j)} \otimes Q_j(-\lambda) \mathrm{CRZ}_j^{[1,j-1]} (2 \gamma \tau) Q_j(-\lambda)^\dagger,
\end{equation}
where $\mathrm{CRZ}_j^{[1,j-1]}(\theta)$ is the multi-controlled $RZ$ gate ($e^{-i\theta Z/2}$) acting on the $j$-th qubit, controlled by the $1,\cdots,(j-1)$-th qubits.
\end{lemma}

\begin{figure}[htbp]
    \centering
    \includegraphics[width=0.8\textwidth]{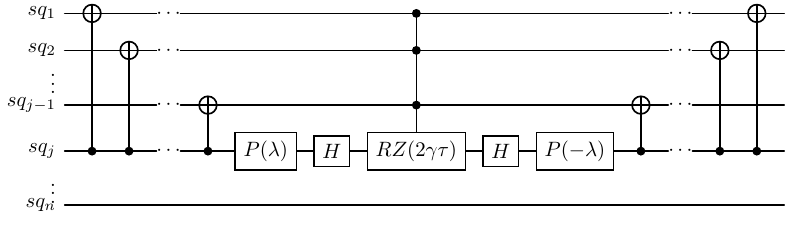}
    \caption{Quantum circuit for $W_j(\gamma \tau, \lambda)$.}
    \label{fig:W_j}
\end{figure}

Next, we introduce two projection operators that characterize the corner entries of the coefficient matrix under the Robin BCs:
\begin{equation}
    \sigma_{00} := \begin{bmatrix} 1 & 0 \\ 0 & 0 \end{bmatrix}, \quad
    \sigma_{11} := \begin{bmatrix} 0 & 0 \\ 0 & 1 \end{bmatrix}.
\end{equation}
It is straightforward to verify that their $n$-fold tensor products are diagonal projectors
\begin{equation}
    \sigma_{00}^{\otimes n} =
    \begin{bmatrix}
        1 & & & \\
        & 0 & & \\
        & & \ddots & \\
        & & & 0
    \end{bmatrix}_{N\times N}, \quad
    \sigma_{11}^{\otimes n} =
    \begin{bmatrix}
        0 & & & \\
        & \ddots & & \\
        & & 0 & \\
        & & & 1
    \end{bmatrix}_{N\times N},
\end{equation}
which correspond to the top-left and bottom-right corner modifications in the FVM matrix under the Robin BCs.

\begin{theorem}\label{th:exp_sigma11}
The time-evolution operator $\exp \left( -i \gamma \tau \sigma_{11}^{\otimes n} \right)$ can be explicitly implemented by
\begin{equation}
    S_n^{(1)}(\gamma \tau) := \mathrm{CP}_n^{[1,n-1]} (-\gamma\tau),
\end{equation}
where $\mathrm{CP}_n^{[1,n-1]}(\theta)$ is the multi-controlled phase gate acting on the $n$-th qubit controlled by the $1,\cdots,(n-1)$-th qubits.
\end{theorem}

\begin{proof}
Note that $\exp\left(-i\gamma\tau\,\sigma_{11}^{\otimes n}\right) = I^{\otimes n} + \left(e^{-i\gamma\tau}-1\right) \replytwo{(|1\rangle\langle 1|)^{\otimes n}}$ and
\replytwo{
\begin{equation}
    \mathrm{CP}_n^{[1,n-1]}(-\gamma\tau)\ket{x_n \cdots x_1} =
    \begin{cases}
    e^{-i\gamma\tau}\ket{x_n \cdots x_1}, & \text{if } x_1=\cdots=x_{n-1}=x_n=1, \\
    \ket{x_n \cdots x_1}, & \text{otherwise}.
    \end{cases}
\end{equation}
coincide on the complete orthonormal basis $\left\{ \ket{x_n \cdots x_1} \,|\, x_i \in\{0,1\} \right\}$ of the $2^n$-dimensional $n$-qubit Hilbert space, hence they are identical.
}
\end{proof}

\begin{proposition}\label{prop:exp_sigma00}
    Since $\sigma_{00} = X\sigma_{11}X$ and $\exp(-i\gamma\tau UAU^\dagger) = U\exp(-i\gamma\tau A)U^\dagger$ holds for any unitary $U$ and Hermitian $A$, the time-evolution operator $\exp \left( -i \gamma \tau \sigma_{00}^{\otimes n} \right)$ can be explicitly implemented by
\begin{equation}
    S_n^{(0)}(\gamma \tau) := X^{\otimes n} S_n^{(1)}(\gamma \tau) X^{\otimes n} = X^{\otimes n} \mathrm{CP}_n^{[1,n-1]} (-\gamma\tau) X^{\otimes n}.
\end{equation}
\end{proposition}

\begin{figure}[htbp]
\centering
\begin{minipage}[t]{0.48\textwidth}
    \centering
    \includegraphics[height=0.15\textheight, keepaspectratio]{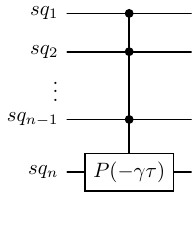}
    \caption{Quantum circuit for $S_n^{(1)}(\gamma \tau)$.}
    \label{fig:S_n^(1)}
\end{minipage}
\begin{minipage}[t]{0.48\textwidth}
    \centering
    \includegraphics[height=0.15\textheight, keepaspectratio]{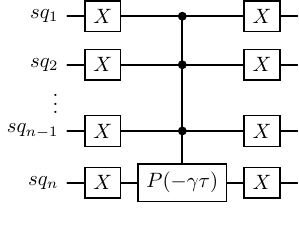}
    \caption{Quantum circuit for $S_n^{(0)}(\gamma \tau)$.}
    \label{fig:S_n^(0)}
\end{minipage}
\end{figure}

Finally, we introduce an operator to represent the top-right and bottom-left corner entries of the coefficient matrix under the periodic BCs.

\begin{lemma}[\cite{sato2024_hamiltonian}]\label{lem:operator2}
The operator $e^{i \lambda}\sigma_{10}^{\otimes n} + e^{-i \lambda}\sigma_{01}^{\otimes n}$ admits the following form
\begin{equation}
    e^{i \lambda}\sigma_{10}^{\otimes n} + e^{-i \lambda}\sigma_{01}^{\otimes n}
    = Q_n(\lambda) \left( I \otimes X^{\otimes (n-1)} \right) Z \otimes \replytwo{(|1\rangle\langle 1|)^{\otimes (n-1)}} \left( I \otimes X^{\otimes (n-1)} \right) Q_n(\lambda)^\dagger.
\end{equation}
The corresponding time-evolution operator $\exp \left( -i \gamma \tau \left( e^{i \lambda}\sigma_{10}^{\otimes n} + e^{-i \lambda}\sigma_{01}^{\otimes n} \right) \right)$
can be explicitly implemented by
\begin{equation}
    V_n(\gamma \tau, \lambda) := Q_n(\lambda) \left(I\otimes X^{\otimes (n-1)}\right) \mathrm{CRZ}_n^{[1,n-1]} (2 \gamma \tau) \left(I\otimes X^{\otimes (n-1)}\right) Q_n(\lambda)^\dagger.
\end{equation}
\end{lemma}

\begin{figure}[htbp]
    \centering
    \includegraphics[width=0.8\textwidth]{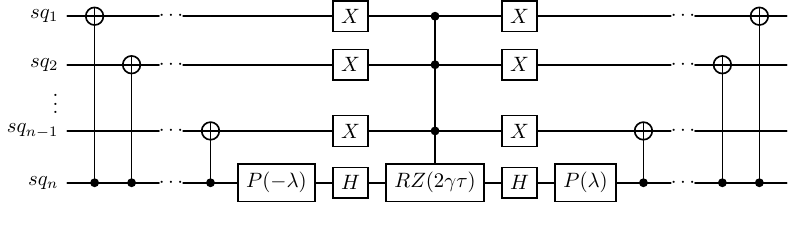}
    \caption{Quantum circuit for $V_n(\gamma \tau, \lambda)$.}
    \label{fig:V_n}
\end{figure}

\subsection{Robin boundary conditions}

Building on the operator representations in Section~\ref{subsec:mpos} and the explicit matrix forms of the ODE system under the Robin BCs (Eqs.~\eqref{eq:robin_central_matrix} and \eqref{eq:robin_exp_matrix}), we rewrite the coefficient matrix $A$ in a unified operator form as
\begin{equation}
    A = 2\alpha I - \left( (\alpha + \beta)S^- + (\alpha - \beta)S^+  + s_0 \sigma_{00}^{\otimes n} + s_1 \sigma_{11}^{\otimes n} \right),
\end{equation}
which decomposes $A$ into its real and imaginary parts
\begin{equation}\label{eq:L and H under robin}
\begin{aligned}
    L &= \frac{A + A^\dagger}{2} = 2\alpha I - \left( \alpha S^- + \alpha S^+  + s_0 \sigma_{00}^{\otimes n} + s_1 \sigma_{11}^{\otimes n} \right), \\
    H &= \frac{A - A^\dagger}{2i} = \beta \left( e^{i\pi/2}S^- + e^{-i\pi/2}S^+ \right),
\end{aligned}
\end{equation}
where $\alpha \geq 0, \beta, s_0, s_1$ are parameters depending on the flux construction scheme.
\replytwo{This unified form can be verified entry-wise by substituting the parameters given below for each scheme.}
For the LCHS implementation, the minimum eigenvalue of $L$ is essential.
To simplify the subsequent analysis, provided that $\alpha \neq 0$, we define the matrix
\begin{equation}
    B := S^- + S^+ + \mu_0 \sigma_{00}^{\otimes n} + \mu_1 \sigma_{11}^{\otimes n} =
    \begin{bmatrix}
    \mu_0 & 1 & & \\
    1 & 0 & 1 & & \\
    & \ddots & \ddots & \ddots & \\
    & & 1 & 0 & 1 & \\
    & & & 1 & \mu_1
    \end{bmatrix}_{N\times N},
\end{equation}
where the parameters are given by $\mu_0 = s_0/\alpha$ and $\mu_1 = s_1/\alpha$.
Accordingly, $L$ simplifies to $L = 2\alpha I - \alpha B$ with the relationship $\lambda_{\min}(L) = \alpha(2-\lambda_{\max}(B))$, where the maximum and minimum eigenvalues of $B$ are analyzed in detail in Appendix~\ref{sec:eigenvalue_of_B}.

\subsubsection{Central scheme}
For the central scheme, the parameters are given by
\begin{equation}
\begin{gathered}
    \alpha = \frac{b}{h^2}, \quad
    \beta = -\frac{a}{2h}, \quad
    s_0 = \frac{(2b + ah)(2\beta_L + h\alpha_L)}{2h^2(2\beta_L - h\alpha_L)}, \quad
    s_1 = \frac{(2b - ah)(2\beta_R - h\alpha_R)}{2h^2(2\beta_R + h\alpha_R)}.
\end{gathered}
\end{equation}

For $b=0$, the governing equation reduces to a pure advection equation, leading to $\alpha = 0$ and $L = - \text{diag}(s_0, 0, \cdots, 0, s_1)$ with $s_0 = \frac{a(2\beta_L + h\alpha_L)}{2h(2\beta_L - h\alpha_L)}, s_1 = -\frac{a(2\beta_R - h\alpha_R)}{2h(2\beta_R + h\alpha_R)}$.
It is straightforward to verify that under the Robin BCs, $\lambda_{\text{min}}(L) = -\max\{s_0, s_1\}$.
Specifically, this identity simplifies to $\lambda_{\text{min}}(L) = -\frac{|a|}{2h}$ for both Dirichlet BCs ($\beta_L = \beta_R = 0$) and Neumann BCs ($\alpha_L = \alpha_R = 0$).

For $b\neq0$, we have
\begin{equation}
    \mu_0 = \frac{s_0}{\alpha} = \frac{(2b + ah)(2\beta_L + h\alpha_L)}{2b(2\beta_L - h\alpha_L)}, \quad
    \mu_1 = \frac{s_1}{\alpha} = \frac{(2b - ah)(2\beta_R - h\alpha_R)}{2b(2\beta_R + h\alpha_R)}.
\end{equation}
Under the Robin BCs, $\lambda_{\text{min}}(L)$ is determined by the specific values of $\mu_0$ and $\mu_1$.
In particular, under the Dirichlet BCs, the coefficients simplify to $\mu_0 = -1 - \frac{ah}{2b}$ and $\mu_1 = -1 + \frac{ah}{2b}$.
Following the analysis in Appendix~\ref{subsec:poly case}, if $(\mu_0, \mu_1)$ lies below the curve of Eq.~\eqref{eq:lam=2}, then we have
\begin{equation}
    \lambda_\text{min}(L) \approx 2\alpha - \alpha \cos \left( \frac{\pi}{N - 1 + \frac{1}{1-\mu_0} + \frac{1}{1-\mu_1} } \right) \geq 0.
\end{equation}
If instead $(\mu_0, \mu_1)$ lies above the curve of Eq.~\eqref{eq:lam=2},
defining $\mu = \max\bigl\{|\mu_0|, |\mu_1|\bigr\}$, we derive
\begin{equation}
    \lambda_\text{min}(L) \approx 2\alpha - \alpha\left( \mu + \frac{1}{\mu} \right) = - \frac{a^2}{2(|a|h+2b)}.
\end{equation}
Similarly, under the Neumann BCs, we have $\mu_0 = 1 + \frac{ah}{2b}, \quad \mu_1 = 1 - \frac{ah}{2b}$.
Note that $\mu_0 + \mu_1 = 2$, so the only solution to Eq.~\eqref{eq:lam=2} is $(\mu_0, \mu_1) = (1,1)$, corresponding to $a=0$.
If $a=0$, then $\mu_0 = \mu_1 = 1$ and $\lambda_\text{min}(L) = 0$.
If $a\neq0$, we still obtain $\lambda_\text{min}(L) \approx 2\alpha - \alpha\left( \mu + \frac{1}{\mu} \right) = - \frac{a^2}{2(|a|h+2b)}$.

\subsubsection{Exponential scheme}
For the exponential scheme, the parameters read
\begin{equation}
\begin{gathered}
    \alpha= \frac{a(e^{\lambda h}+1)}{2h(e^{\lambda h}-1)}, \quad
    \beta = -\frac{a}{2h}, \\
    \replytwo{s_0 = \frac{a \beta_L \lambda e^{\lambda h} + a \alpha_L (e^{\lambda h/2}-1) e^{\lambda h/2}}{h (e^{\lambda h}-1) \left[ \beta_L \lambda - \alpha_L (e^{\lambda h/2}-1) \right]}, \quad
    s_1 = \frac{a \beta_R \lambda - a \alpha_R (e^{\lambda h/2}-1)}{h (e^{\lambda h}-1) \left[ \beta_R \lambda + \alpha_R(1-e^{-\lambda h/2}) \right]},} \\
    \mu_0 = \frac{s_0}{\alpha} = \frac{2 \beta_L \lambda e^{\lambda h} + 2 \alpha_L (e^{\lambda h/2}-1) e^{\lambda h/2}}{(e^{\lambda h}+1) \left[ \beta_L \lambda - \alpha_L (e^{\lambda h/2}-1) \right]}, \quad
    \mu_1 = \frac{s_1}{\alpha} = \frac{2 \beta_R \lambda - 2 \alpha_R (e^{\lambda h/2}-1)}{ (e^{\lambda h}+1) \left[ \beta_R \lambda + \alpha_R(1-e^{-\lambda h/2}) \right]}.
\end{gathered}
\end{equation}

As in the central scheme above, $\lambda_{\text{min}}(L)$ is determined by the specific values of $\mu_0$ and $\mu_1$.
In particular, we derive explicit results for the Dirichlet and Neumann BCs as two special cases.

First, under the Dirichlet BCs, we directly derive $\mu_0 = \mu_1 = -\frac{2e^{\lambda h/2}}{e^{\lambda h}+1}$, which yields
\begin{equation}
    \lambda_\text{min}(L) \approx 2\alpha - \alpha \cos \left( \frac{\pi}{N - 1 + \frac{1}{1-\mu_0} + \frac{1}{1-\mu_1} } \right) \geq 0.
\end{equation}

Next, under the Neumann BCs, we have $\mu_0 = \frac{2e^{\lambda h}}{e^{\lambda h}+1}, \mu_1 = \frac{2}{e^{\lambda h}+1}$, which still gives $\mu_0 + \mu_1 = 2$.
If $a=0$, then $\mu_0 = \mu_1 = 1$ and $\lambda_\text{min}(L)= 0$; if $a>0$, we obtain
\begin{equation}
    \lambda_\text{min}(L) \approx 2\alpha - \alpha \left( \mu_0 + \frac{1}{\mu_0} \right)
    = - \frac{a(e^{\lambda h}-1)}{4he^{\lambda h}};
\end{equation}
and if $a<0$, it follows that
\begin{equation}
    \lambda_\text{min}(L) \approx 2\alpha - \alpha \left( \mu_1 + \frac{1}{\mu_1} \right) = - \frac{a(e^{\lambda h}-1)}{4h}.
\end{equation}

\begin{remark}\label{remark:shift_L}
\replyone{The estimates above show that $\lambda_{\min}(L)$ may become negative under the Robin or Neumann BCs with $a\neq0$, in which case the condition $L\succeq0$ required by the LCHS method is violated. In the remaining cases analyzed above, namely the Dirichlet BCs and the Robin BCs with $(\mu_0,\mu_1)$ lying below the curve of Eq.~\eqref{eq:lam=2}, the estimates above give $\lambda_{\min}(L)\geq0$ already. When $\lambda_{\min}(L)<0$, the condition can be restored by a shift of the generator: enlarging the reaction coefficient $c$ in Eq.~\eqref{eq:gov_eq} by an amount $\Delta c\geq-\min\{0,\lambda_{\min}(L)\}$, i.e., replacing $A$ by $A+\Delta c\,I$, guarantees $L+\Delta c\,I\succeq0$. The solution of the original system is then recovered exactly by the classical rescaling $\mathbf{u}(t)=e^{\Delta c\,t}\mathbf{v}(t)$, where $\mathbf{v}$ solves the shifted system with generator $A+\Delta c\,I$ and source term $e^{-\Delta c\,t}\mathbf{f}(t)$. The shift therefore introduces an amplification factor $e^{\Delta c\,T}$ in the recovered solution, which enters the success probability and hence the overall complexity (cf.~Section~\ref{sec:complexity analysis}). Importantly, for the present discretizations a negative $\lambda_{\min}(L)$ approaches $-a^{2}/(4b)$ as the mesh is refined, so the required shift and the associated amplification factor remain bounded under grid refinement.}
\end{remark}

\subsubsection{Quantum circuit for the Robin boundary conditions}\label{subsubsec:quantum circuit under the robin boundary conditions}

From the definitions of $S^-$ and $S^+$ in Eq.~\eqref{eq:def of S^- and S^+}, and the commutation relations (detailed in Appendix~\ref{sec:trotter error}, Eq.~\eqref{eq:first order commutator relationships})
\begin{equation}
\begin{aligned}
    \left[ e^{i\lambda} s_{j,n}^- + e^{-i\lambda} s_{j,n}^+, e^{i\lambda} s_{1,n}^- + e^{-i\lambda} s_{1,n}^+ \right] &\neq 0, \quad j >1, \\
    \left[ e^{i\lambda} s_{j,n}^- + e^{-i\lambda} s_{j,n}^+, e^{i\lambda} s_{j',n}^- + e^{-i\lambda} s_{j',n}^+ \right] &= 0, \quad j \geq j' >1,
\end{aligned}
\end{equation}
we define the following Hamiltonian components
\begin{equation}\label{eq:def of H1 and H2}
    \mathcal{H}_1(\lambda, n) :=  e^{i\lambda} s_{1,n}^- + e^{-i\lambda} s_{1,n}^+, \quad
    \mathcal{H}_2(\lambda, n) := \sum_{j=2}^{n} \left(e^{i\lambda} s_{j,n}^- + e^{-i\lambda} s_{j,n}^+\right).
\end{equation}
Substituting Eq.~\eqref{eq:def of H1 and H2} into Eq.~\eqref{eq:L and H under robin}, we rewrite $L$ (suppose $\alpha \neq 0$, the case that $\alpha = 0$ is analyzed in Appendix~\ref{subsec:robin with alpha=0}) and $H$ as
\begin{equation}
\begin{aligned}
    L &= 2\alpha I - \alpha B = 2\alpha I - \left( \alpha \mathcal{H}_1 \left( 0, n \right) + \alpha \mathcal{H}_2 \left( 0, n \right) + s_0 \sigma_{00}^{\otimes n} + s_1 \sigma_{11}^{\otimes n} \right), \\
    H &=  \beta \mathcal{H}_1 \left( \frac{\pi}{2}, n \right) + \beta \mathcal{H}_2 \left( \frac{\pi}{2}, n  \right).
\end{aligned}
\end{equation}
The commutator between $L$ and $H$ is computed as
\begin{equation}
\begin{aligned}
    \left[L, H \right]
    =& 2i \alpha \beta \sum_{j=1}^n I^{\otimes(n-j)} \otimes \left( \sigma_{00} \otimes \sigma_{11}^{\otimes(j-1)} - \sigma_{11} \otimes \sigma_{00}^{\otimes(j-1)} \right) \\
    &- \beta \left( s_0 \sigma_{00}^{\otimes(n-1)} - s_1 \sigma_{11}^{\otimes(n-1)} \right) \otimes \left( e^{i \pi/2} \sigma_{01} - e^{-i \pi/2} \sigma_{10} \right),
\end{aligned}
\end{equation}
which is non-vanishing. We therefore apply the second-order Trotter-Suzuki (TS) formula (also termed the second-order product formula) to $U_j(\tau) = \exp\left(-i(H + r_j L)\tau \right)$ in the LCHS framework, obtaining
\begin{equation}\label{eq:trotter of Uj under robin}
    U_j(\tau) \approx \exp\left( -i H\frac{\tau}{2} \right) \exp\left( -ir_j L\tau \right) \exp\left( -iH\frac{\tau}{2} \right)
\end{equation}
with the Trotter error~\cite{Childs2021_Theory} bounded by
\begin{equation}
\begin{aligned}
    e_{j,1}
    &:= \left\| U_j(\tau) - \exp\left( -i H\frac{\tau}{2} \right)
    \exp\left( -ir_j L\tau \right)
    \exp\left( -iH\frac{\tau}{2} \right) \right\| \\
    &\leq \frac{r_j^2 \tau^3}{12} \left\| \left[ L, [L, H] \right] \right\| + \frac{|r_j|\tau^3}{24} \left\| \left[ H, [H, L] \right] \right\| \\
    & = \frac{|r_j\beta|\tau^3}{12} \left( |r_j|(2\alpha^2 + s^2) + |\beta| \sqrt{\alpha^2 + s^2} \right),
\end{aligned}
\end{equation}
where $s := \max \left\{ |s_0|, |s_1| \right\}$ and we utilize the results
\begin{equation}
    \left\| \left[L, [L, H] \right] \right\| = |\beta| \left( 2\alpha^2 + s^2 \right), \quad
    \left\| \left[H, [H, L] \right] \right\| = 2 \beta^2 \sqrt{\alpha^2 + s^2}
\end{equation}
in Eqs.~\eqref{eq:norm of [L,[L,H]]} and \eqref{eq:norm of [H,[H,L]]}.

Next apply the second-order TS formula to $\exp\left( -iH\frac{\tau}{2} \right)$, yielding
\begin{equation}\label{eq:trotter of H under robin}
    \exp\left( -iH\frac{\tau}{2} \right)
    \approx \exp\left(-i \beta \mathcal{H}_1\left(\frac{\pi}{2},n\right) \frac{\tau}{4} \right)
    \exp\left(-i \beta \mathcal{H}_2\left(\frac{\pi}{2},n\right)\frac{\tau}{2} \right)
    \exp\left(-i \beta \mathcal{H}_1\left(\frac{\pi}{2},n\right) \frac{\tau}{4} \right)
\end{equation}
with the Trotter error bounded by
\begin{equation}
\begin{aligned}
    e_{j,2}
    \leq& \frac{|\beta|^3 \tau^3}{96} \left\| \left[ \mathcal{H}_2\left(\frac{\pi}{2}, n\right), \left[ \mathcal{H}_2\left(\frac{\pi}{2}, n\right), \mathcal{H}_1\left(\frac{\pi}{2}, n\right) \right] \right] \right\|
    + \frac{|\beta|^3 \tau^3}{192} \left\| \left[ \mathcal{H}_1\left(\frac{\pi}{2}, n\right), \left[ \mathcal{H}_1\left(\frac{\pi}{2}, n\right), \mathcal{H}_2\left(\frac{\pi}{2}, n\right) \right] \right] \right\| \\
    \leq& \frac{|\beta|^3 \tau^3}{96} \cdot 4 + \frac{|\beta|^3 \tau^3}{192} \cdot 4 = \frac{|\beta|^3 \tau^3}{16},
\end{aligned}
\end{equation}
where we employ the bounds
\begin{equation}
\begin{aligned}
    \left\| \left[ \mathcal{H}_2\left(\frac{\pi}{2}, n\right), \left[ \mathcal{H}_2\left(\frac{\pi}{2}, n\right), \mathcal{H}_1\left(\frac{\pi}{2}, n\right) \right] \right] \right\|
    &\leq 4, \\
    \left\| \left[ \mathcal{H}_1\left(\frac{\pi}{2}, n\right), \left[ \mathcal{H}_1\left(\frac{\pi}{2}, n\right), \mathcal{H}_2\left(\frac{\pi}{2}, n\right) \right] \right] \right\|
    &= 4 \cos \left( \frac{\pi}{2^{n-1}+1} \right)
    \leq 4,
\end{aligned}
\end{equation}
established in Eqs.~\eqref{eq:norm of [H2,[H2,H1]]} and \eqref{eq:norm of [H1,[H1,H2]]}.

Finally, using the fact that $\mathcal{H}_2(\lambda, n)$, $\sigma_{00}^{\otimes n}$, and $\sigma_{11}^{\otimes n}$ mutually commute (see Eq.~\eqref{eq:first order commutator relationships}), we apply the second-order TS formula to $\exp\left(-ir_jL\tau\right)$ and substituting $r_j = j\Delta r - \widetilde{R}$, giving
\begin{equation}\label{eq:trotter of L under robin}
\begin{aligned}
    \exp\left( -i r_j L\tau \right)
    \approx& \exp\left( -i 2\alpha r_j\tau \right)
    \exp\left( i\alpha r_j \mathcal{H}_1\left(0,n\right) \frac{\tau}{2} \right)
    \exp\left(i\alpha r_j \mathcal{H}_2\left(0,n\right) \tau \right) \\
    &\exp\left(i s_0 r_j \sigma_{00}^{\otimes n} \tau \right)
    \exp\left(i s_1 r_j \sigma_{11}^{\otimes n} \tau \right)
    \exp\left( i\alpha r_j \mathcal{H}_1\left(0,n\right) \frac{\tau}{2} \right) \\
    =& \exp\left(i 2\alpha \widetilde{R}\tau \right)
    \exp\left(-i2\alpha \Delta r \tau j\right)
    \exp\left(-i \alpha \widetilde{R} \mathcal{H}_1\left(0,n\right) \frac{\tau}{2} \right) \\
    &\exp\left(-i (-\alpha\Delta r) \mathcal{H}_1\left(0,n\right) \frac{\tau}{2}j \right)
    \exp\left(-i \alpha \widetilde{R} \mathcal{H}_2\left(0,n\right) \tau \right) \\
    &\exp\left(-i (-\alpha\Delta r) \mathcal{H}_2\left(0,n\right) \tau j \right)
    \exp\left(-i s_0 \widetilde{R} \sigma_{00}^{\otimes n} \tau \right) \\
    &\exp\left(-i (-s_0\Delta r) \sigma_{00}^{\otimes n} \tau j \right)
    \exp\left(-i s_1 \widetilde{R} \sigma_{11}^{\otimes n} \tau \right)
    \exp\left(-i (-s_1\Delta r) \sigma_{11}^{\otimes n} \tau j \right) \\
    &\exp\left(-i \alpha \widetilde{R} \mathcal{H}_1\left(0,n\right) \frac{\tau}{2} \right)
    \exp\left(-i (-\alpha\Delta r) \mathcal{H}_1\left(0,n\right) \frac{\tau}{2}j \right),
\end{aligned}
\end{equation}
with the Trotter error bounded by
\begin{equation}
\begin{aligned}
    e_{j,3}
    \leq& \frac{|r_j|^3 \tau^3}{12} \left\| \left[ \alpha \mathcal{H}_2(0, n) + s_0 \sigma_{00}^{\otimes n} + s_1 \sigma_{11}^{\otimes n}, [\alpha \mathcal{H}_2(0, n) + s_0 \sigma_{00}^{\otimes n} + s_1 \sigma_{11}^{\otimes n}, \alpha \mathcal{H}_1(0, n)] \right] \right\| \\
    &+ \frac{|r_j|^3 \tau^3}{24} \left\| \left[ \alpha \mathcal{H}_1(0, n), [\alpha \mathcal{H}_1(0, n), \alpha \mathcal{H}_2(0, n) + s_0 \sigma_{00}^{\otimes n} + s_1 \sigma_{11}^{\otimes n}] \right] \right\| \\
    \leq& \begin{cases}
        \frac{\alpha^3|r_j|^3\tau^3}{2}, &|\lambda|_{\max}(B) \leq 2; \\
        \frac{|r_j|^3\tau^3}{12} \left( 4 \alpha^3 + \alpha \sqrt{\alpha^4 + s^4 - \alpha^2 s^2} + \alpha^3 \left( \frac{s}{\alpha} + \frac{\alpha}{s} \right) \right), &|\lambda|_{\max}(B) > 2.
    \end{cases}
\end{aligned}
\end{equation}
Here we use the facts that if $|\lambda|_{\max}(B) \leq 2$ then
\begin{equation}
\begin{aligned}
    \left\| \left[ \alpha \mathcal{H}_2(0, n) + s_0 \sigma_{00}^{\otimes n} + s_1 \sigma_{11}^{\otimes n}, [\alpha \mathcal{H}_2(0, n) + s_0 \sigma_{00}^{\otimes n} + s_1 \sigma_{11}^{\otimes n}, \alpha \mathcal{H}_1(0, n)] \right] \right\|
    &\leq 4 \alpha^3, \\
    \left\| \left[ \alpha \mathcal{H}_1(0, n), [\alpha \mathcal{H}_1(0, n), \alpha \mathcal{H}_2(0, n) + s_0 \sigma_{00}^{\otimes n} + s_1 \sigma_{11}^{\otimes n}] \right] \right\|
    &\leq 4 \alpha^3,
\end{aligned}
\end{equation}
from Eqs.~\eqref{eq:norm of [H2,[H2,H1]] in L when abs_lam <= 2} and \eqref{eq:norm of [H1,[H1,H2]] in L when abs_lam <= 2}; if $|\lambda|_{\max}(B) > 2$ then
\begin{equation}
\begin{aligned}
    \left\| \left[ \alpha \mathcal{H}_2(0, n) + s_0 \sigma_{00}^{\otimes n} + s_1 \sigma_{11}^{\otimes n}, [\alpha \mathcal{H}_2(0, n) + s_0 \sigma_{00}^{\otimes n} + s_1 \sigma_{11}^{\otimes n}, \alpha \mathcal{H}_1(0, n)] \right] \right\|
    &\leq 4 \alpha^3 + \alpha \sqrt{\alpha^4 + s^4 - \alpha^2 s^2}, \\
    \left\| \left[ \alpha \mathcal{H}_1(0, n), [\alpha \mathcal{H}_1(0, n), \alpha \mathcal{H}_2(0, n) + s_0 \sigma_{00}^{\otimes n} + s_1 \sigma_{11}^{\otimes n}] \right] \right\|
    &\leq 2 \alpha^3 \left( \frac{s}{\alpha} + \frac{\alpha}{s} \right),
\end{aligned}
\end{equation}
from Eqs.~\eqref{eq:norm of [H2,[H2,H1]] in L when abs_lam > 2} and \eqref{eq:norm of [H1,[H1,H2]] in L when abs_lam > 2}.

Via Lemmas~\ref{lem:select oracle}, \ref{lem:operator1}, Theorem~\ref{th:exp_sigma11} and Proposition~\ref{prop:exp_sigma00}, and combining Eqs.~\eqref{eq:trotter of Uj under robin}, \eqref{eq:trotter of H under robin} and \eqref{eq:trotter of L under robin}, we construct the quantum circuit for the select oracle under the Robin BCs as
\begin{equation}\label{eq:select under robin}
\begin{aligned}
    \mathrm{SEL_R}(\tau)
    =& \sum_{j=0}^{M - 1}|j\rangle\langle j| \otimes U_j(\tau) \\
    \approx& \replytwo{\mathcal{U}_H^{\mathrm{R}}\left( \frac{\beta\tau}{2}, \frac{\pi}{2} \right) \cdot \mathcal{U}_{\mathrm{ph}}(\alpha\tau) \cdot
    \mathfrak{B}_{\alpha\tau/2}\left[W_1(\cdot,0)\right] \cdot
    \mathfrak{B}_{\alpha\tau}\left[ \prod_{k=2}^n W_k(\cdot,0) \right]} \\
    & \replytwo{\cdot
    \mathfrak{B}_{s_0\tau}\left[S_n^{(0)}\right] \cdot
    \mathfrak{B}_{s_1\tau}\left[S_n^{(1)}\right] \cdot
    \mathfrak{B}_{\alpha\tau/2}\left[W_1(\cdot,0)\right] \cdot \mathcal{U}_H^{\mathrm{R}}\left( \frac{\beta\tau}{2}, \frac{\pi}{2} \right),}
\end{aligned}
\end{equation}
\replytwo{where
\begin{equation}\label{eq:shorthand blocks}
\begin{aligned}
    \mathfrak{B}_{\theta}[F] &:= F(\theta \widetilde{R})
    \prod_{j=1}^{m} \left[ \mathrm{C} F\left(-2^{j-1}\theta\Delta r\right) \right]_{[1,n]}^j, \quad
    \mathcal{U}_{\mathrm{ph}}(\theta) := G\left(2\theta\widetilde{R}\right)
    \left[ \prod_{j=1}^{m} P_j\left(-2^j\theta\Delta r\right) \right], \\
    \mathcal{U}_H^{\mathrm{R}}(\theta,\lambda) &:= W_1\left(\frac{\theta}{2},\lambda\right)
    \left[ \prod_{k=2}^{n} W_k\left(\theta,\lambda\right) \right]
    W_1\left(\frac{\theta}{2},\lambda\right).
\end{aligned}
\end{equation}
Here $\mathfrak{B}_{\theta}[F]$ collects the binary-decomposition pattern shared by all register-resolved gates, with the controls taken from the ancilla register as in Lemma~\ref{lem:select oracle}}, and $G(\theta):=e^{i\theta}I^{\otimes(m+n)}$ denotes the global phase gate.

\begin{remark}\label{remark:trotter error under robin}
Denote $\widetilde{U}_j(\tau)$ as the second-order TS decomposition of $U_j(\tau)$.
Under the Robin BCs, the total Trotter error $e_j = e_{j,1} + 2e_{j,2} + e_{j,3}$ from $\widetilde{U}_j(\tau)$ to $U_j(\tau)$ satisfies $e_j \leq (w_0 + |r_j| w_1 + |r_j|^2 w_2 + |r_j|^3 w_3) \tau^3$ with
\begin{equation}
\begin{aligned}
    w_0 &= \frac{|\beta|^3}{8}, \quad
    w_1 = \frac{\beta^2}{12} \sqrt{\alpha^2 + s^2}, \quad
    w_2 = \frac{|\beta|}{12} (2\alpha^2 + s^2), \\
    w_3 &=
    \begin{cases}
        \frac{\alpha^3}{2}, &|\lambda|_{\max}(B) \leq 2; \\
        \frac{1}{12} \left( 4 \alpha^3 + \alpha \sqrt{\alpha^4 + s^4 - \alpha^2 s^2} + \alpha^3 \left( \frac{s}{\alpha} + \frac{\alpha}{s} \right) \right), &|\lambda|_{\max}(B) > 2.
    \end{cases}
\end{aligned}
\end{equation}
\end{remark}

Denoting the approximate quantum circuit at the end of Eq.~\eqref{eq:select under robin} as $\widetilde{\mathrm{SEL}}_{\mathrm{R}}(\tau)$, we establish the following circuit error bound for implementing the select oracle $\mathrm{SEL_R}(\tau)$ by $\widetilde{\mathrm{SEL}}_{\mathrm{R}}(\tau)$.

\begin{theorem}\label{th:circuit error under robin}
Under the Robin BCs, the circuit error between $\sum_{j=0}^{M-1} c_j U_j(\tau)$ and $\sum_{j=0}^{M-1} c_j \widetilde{U}_j(\tau)$ satisfies
\begin{equation}\label{eq:discrete kernel moment bound}
\begin{aligned}
    \left\| \sum_{j=0}^{M-1} c_j U_j(\tau) - \sum_{j=0}^{M-1} c_j \widetilde{U}_j(\tau) \right\|
    &\leq \sum_{j=0}^{M-1} |c_j| e_j
    \replythr{\leq \tau^3 \sum_{k=0}^3 I_k^{(M)}(\gamma,\delta)\, w_k
    \leq \tau^3 \sum_{k=0}^3 \left( I_k(\gamma,\delta) + \varepsilon_{k,M} \right) w_k
    \leq 2\tau^3\sum_{k=0}^3 I_k(\gamma,\delta)\,w_k ,}
\end{aligned}
\end{equation}
where
\begin{equation}
\begin{aligned}
    I_k(\gamma,\delta)
    &:= \frac{1}{\sqrt{2\pi}} \int_{\mathbb{R}} |\replythr{\hat{\kappa}}(r;\gamma,\delta)| |r|^k \, \mathrm{d} r
    = \frac{2 e^{\delta - \frac{1}{4\gamma^2}}}{\pi} \int_{0}^{+ \infty} \frac{r^k}{1+r^2} e^{-\frac{r^2}{4\gamma^2}} \, \mathrm{d} r
    = \frac{e^\delta}{\pi} \Gamma\left(\frac{k+1}{2}\right) \Gamma\left(\frac{1 - k}{2}, \frac{1}{4\gamma^2}\right), \\
    I_k^{(M)}(\gamma,\delta)
    &:=\sum_{j=0}^{M-1}|c_j||r_j|^k
    =\frac{1}{\sqrt{2\pi}}\sum_{j=0}^{M-1} |\replythr{\hat{\kappa}}(r_j;\gamma,\delta)|\,|r_j|^k \Delta r.
\end{aligned}
\end{equation}
\replythr{Here $I_k^{(M)}$ denote the discrete kernel moments, $I_k(\gamma,\delta)$ denote their whole-line counterparts, and $\varepsilon_{k,M}:=|I_k^{(M)}-I_k|$ collects the scalar quadrature and tail errors of the weighted kernel moments; the last inequality in Eq.~\eqref{eq:discrete kernel moment bound} holds in the resolved quadrature regime where $\varepsilon_{k,M}\leq I_k$ for $k=0,1,2,3$,}
and by direct calculation and estimation we have
\begin{equation}
\begin{gathered}
    I_0(\gamma, \delta) = e^\delta \mathrm{erfc}\left(\frac{1}{2 \gamma}\right), \quad
    I_1(\gamma, \delta) = \frac{e^\delta}{\pi} \left( 2\ln \gamma + \frac{1}{4\gamma^2} +2 \ln 2 - \gamma_E + \mathcal{O}\left( \frac{1}{\gamma^4} \right) \right), \\
    I_2(\gamma, \delta) = e^\delta \left( \frac{2 \gamma}{\sqrt{\pi}} e^{-\frac{1}{4\gamma^2}} - \mathrm{erfc} \left( \frac{1}{2\gamma} \right) \right), \quad
    I_3(\gamma, \delta) = \frac{e^\delta}{\pi} \left( 4 \gamma^2 e^{-\frac{1}{4\gamma^2}} -2\ln \gamma - \frac{1}{4\gamma^2} -2 \ln 2 + \gamma_E + \mathcal{O}\left( \frac{1}{\gamma^4} \right) \right),
\end{gathered}
\end{equation}
where $\gamma_E = 0.5772\ldots$ is the Euler–Mascheroni constant.
\end{theorem}

\begin{proof}
\replythr{The first inequality in Eq.~\eqref{eq:discrete kernel moment bound} is the triangle inequality applied to the weighted sum, where $e_j$ bounds the per-node error $\|U_j(\tau)-\widetilde{U}_j(\tau)\|$ as estimated in Remark~\ref{remark:trotter error under robin}. Substituting the per-node bound $e_j\leq\tau^3\sum_{k=0}^3|r_j|^kw_k$ from Remark~\ref{remark:trotter error under robin} and interchanging the two finite sums yield the second inequality,
\begin{equation}
\sum_{j=0}^{M-1}|c_j|e_j
\leq \tau^3\sum_{j=0}^{M-1}|c_j|\sum_{k=0}^3|r_j|^kw_k
=\tau^3\sum_{k=0}^3 I_k^{(M)}(\gamma,\delta)\, w_k ,
\end{equation}
where the last step is the definition of the discrete kernel moments. The third inequality in Eq.~\eqref{eq:discrete kernel moment bound} follows from $I_k^{(M)}\leq I_k+|I_k^{(M)}-I_k|=I_k+\varepsilon_{k,M}$, which is the triangle inequality applied to the scalar quadrature and tail error, and the last inequality then follows from the condition $\varepsilon_{k,M}\leq I_k$.}
\end{proof}

\begin{remark}\label{remark:time steps}
A standard approach to mitigate the Trotter error is to decompose the total evolution time $\sum_{k=1}^{\replythr{N_t}} \alpha_k t$ into $\replythr{N_t}$ successive steps with duration $\alpha_k t$ ($\alpha_k > 0$ for all $k$), and apply the corresponding evolution operators iteratively. By the triangle inequality, we derive
\begin{equation}
    \left\| \sum_{j=0}^{M-1} c_j U_j \left( \sum_{k=1}^{\replythr{N_t}} \alpha_k t \right) - \sum_{j=0}^{M-1} c_j \prod_{k=1}^{\replythr{N_t}} \widetilde{U}_j(\alpha_k t) \right\|
    \leq \sum_{k=1}^{\replythr{N_t}} \alpha_k^3 \sum_{j=0}^{M-1} |c_j| \left\| U_j(t) - \widetilde{U}_j(t) \right\|.
\end{equation}
Specifically, setting $\alpha_k = 1/\replythr{N_t}$ (uniform \replythr{Trotter} step partitioning), we obtain
\begin{equation}
    \left\| \sum_{j=0}^{M-1} c_j U_j(t) - \sum_{j=0}^{M-1} c_j \left( \widetilde{U}_j \left( \frac{t}{\replythr{N_t}} \right) \right)^{\replythr{N_t}} \right\|
    \leq \frac{1}{\replythr{N_t}^2} \sum_{j=0}^{M-1} |c_j| \left\| U_j(t) - \widetilde{U}_j(t) \right\|.
\end{equation}
\end{remark}

\subsection{Periodic boundary conditions}

As for the Robin BCs, we rewrite the coefficient matrix of the ODE system under the periodic BCs (Eqs.~\eqref{eq:periodic_central_matrix} and~\eqref{eq:periodic_exp_matrix}) in a unified operator form as
\begin{equation}
    A = 2\alpha I - \left( (\alpha + \beta)(S^- + \sigma_{10}^{\otimes n}) + (\alpha - \beta)(S^+ + \sigma_{01}^{\otimes n}) \right),
\end{equation}
where $\alpha = \frac{b}{h^2}, \beta = -\frac{a}{2h}$ for the central scheme, and $\alpha = \frac{a(e^{\lambda h}+1)}{2h(e^{\lambda h}-1)}$, $\beta = -\frac{a}{2h}$ for the exponential scheme.

We decompose $A$ into its real and imaginary parts as
\begin{equation}\label{eq:L and H under periodic}
\begin{aligned}
    L &= \frac{A + A^\dagger}{2} = 2\alpha I - \alpha \left( S^- + S^+ + \sigma_{10}^{\otimes n} + \sigma_{01}^{\otimes n} \right), \\
    H &= \frac{A - A^\dagger}{2i} = \beta \left( e^{i\pi/2} S^- + e^{-i\pi/2} S^+ + e^{i\pi/2} \sigma_{10}^{\otimes n} + e^{-i\pi/2} \sigma_{01}^{\otimes n} \right).
\end{aligned}
\end{equation}
By direct calculation, the eigenvalues of $L$ are
\begin{equation}\label{eq:eigenvalue under periodic}
    \lambda_k(L) = 2\alpha - 2\alpha \cos \left( \frac{2(k-1) \pi}{N} \right), \quad
    k = 1,\cdots,N,
\end{equation}
which implies $\lambda_\text{min}(L)=0$\replyone{, so that the condition $L\succeq0$ required by the LCHS method holds automatically under the periodic BCs.}

\subsubsection{Quantum circuit for the periodic boundary conditions}\label{subsubsec:quantum circuit under the periodic boundary conditions}

For simplicity, we first define the Hamiltonian component associated with periodic BCs as
\begin{equation}\label{eq:def of H3}
    \mathcal{H}_3(\lambda, n)
    := e^{i \lambda}\sigma_{10}^{\otimes n}
    + e^{-i \lambda}\sigma_{01}^{\otimes n}.
\end{equation}
Substituting Eqs.~\eqref{eq:def of H1 and H2} and \eqref{eq:def of H3} into Eq.~\eqref{eq:L and H under periodic}, $L$ and $H$ take the compact form
\begin{equation}
\begin{aligned}
    L &= 2\alpha I - \alpha \left( \mathcal{H}_1 \left(0,n\right) + \mathcal{H}_2 \left(0,n\right) + \mathcal{H}_3 \left(0,n\right) \right), \\
    H &=  \beta \left( \mathcal{H}_1 \left( \frac{\pi}{2}, n \right) + \mathcal{H}_2 \left( \frac{\pi}{2}, n \right) + \mathcal{H}_3 \left( \frac{\pi}{2}, n \right) \right).
\end{aligned}
\end{equation}
Here, we restrict our attention to the case $\alpha \neq 0$, while the scenario $\alpha = 0$ is analyzed separately in Appendix~\ref{subsec:periodic with alpha=0}.

A key simplification arises from the cyclic shift operator
\begin{equation}
    S := \sigma_{10}^{\otimes n} + \sum_{j=1}^{n} s_j^- =
    \begin{bmatrix}
        0 & 1 & & \\
        & 0 & \ddots & \\
        & & \ddots & 1 \\
        1 & & & 0
    \end{bmatrix}_{N\times N},
\end{equation}
which satisfies $L = 2\alpha I - \alpha(S + S^\dagger)$ and $H = i \beta (S - S^\dagger)$.
Direct computation yields
\begin{equation}
    \left[ L, H \right] = -i \alpha \beta \left[ S + S^\dagger, S - S^\dagger \right] = 0,
\end{equation}
so $U_j(\tau)$ under the periodic BCs factorizes exactly as
\begin{equation}\label{eq:trotter of Uj under periodic}
    U_j(\tau) = \exp\left(-i(H + r_j L)\tau\right) = \exp\left(-iH\tau\right) \exp\left(-ir_j L\tau\right).
\end{equation}

Using the commutativity of $\mathcal{H}_2(\lambda,n)$ and $\mathcal{H}_3(\lambda,n)$ (cf.~Eq.~\eqref{eq:first order commutator relationships}), we apply the second-order TS formula to $\exp\left(-iH\tau\right)$ and obtain
\begin{equation}\label{eq:trotter of H under periodic}
\begin{aligned}
    \exp\left(-iH\tau\right)
    \approx& \exp\left(-i \beta \mathcal{H}_1\left(\frac{\pi}{2},n\right) \frac{\tau}{2} \right)
    \exp\left(-i \beta \mathcal{H}_2\left(\frac{\pi}{2},n\right)\tau\right) \\
    &\exp\left(-i \beta \mathcal{H}_3\left(\frac{\pi}{2},n\right)\tau\right)
    \exp\left(-i \beta \mathcal{H}_1\left(\frac{\pi}{2},n\right) \frac{\tau}{2} \right),
\end{aligned}
\end{equation}
with the Trotter error bounded by
\begin{equation}
    e_{j,1}
    \leq \frac{|\beta|^3 \tau^3}{12} \left\| \left[ \mathcal{H}_2+\mathcal{H}_3, \left[ \mathcal{H}_2+\mathcal{H}_3, \mathcal{H}_1 \right] \right] \right\|_{\pi/2}
    + \frac{|\beta|^3 \tau^3}{24} \left\| \left[ \mathcal{H}_1, \left[ \mathcal{H}_1, \mathcal{H}_2+\mathcal{H}_3 \right] \right] \right\|_{\pi/2}
    \leq \frac{|\beta|^3 \tau^3}{2},
\end{equation}
where we use the facts that both double commutator norms equal $4$ from Eqs.~\eqref{eq:norm of [H2+H3,[H2+H3,H1]]} and \eqref{eq:norm of [H1,[H1,H2+H3]]} and the subscript $\pi/2$ denotes evaluation at $\lambda=\pi/2$.

Next, applying the second-order TS formula to $\exp\left(-ir_j L\tau\right)$, we get
\begin{equation}\label{eq:trotter of L under periodic}
\begin{aligned}
    \exp\left(-ir_jL\tau\right)
    \approx& \exp\left(-i2 \alpha r_j\tau\right)
    \exp\left(i \alpha r_j \mathcal{H}_1(0,n) \frac{\tau}{2} \right)
    \exp\left(i \alpha r_j \mathcal{H}_2(0,n) \tau \right)
    \exp\left(i \alpha r_j \mathcal{H}_3(0,n) \tau \right)
    \exp\left(i \alpha r_j \mathcal{H}_1(0,n) \frac{\tau}{2} \right) \\
    =& \exp\left(i 2\alpha \widetilde{R}\tau \right)
    \exp\left(-i2\alpha \Delta r \tau j\right)
    \exp\left(-i \alpha \widetilde{R} \mathcal{H}_1(0,n) \frac{\tau}{2} \right)
    \exp\left(-i (-\alpha\Delta r) \mathcal{H}_1(0,n) \frac{\tau}{2}j \right) \\
    &\exp\left(-i \alpha \widetilde{R} \mathcal{H}_2(0,n) \tau \right)
    \exp\left(-i (-\alpha\Delta r) \mathcal{H}_2(0,n) \tau j \right)
    \exp\left(-i \alpha \widetilde{R} \mathcal{H}_3(0,n) \tau \right) \\
    &\exp\left(-i (-\alpha\Delta r) \mathcal{H}_3(0,n) \tau j \right)
    \exp\left(-i \alpha \widetilde{R} \mathcal{H}_1(0,n) \frac{\tau}{2} \right)
    \exp\left(-i (-\alpha\Delta r) \mathcal{H}_1(0,n) \frac{\tau}{2}j \right).
\end{aligned}
\end{equation}
The corresponding Trotter error satisfies
\begin{equation}
\begin{aligned}
    e_{j,2}
    \leq& \frac{\alpha^3 |r_j|^3 \tau^3}{12} \left\| \left[ \mathcal{H}_2+\mathcal{H}_3, \left[ \mathcal{H}_2+\mathcal{H}_3, \mathcal{H}_1 \right] \right] \right\|_{0}
    + \frac{\alpha^3 |r_j|^3 \tau^3}{24} \left\| \left[ \mathcal{H}_1, \left[ \mathcal{H}_1, \mathcal{H}_2+\mathcal{H}_3 \right] \right] \right\|_{0}
    = \frac{\alpha^3 |r_j|^3 \tau^3}{2},
\end{aligned}
\end{equation}
where the final equality uses the same norm bound $\|\cdot\|_0 = 4$ from Eqs.~\eqref{eq:norm of [H2+H3,[H2+H3,H1]]} and \eqref{eq:norm of [H1,[H1,H2+H3]]}.

Using Lemmas~\ref{lem:select oracle}, \ref{lem:operator1} and \ref{lem:operator2} and combining Eqs.~\eqref{eq:trotter of Uj under periodic}, \eqref{eq:trotter of H under periodic} and \eqref{eq:trotter of L under periodic}, we construct the quantum circuit for the select oracle under the periodic BCs as
\begin{equation}\label{eq:select under periodic}
\begin{aligned}
    \mathrm{SEL_P}(\tau)
    =& \sum_{j=0}^{M - 1}|j\rangle\langle j|\otimes U_j(\tau) \\
    \approx& \replytwo{\mathcal{U}_H^{\mathrm{P}} \left( \beta\tau,\frac{\pi}{2} \right) \cdot \mathcal{U}_{\mathrm{ph}}(\alpha\tau) \cdot
    \mathfrak{B}_{\alpha\tau/2}\left[W_1(\cdot,0)\right] \cdot
    \mathfrak{B}_{\alpha\tau}\left[ \prod_{k=2}^n W_k(\cdot,0) \right] \cdot
    \mathfrak{B}_{\alpha\tau}\left[V_n(\cdot,0)\right] \cdot
    \mathfrak{B}_{\alpha\tau/2}\left[W_1(\cdot,0)\right],}
\end{aligned}
\end{equation}
\replytwo{where
\begin{equation}\label{eq:shorthand periodic block}
    \mathcal{U}_H^{\mathrm{P}}(\theta,\lambda) := W_1\left(\frac{\theta}{2},\lambda\right)
    \left[ \prod_{k=2}^{n} W_k\left(\theta,\lambda\right) \right]
    V_n\left(\theta, \lambda\right)
    W_1\left(\frac{\theta}{2},\lambda\right).
\end{equation}}

\begin{remark}\label{remark:trotter error under periodic}
Under the periodic BCs, the total Trotter error $e_j = e_{j,1} + e_{j,2}$ from $\widetilde{U}_j(\tau)$ to $U_j(\tau)$ satisfies $e_j \leq (\alpha^3 |r_j|^3 + |\beta|^3) \tau^3/2$.
\end{remark}

Using an estimation analogous to the Robin case, we establish the circuit error bound for implementing the select oracle $\mathrm{SEL_P}(\tau)$ by the approximate quantum circuit $\widetilde{\mathrm{SEL}}_{\mathrm{P}}(\tau)$ under the periodic BCs.

\begin{theorem}\label{th:circuit error under periodic}
Under the periodic BCs, the circuit error between $\sum_{j=0}^{M-1} c_j U_j(\tau)$ and $\sum_{j=0}^{M-1} c_j \widetilde{U}_j(\tau)$ satisfies
\begin{equation}
\begin{aligned}
    \left\| \sum_{j=0}^{M-1} c_j U_j(\tau) - \sum_{j=0}^{M-1} c_j \widetilde{U}_j(\tau) \right\|
    &\leq \sum_{j=0}^{M-1} |c_j| e_j
    \replythr{\leq \frac{\tau^3}{2} \left( I_0^{(M)}(\gamma,\delta)\, |\beta|^3 + I_3^{(M)}(\gamma,\delta)\, \alpha^3 \right)}\\
    &\replythr{\leq \frac{\tau^3}{2} \left( \left(I_0(\gamma,\delta)+\varepsilon_{0,M}\right) |\beta|^3 + \left(I_3(\gamma,\delta)+\varepsilon_{3,M}\right) \alpha^3 \right)
    \leq \tau^3 \left( I_0(\gamma,\delta)\,|\beta|^3 + I_3(\gamma,\delta)\,\alpha^3 \right),}
\end{aligned}
\end{equation}
\replythr{where $I_k^{(M)}$, $I_k$ and $\varepsilon_{k,M}$ are defined in Theorem~\ref{th:circuit error under robin}, and the last inequality holds in the resolved quadrature regime stated there.}
\end{theorem}
\begin{proof}
    Using Remark~\ref{remark:trotter error under periodic} and the same estimations as in the proof of Theorem~\ref{th:circuit error under robin}, the result follows immediately.
\end{proof}

\section{Complexity analysis}\label{sec:complexity analysis}

\subsection{Quantum circuits for the $d$-dimensional equation}
For the $d$-dimensional advection-diffusion equation in Eq.~\eqref{eq:gov_eq}, applying the FVM to each dimension yields the global semi-discrete system as in Eq.~\eqref{eq:d-dimensional equation}, where the global coefficient matrix $A$ takes the form as in Eq.~\eqref{eq:d-dimensional matrix}.
Analogous to the $1$-dimensional case, we decompose $A$ into its real and imaginary parts:
\begin{equation}
\begin{aligned}
    L &= L_1 \otimes I^{\otimes n_2} \otimes \cdots\otimes I^{\otimes n_d} + \cdots + I^{\otimes n_1} \otimes \cdots \otimes I^{\otimes n_{d-1}} \otimes L_d + c I^{\otimes n_1} \otimes I^{\otimes n_2} \otimes \cdots\otimes I^{\otimes n_d}, \\
    H &= H_1 \otimes I^{\otimes n_2} \otimes \cdots\otimes I^{\otimes n_d} + \cdots + I^{\otimes n_1} \otimes \cdots \otimes I^{\otimes n_{d-1}} \otimes H_d,
\end{aligned}
\end{equation}
which implies that the minimum eigenvalue of $L$ is $\lambda_{\min}(L) = c + \sum_{p=1}^d \lambda_{\min}(L_p)$\replyone{, where each $\lambda_{\min}(L_p)$ is given by the case-by-case analysis in Remark~\ref{remark:shift_L} and Eq.~\eqref{eq:eigenvalue under periodic}.}
\begin{remark}
To ensure the positive semidefiniteness of $L$ required by the LCHS method, the attenuation coefficient $c$ must satisfy $c \geq -\sum_{p=1}^d \lambda_{\min}(L_p)$.
\end{remark}

Applying the LCHS method, we obtain the following approximate representation of the matrix exponential:
\begin{equation}
\begin{aligned}
    \exp(-A\tau)
    &= \frac{1}{\sqrt{2\pi}} \int_{\mathbb{R}} \replythr{\hat{\kappa}}(r) \exp(-i c r \tau) \bigotimes_{p=1}^d \exp\left( -i \left( H_p + r L_p \right)  \tau \right) \, \mathrm{d}r \\
    &\approx \frac{1}{\sqrt{2\pi}} \int_{-R}^{R} \replythr{\hat{\kappa}}(r) \exp(-i c r \tau) \bigotimes_{p=1}^d \exp\left( -i \left( H_p + r L_p \right)  \tau \right) \, \mathrm{d}r \\
    &\approx \sum_{j=0}^{M-1} c_j \exp(-i c r_j \tau) \bigotimes_{p=1}^d \exp\left( -i \left( H_p + r_j L_p \right)  \tau \right) \\
    &= \sum_{j=0}^{M-1} c_j \exp(-i c r_j \tau) U_j(\tau),
\end{aligned}
\end{equation}
with the error bounded by
\begin{equation}
    \left\| \exp(-A\tau) - \sum_{j=0}^{M-1} c_j e^{-i c r_j \tau} U_j(\tau) \right\|
    \leq \varepsilon_{\mathrm{lchs}} + \varepsilon_{\mathrm{quad}}
    = \mathcal{O}\left(  e^{-R} \replytwo{+} e^{\frac{1}{2}\left\|L\right\|_{L^1} - \pi/\Delta r} \right),
\end{equation}
where
\begin{equation}
    U_j(\tau) := \bigotimes_{p=1}^d U_{p,j}(\tau), \quad
    U_{p,j}(\tau) := \exp\left( -i \left( H_p + r_j L_p \right)  \tau \right).
\end{equation}

For the phase factor, we have
\begin{equation}
    \exp\left( -i c r_j\tau \right) = \exp\left(-i c \left(j\Delta r - \widetilde{R}\right) \tau\right)
    = \exp\left(i c \widetilde{R}\tau \right)
    \exp\left(-i c \Delta r \tau j\right),
\end{equation}
which directly leads to
\begin{equation}
    \sum_{j=0}^{M - 1}|j\rangle\langle j| \otimes \exp\left( -i c r_j\tau \right)
    = G\left( c \widetilde{R}\tau\right)
    \prod_{j = 1}^m P_j\left(-2^{j-1} c \Delta r\tau\right) .
\end{equation}

For the $p$-th dimension, we construct the quantum circuit $\widetilde{\mathrm{SEL}}_p(\tau)$ that approximates the select oracle $\mathrm{SEL}_p(\tau) = \sum_{j=0}^{M-1}|j\rangle\langle j|\otimes \exp\left( -i \left( H_p + r_j L_p \right)  \tau \right)$ based on the given boundary conditions and the derived matrices $L_p, H_p$. These results yield the global quantum circuit
\begin{equation}
    \widetilde{\mathrm{SEL}}(\tau)
    = G\left( c \widetilde{R}\tau\right)
    \left[ \prod_{j = 1}^m P_j\left(-2^{j-1} c \Delta r\tau\right) \right]
    \bigotimes_{p=1}^d \widetilde{\mathrm{SEL}}_p(\tau)
\end{equation}
for the global select oracle
\begin{equation}\label{eq:global select}
    \mathrm{SEL}(\tau)
    = G\left( c \widetilde{R}\tau\right)
    \left[ \prod_{j = 1}^m P_j\left(-2^{j-1} c \Delta r\tau\right) \right]
    \bigotimes_{p=1}^d \mathrm{SEL}_p(\tau).
\end{equation}

Denote $\widetilde{U}_j(\tau)$ and $\widetilde{U}_{p,j}(\tau)$ as the second-order TS decomposition of $U_j(\tau)$ and $U_{p,j}(\tau)$.
Since $\left\| U_{p,j}(\tau) - \widetilde{U}_{p,j}(\tau) \right\| \leq \tau^3 \sum_{k=0}^3 |r_j|^k w_{p,k}$, the unitarity of $U_{p,j}(\tau)$ and $\widetilde{U}_{p,j}(\tau)$ gives
\begin{equation}
\begin{aligned}
    \left\| U_j(\tau) - \widetilde{U}_j(\tau) \right\|
    &\leq \sum_{p=1}^d \left( \prod_{q=1}^{p-1} \left\| \widetilde{U}_{q,j}(\tau) \right\| \right) \cdot \left\| U_{p,j}(\tau) - \widetilde{U}_{p,j}(\tau) \right\| \cdot \left( \prod_{q=p+1}^d \left\| U_{q,j}(\tau) \right\| \right) \\
    &= \sum_{p=1}^d \left\| U_{p,j}(\tau) - \widetilde{U}_{p,j}(\tau) \right\|
    \leq \tau^3 \sum_{k=0}^3 |r_j|^k \sum_{p=1}^d w_{p,k},
\end{aligned}
\end{equation}
which further \replythr{shows that} the circuit error between $\sum_{j=0}^{M-1} c_j U_j(\tau)$ and $\sum_{j=0}^{M-1} c_j \widetilde{U}_j(\tau)$ satisfies
\begin{equation}\label{eq:lcu error of multi dimension}
\begin{aligned}
    \left\| \sum_{j=0}^{M-1} c_j U_j(\tau) - \sum_{j=0}^{M-1} c_j \widetilde{U}_j(\tau) \right\|
    &\leq \sum_{j=0}^{M-1} |c_j| \left\| U_j(\tau) - \widetilde{U}_j(\tau) \right\|
    \replythr{\leq \tau^3 \sum_{k=0}^3 I_k^{(M)} \sum_{p=1}^d w_{p,k}}
    \replythr{\leq 2\tau^3 \sum_{k=0}^3 I_k \sum_{p=1}^d w_{p,k},}
\end{aligned}
\end{equation}
\replythr{where the last inequality holds in the resolved quadrature regime of Theorem~\ref{th:circuit error under robin}.}

Let the circuit error be bounded by $\varepsilon_{\mathrm{circ}} > 0$.
By examining the right-hand side of Eq.~\eqref{eq:lcu error of multi dimension}, we find that $\varepsilon_{\mathrm{circ}}$ is dominated by $ \gamma^2 \tau^3 \sum_{p=1}^d w_{p,3}$, where $\gamma^2 = \mathcal{O}(R/\delta)$ and $\sum_{p=1}^d w_{p,3} = \mathcal{O}(d\alpha^3) = \mathcal{O}(db^3/h^6 \replytwo{+} da^3/h^3)$.
Here, $b = \max_{p} b_p$, $a = \max_{p} |a_p|$, and $h = \min_{p} h_p$.

With $\replythr{N_t}$ \replythr{Trotter} steps for the evolution of $\widetilde{U}_j(T)$, combining Eq.~\eqref{eq:quadrature error} with the triangle inequality gives
\begin{equation}\label{eq:total error of d-dim equation}
    \left\| e^{-AT} - \sum_{j=0}^{M-1} c_j \left( \widetilde{U}_j \left( \frac{T}{\replythr{N_t}} \right) \right)^{\replythr{N_t}} \right\|
    \leq \varepsilon_{\mathrm{lchs}} + \varepsilon_{\mathrm{quad}} + \varepsilon_{\mathrm{circ}}
    = \mathcal{O} \left( e^{-R} \replytwo{+} e^{\left\|L\right\|T/2 - \pi/\Delta r} \replytwo{+} \frac{dRb^3T^3}{\delta \replythr{N_t}^2h^6} \replytwo{+} \frac{dRa^3T^3}{\delta \replythr{N_t}^2h^3} \right).
\end{equation}

This yields the following result:
\begin{theorem}\label{th:r of homo term}
To bound the total error in Eq.~\eqref{eq:total error of d-dim equation} by $\mathcal{O}(\varepsilon)$ (for $\varepsilon>0$), it suffices to choose
\begin{equation}
    \replythr{N_t} \geq \frac{d^{1/2} b^{3/2} T^{3/2} \log^{1/2}(1/\varepsilon)}{\varepsilon^{1/2} \delta^{1/2} h^3 }
\end{equation}
for advection-diffusion equations (non-vanishing diffusion), and
\begin{equation}
    \replythr{N_t} \geq \frac{d^{1/2} a^{3/2} T^{3/2} \log^{1/2}(1/\varepsilon)}{\varepsilon^{1/2} \delta^{1/2} h^{3/2} }
\end{equation}
for pure advection equations (diffusion-free case).
\end{theorem}

\begin{proof}
\replythr{Starting from Eq.~\eqref{eq:total error of d-dim equation}, the LCHS truncation and quadrature errors are budgeted separately as $\varepsilon_{\mathrm{lchs}}$ and $\varepsilon_{\mathrm{quad}}$, and it remains to control the Trotter contribution
\begin{equation*}
    \varepsilon_{\mathrm{circ}} \leq \frac{2T^3}{N_t^2} \sum_{k=0}^{3} I_k(\gamma,\delta) \sum_{p=1}^{d} w_{p,k}.
\end{equation*}
The dominant term is $k = 3$: the kernel moments defined above satisfy $I_3(\gamma,\delta) = \mathcal{O}(e^{\delta}\gamma^{2}) = \mathcal{O}(e^{\delta}R/\delta)$ since $R = 2\delta\gamma^{2}$, and the finite-volume parameters give $\sum_{p=1}^{d} w_{p,3} = \mathcal{O}(db^{3}/h^{6})$ for advection-diffusion equations and $\mathcal{O}(da^{3}/h^{3})$ in the diffusion-free case, where $b = \max_p b_p$, $a = \max_p |a_p|$, and $h = \min_p h_p$. Hence
\begin{equation*}
    \varepsilon_{\mathrm{circ}} = \mathcal{O}\left( \frac{d R b^{3} T^{3}}{\delta N_t^{2} h^{6}} \right)
    \quad \text{or} \quad
    \mathcal{O}\left( \frac{d R a^{3} T^{3}}{\delta N_t^{2} h^{3}} \right),
\end{equation*}
respectively. Requiring $\varepsilon_{\mathrm{circ}} \leq \varepsilon$, solving for $N_t$, and substituting $R = \mathcal{O}(\log(1/\varepsilon))$ yields the displayed bounds with a sufficiently large implicit constant.}
\end{proof}

\subsection{Quantum circuits for the inhomogeneous term}

For the inhomogeneous term, by applying the trapezoidal rule (Eq.~\eqref{eq:outer trapezoidal}), we obtain
\begin{equation}
    \int_0^{T} \mathcal{T} e^{-A(T - s)} \mathbf{f}(s) \, \mathrm{d}s
    \approx \sum_{k=0}^{M_o-1} \Delta t \left( \sum_{j=0}^{M-1} c_j U_{k,j}(T) \right) \mathbf{f}(T_k),
\end{equation}
where $U_{k,j}(T) = U_j(k \Delta t + \Delta t/2)$, $T_k = T - (k+1/2) \Delta t$, and $\Delta t = T/M_o$.
For each integer $k$, let $k = k_{m_o-1}\cdots k_0$ denote the binary representation of $k$\replytwo{, i.e., $k = \sum_{s=0}^{m_o-1} k_s 2^s$ with $k_s \in \{0,1\}$}.
Denoting $\widetilde{U}_{k,j}(T)$ as the approximate quantum circuit of $U_{k,j}(T)$, and recalling the construction of $\text{SEL-}U_{k,j}(T)$ (Eq.~\eqref{eq:sel-U_kj}), we then derive
\begin{equation}
     \widetilde{U}_{k,j}(T) = \left( \prod_{s=0}^{m_o-1} \widetilde{U}_j \left( k_s 2^s \Delta t \right) \right) \widetilde{U}_j \left( \frac{\Delta t}{2} \right),
\end{equation}
which implies
\begin{equation}
\begin{aligned}
    \left\| \sum_{j=0}^{M-1} c_j U_{k,j} (T) - \sum_{j=0}^{M-1} c_j \widetilde{U}_{k,j}(T) \right\|
    &\leq \sum_{j=0}^{M-1} |c_j| \left\| U_{k,j} (T) - \widetilde{U}_{k,j}(T) \right\| \\
    & = \sum_{j=0}^{M-1} |c_j| \left\| U_j \left( k \Delta t + \frac{\Delta t}{2} \right) - \left( \prod_{s=0}^{m_o-1} \widetilde{U}_j \left( k_s 2^s \Delta t \right) \right) \widetilde{U}_j \left( \frac{\Delta t}{2} \right) \right\| \\
    &\leq \frac{1}{M_o^3} \left( \sum_{s=0}^{m_o-1} k_s^3 2^{3s} + \frac{1}{2^3} \right) \sum_{j=0}^{M-1} |c_j| \left\| U_j (T) - \widetilde{U}_j (T) \right\|.
\end{aligned}
\end{equation}
This further leads to
\begin{equation}\label{eq:circuit error of inhomo term}
\begin{aligned}
    &\left\| \sum_{k=0}^{M_o-1} \Delta t \left( \sum_{j=0}^{M-1} c_j U_{k,j} (T) \right) \mathbf{f}(T_k) - \sum_{k=0}^{\replytwo{M_o}-1} \Delta t \left( \sum_{j=0}^{M-1} c_j \widetilde{U}_{k,j} (T) \right) \mathbf{f}(T_k) \right\| \\
    \leq& \sum_{k=0}^{M_o-1} \Delta t \left\| \sum_{j=0}^{M-1} c_j U_{k,j} (T) - \sum_{j=0}^{M-1} c_j \widetilde{U}_{k,j}(T) \right\| \left\| \mathbf{f}(T_k) \right\| \\
    \leq& \sum_{k=0}^{M_o-1} \Delta t \frac{1}{M_o^3} \left( \sum_{s=0}^{m_o-1} k_s^3 2^{3s} + \frac{1}{2^3} \right) \sum_{j=0}^{M-1} |c_j| \left\| U_j (T) - \widetilde{U}_j (T) \right\| \|\mathbf{f}\|_T \\
    =& \left( \frac{1}{14} + \frac{3}{56M_o^3} \right) T \|\mathbf{f}\|_T \sum_{j=0}^{M-1} |c_j| \left\| U_j (T) - \widetilde{U}_j (T) \right\|.
\end{aligned}
\end{equation}
\replytwo{Here, the last equality follows from
\begin{equation}
    \frac{\Delta t}{M_o^3} \sum_{k=0}^{M_o-1} \left( \sum_{s=0}^{m_o-1} k_s^3 2^{3s} + \frac{1}{2^3} \right)
    = \frac{\Delta t}{M_o^3} \left( \frac{M_o}{2} \sum_{s=0}^{m_o-1} 8^s + \frac{M_o}{8} \right)
    = \frac{\Delta t}{M_o^3} \left( \frac{M_o(M_o^3-1)}{14} + \frac{M_o}{8} \right)
    = T \left( \frac{1}{14} + \frac{3}{56M_o^3} \right),
\end{equation}
where we have used $k_s^3 = k_s$ for $k_s \in \{0,1\}$ and each bit $k_s$ equals one for exactly $M_o/2$ values of $k$.}
By combining Lemma~\ref{lem:outer quadrature} with Eqs.~\eqref{eq:quadrature error} and \eqref{eq:circuit error of inhomo term}, we derive
\begin{equation}\label{eq:total error of inhomo d-dim equation}
\begin{aligned}
    &\left\| \int_0^{T} e^{-A(T-s)} \mathbf{f}(s) \, \mathrm{d}s - \sum_{k=0}^{M_o-1} \Delta t \left( \sum_{j=0}^{M-1} c_j \left( \widetilde{U}_{k,j} \left( \frac{T}{\replythr{N_t}} \right) \right)^\replythr{N_t} \right) \mathbf{f}(T_k) \right\| \\
    \leq& \varepsilon_{\mathrm{quad,o}}
    + \varepsilon_{\mathrm{lchs}} T \|\mathbf{f}\|_T
    + \frac{64 T \|\mathbf{f}\|_T}{15M_o} \frac{e^{\|L\|T/2}-1}{e^{\|L\|T/2M_o}-1} e^{3\delta/2 +\|L\|T/4M_o - \pi/\Delta r} \\
    &+ \frac{1}{\replythr{N_t}^2} \left( \frac{1}{14} + \frac{3}{56M_o^3} \right) T \|\mathbf{f}\|_T \sum_{j=0}^{M-1} |c_j| \left\| U_j (T) - \widetilde{U}_j (T) \right\| \\
    \leq& \varepsilon_{\mathrm{quad,o}} + \varepsilon_{\mathrm{lchs}} T \|\mathbf{f}\|_T
    + \frac{64 T \|\mathbf{f}\|_T}{15M_o} \frac{e^{\|L\|T/2}-1}{e^{\|L\|T/2M_o}-1} e^{3\delta/2 +\|L\|T/4M_o - \pi/\Delta r}
    + \left( \frac{1}{14} + \frac{3}{56M_o^3} \right) \varepsilon_{\mathrm{circ}} T \|\mathbf{f}\|_T \\
    =& \mathcal{O} \left( \varepsilon_{\mathrm{quad,o}} \replytwo{+} e^{-R} T \|\mathbf{f}\|_T \replytwo{+} \frac{T \|\mathbf{f}\|_T}{M_o} \frac{e^{\|L\|T/2}-1}{e^{\|L\|T/2M_o}-1} e^{3\delta/2 +\|L\|T/4M_o - \pi/\Delta r} \replytwo{+} \frac{dRb^3T^4\|\mathbf{f}\|_T}{\delta \replythr{N_t}^2h^6} \replytwo{+} \frac{dRa^3T^4\|\mathbf{f}\|_T}{\delta \replythr{N_t}^2h^3} \right),
\end{aligned}
\end{equation}
which yields the following result.
\begin{theorem}\label{th:r of inhomo term}
To bound the total error in Eq.~\eqref{eq:total error of inhomo d-dim equation} by $\varepsilon > 0$, where $R = \mathcal{O} \left( \log \frac{1}{\varepsilon} \replytwo{+} \log \frac{T \|\mathbf{f}\|_T}{\varepsilon} \right)$, it suffices to choose
\begin{equation}
    \replythr{N_t} \geq \frac{d^{1/2} R^{1/2} b^{3/2} T^2 \|\mathbf{f}\|_T^{1/2}}{2\sqrt{2} \varepsilon^{1/2} \delta^{1/2} h^3}
\end{equation}
for advection-diffusion equations (non-vanishing diffusion), as well as
\begin{equation}
    \replythr{N_t} \geq \frac{d^{1/2} R^{1/2} a^{3/2} T^2 \|\mathbf{f}\|_T^{1/2}}{2\sqrt{2}\varepsilon^{1/2} \delta^{1/2} h^{3/2}}
\end{equation}
for pure advection equations (diffusion-free case).
\end{theorem}

\begin{proof}
\replythr{Starting from Eq.~\eqref{eq:total error of inhomo d-dim equation}, the outer-quadrature and LCHS contributions are budgeted separately, and the Trotter contribution equals the inner-circuit error multiplied by $\left(\frac{1}{14} + \frac{3}{56M_o^{3}}\right) T \|\mathbf{f}\|_T$, where $\frac{1}{14} + \frac{3}{56M_o^{3}} \leq \frac{1}{8}$. Using the dominant scaling from the proof of Theorem~\ref{th:r of homo term}, the leading term is $\mathcal{O}\left(\frac{d R b^{3} T^{4} \|\mathbf{f}\|_T}{\delta N_t^{2} h^{6}}\right)$ for advection-diffusion equations and $\mathcal{O}\left(\frac{d R a^{3} T^{4} \|\mathbf{f}\|_T}{\delta N_t^{2} h^{3}}\right)$ in the diffusion-free case. Requiring this term to be bounded by $\varepsilon$ and solving for $N_t$ yields the displayed bounds.}
\end{proof}

\subsection{Gate complexity}

As single-qubit and CNOT gates are fundamental to quantum computing, we estimate the upper bounds of gate counts for each component in the select oracle. \replythr{The counts are analytical upper bounds obtained in the logical model with arbitrary single-qubit rotations, CNOT gates, and all-to-all connectivity.} The estimation is based on the following observations:
\begin{enumerate}
    \item For a $2\times 2$ unitary matrix $U$, a $\mathrm{C}U$ gate decomposes into at most 3 single-qubit gates and 2 $\mathrm{CNOT}$ gates; a $\mathrm{C}^2U$ gate decomposes into at most 9 single-qubit gates and 8 $\mathrm{CNOT}$ gates; a $\mathrm{C}^3U$ gate decomposes into at most 21 single-qubit gates and 20 $\mathrm{CNOT}$ gates\cite{barenco1995_elementary}.
    \item A $\mathrm{C}^2\mathrm{NOT}$ gate decomposes into 9 single-qubit gates and 6 $\mathrm{CNOT}$ gates\cite{nielsen2010_quantum}.
    \item A $\mathrm{C}^j\mathrm{RZ}(\theta)$ gate ($j \geq4$) decomposes into single-qubit gates and at most $16j-24$ $\mathrm{CNOT}$ gates\cite{vale2024_circuit}.
    \item A $\mathrm{C}^j\mathrm{P}(\theta)$ gate ($j \geq4$) transforms to a $\mathrm{C}^{j+1}\mathrm{RZ}(-2\theta)$ gate acting on an additional ancilla qubit (maintained in state $\ket{0}$), hence decomposing into single-qubit gates and at most $16j-8$ $\mathrm{CNOT}$ gates\cite{rosa2025_optimizing}.
\end{enumerate}

For concreteness, we write $\varphi(j)= \mathcal{O}(j^2)$ \cite{dasilva2022_linear} for the number of single-qubit gates in the decomposition of a $\mathrm{C}^j\mathrm{RZ}$ gate.
The resulting gate count upper bounds are summarized in Tables~\ref{tab:gate cost} and \ref{tab:gate cost for select}.

\begin{table}[htbp]
\centering
\caption{\replythr{Analytical} quantum gate complexity estimation for the quantum circuit \replythr{($n\geq 5$)}.}
\label{tab:gate cost}
\begin{tabular}{ccc}
    \toprule
    Quantum circuit & Single-qubit gates & $\mathrm{CNOT}$ gates \\
    \midrule
    $W_1(\gamma \tau, \lambda)$                & $5$                & $0$          \\
    $W_2(\gamma \tau, \lambda)$                & $7$                & $4$          \\
    $W_3(\gamma \tau, \lambda)$                & $13$               & $12$         \\
    $W_4(\gamma \tau, \lambda)$                & $25$               & $26$         \\
    $W_j(\gamma \tau, \lambda)$ ($j\geq5$)     & $4 + \varphi(j-1)$ & $18j-42$     \\
    $\mathrm{C}W_1(\gamma \tau, \lambda)$      & $15$               & $10$         \\
    $\mathrm{C}W_2(\gamma \tau, \lambda)$      & $39$               & $28$         \\
    $\mathrm{C}W_3(\gamma \tau, \lambda)$      & $69$               & $52$         \\
    $\mathrm{C}W_j(\gamma \tau, \lambda)$ ($j\geq4$)  & $18j-6+\varphi(j)$ & $28j-28$     \\
    $\mathrm{C}^2W_1(\gamma \tau, \lambda)$    & $45$               & $40$         \\
    $\mathrm{C}^2W_2(\gamma \tau, \lambda)$    & $99$               & $92$         \\
    $\mathrm{C}^2W_j(\gamma \tau, \lambda)$ ($j\geq3$) & $42j-6+\varphi(j+1)$ & $56j-16$   \\
    $S_n^{(1)}(\gamma \tau)$                   & $\varphi(n)$       & $16n-24$     \\
    $\mathrm{C}S_n^{(1)}(\gamma \tau)$         & $\varphi(n+1)$     & $16n-8$      \\
    $\mathrm{C}^2S_n^{(1)}(\gamma \tau)$       & $\varphi(n+2)$     & $16n+8$      \\
    $S_n^{(0)}(\gamma \tau)$                   & $2n+\varphi(n)$    & $16n-24$     \\
    $\mathrm{C}S_n^{(0)}(\gamma \tau)$         & $\varphi(n+1)$     & $18n-8$      \\
    $\mathrm{C}^2S_n^{(0)}(\gamma \tau)$       & $18n+\varphi(n+2)$ & $28n+8$      \\
    $V_n(\gamma \tau, \lambda)$                & $2n + 2 + \varphi(n-1)$ & $18n - 42$ \\
    $\mathrm{C}V_n(\gamma \tau, \lambda)$      & $18n-6 +\varphi(n)$ & $\replythr{30n-30}$    \\
    $\mathrm{C}^2V_n(\gamma \tau, \lambda)$    & $60n-24 +\varphi(n+1)$ & $68n-28$  \\
    \bottomrule
\end{tabular}
\end{table}

\begin{table}[htbp]
\centering
\caption{\replythr{Analytical} quantum gate complexity estimation for the select oracle quantum circuit ($n\geq 5$).}
\label{tab:gate cost for select}
\begin{tabular}{ccc}
    \toprule
    Quantum circuit & Single-qubit gates & $\mathrm{CNOT}$ gates \\
    \midrule
    $\mathrm{SEL_R}(\tau)$ &
    \makecell[c]{$9mn^2+(\replythr{3m+14})n+\replythr{49}m+118$ \\
    $+(m+3)\sum_{j=4}^n\varphi(j) -\varphi(n) + 2m\varphi(n+1)$} &
    \makecell[c]{$(14m+27)n^2+(20m-67)n$ \\
    $+42$} \\
    \midrule
    $\mathrm{CSEL_R}(\tau)$ &
    \makecell[c]{$(21m+27)n^2+(33m+9)n+\replythr{141}m+147$ \\
    $+(m+3)\sum_{j=4}^n\varphi(j)+(m+2)\varphi(n+1) + 2m\varphi(n+2)$} &
    \makecell[c]{$(28m+42)n^2+(56m-8)n$ \\
    $+54m+34$} \\
    \midrule
    $\mathrm{SEL_P}(\tau)$ &
    \makecell[c]{$9mn^2+(21m+12)n+43m+83$ \\
    $+(m+2)\sum_{j=4}^n\varphi(j) + 2\varphi(n-1)+(m-2)\varphi(n)$} &
    \makecell[c]{$(14m+18)n^2+(\replythr{16m-30})n$ \\
    $\replythr{-14m}-24$} \\
    \midrule
    $\mathrm{CSEL_P}(\tau)$ &
    \makecell[c]{$(21m+18)n^2+(75m+42)n+\replythr{117}m+87$ \\
    $+(m+2) \sum_{j=4}^n \varphi(j) + 2\varphi(n) + 2m\varphi(n+1)$} &
    \makecell[c]{$(28m+28)n^2+(80m+28)n$ \\
    $+10m-22$} \\
    \bottomrule
\end{tabular}
\end{table}

\replythr{Beyond this counting model, several qualifications are in order: in a Clifford+$T$ basis, each arbitrary rotation requires synthesis with a $T$-count logarithmic in the synthesis tolerance; on hardware with restricted connectivity, routing via SWAP networks increases both count and depth; and quantum-error-correction overhead (syndrome extraction, magic-state distillation, etc.) is not included. The counts should therefore be read as logical resource estimates rather than physical ones.}

The single-qubit and $\mathrm{CNOT}$ gate complexities of the per-dimension selectors $\mathrm{SEL}(\tau)$ and $\mathrm{CSEL}(\tau)$ scale as $\mathcal{O}(mn^2)$.
Accordingly, the select oracle gate complexity is $\mathcal{O}(dmn^2)$ for a $d$-dimensional homogeneous equation, and $\mathcal{O}(dm_omn^2)$ for an inhomogeneous $d$-dimensional equation.
Here we denote $n = \max_p n_p$ for simplicity.
Combining results from Theorems~\ref{th:r of homo term} and \ref{th:r of inhomo term}, we obtain the gate complexity scalings for time evolution as follows:

\begin{theorem}\label{th:gate complexity}
To prepare an $\varepsilon$-approximation of the normalized quantum state $\ket{\mathbf{u}_1(T)}$ proportional to the homogeneous solution $\mathbf{u}_1(T)$ with $\Omega(1)$ success probability and a flag indicating success, the gate complexity for the time evolution is
\begin{equation}
    \mathcal{C}_{\mathrm{qc,homo}} = \mathcal{O}\left( \frac{e^{2\delta} \left\| \mathbf{u}(0) \right\|^2 d^{3/2} m n^2 b^{3/2} T^{3/2} \log^{1/2}(1/\varepsilon)}{\left\| \mathbf{u}_1(T) \right\|^2 \delta^{1/2} h^3 \varepsilon^{1/2} } \right)
\end{equation}
for advection-diffusion equations (non-vanishing diffusion) and
\begin{equation}
    \mathcal{C}_{\mathrm{qc,homo}} = \mathcal{O}\left( \frac{e^{2\delta} \left\| \mathbf{u}(0) \right\|^2 d^{3/2} m n^2 a^{3/2} T^{3/2} \log^{1/2}(1/\varepsilon) }{\left\| \mathbf{u}_1(T) \right\|^2 \delta^{1/2} h^{3/2} \varepsilon^{1/2} } \right)
\end{equation}
for pure advection equations (diffusion-free case).

To prepare an $\varepsilon$-approximation of the normalized quantum state $\ket{\mathbf{u}_2(T)}$ proportional to the inhomogeneous solution $\mathbf{u}_2(T)$ with $\Omega(1)$ success probability and a flag indicating success with $R = \mathcal{O} \left( \log \frac{1}{\varepsilon} \replytwo{+} \log \frac{T \|\mathbf{f}\|_T}{\varepsilon} \right)$, the gate complexity for the time evolution is
\begin{equation}
    \mathcal{C}_{\mathrm{qc,inhomo}} = \mathcal{O}\left( \frac{e^{2\delta} d^{3/2} m_o m n^2 R^{1/2} b^{3/2} T^4 \|\mathbf{f}\|_T^{5/2}}{\left\| \mathbf{u}_2(T) \right\|^2 \delta^{1/2} h^3 \varepsilon^{1/2} } \right)
\end{equation}
for advection-diffusion equations (non-vanishing diffusion) and
\begin{equation}
    \mathcal{C}_{\mathrm{qc,inhomo}} = \mathcal{O}\left( \frac{e^{2\delta} d^{3/2} m_o m n^2 R^{1/2} a^{3/2} T^4 \|\mathbf{f}\|_T^{5/2}}{\left\| \mathbf{u}_2(T) \right\|^2 \delta^{1/2} h^{3/2} \varepsilon^{1/2} } \right)
\end{equation}
for pure advection equations (diffusion-free case).
\end{theorem}

\replythr{
\begin{proof}
The select oracle $\mathrm{SEL}(T) = \sum_{j=0}^{M-1} |j\rangle\langle j| \otimes U_j(T)$ with $U_j(T) = \widetilde{U}_j(T/N_t)^{N_t}$ consists of $N_t$ controlled Trotter blocks per invocation. By the gate counts in Tables~\ref{tab:gate cost} and~\ref{tab:gate cost for select} and the per-dimension construction in Eq.~\eqref{eq:global select}, each block costs $\mathcal{O}(dmn^2)$ gates for the homogeneous term and $\mathcal{O}(dm_omn^2)$ gates for the inhomogeneous term, the latter carrying the $m_o$ additional controls of the outer LCU, so one time evolution costs $\mathcal{O}(dmn^2 N_t)$ and $\mathcal{O}(dm_omn^2 N_t)$ gates, respectively. The total error budget is split with fixed proportions, so that $\varepsilon_{\mathrm{circ}} = \Theta(\varepsilon)$, and Theorems~\ref{th:r of homo term} and~\ref{th:r of inhomo term} then provide the required $N_t$ in the four cases. Each run reports success through the ancilla measurement with probability $P_1 \gtrsim \|\mathbf{u}_1(T)\|^2/(e^{2\delta}\|\mathbf{u}(0)\|^2)$ (Eq.~\eqref{eq:probability of homo}) for the homogeneous term and $P_2 \gtrsim \|\mathbf{u}_2(T)\|^2/(e^{2\delta}T^2\|\mathbf{f}\|_T^2)$ (Eq.~\eqref{eq:probability of inhomo}) for the inhomogeneous term; repeating the flagged circuit $\mathcal{O}(1/P_1)$ or $\mathcal{O}(1/P_2)$ times raises the overall success probability to $\Omega(1)$. Multiplying the repetition factor, the step count $N_t$, and the per-block cost yields the four displayed estimates.
\end{proof}
}

\replythr{The four estimates in Theorem~\ref{th:gate complexity} admit a transparent interpretation: the cost separates into three factors, namely the per-step selector cost $\mathcal{O}(dmn^2)$ (or $\mathcal{O}(dm_omn^2)$ in the inhomogeneous case), which is only polylogarithmic in $1/h$; the Trotter step count $N_t$, which dominates the $h$-dependence through $h^{-3}$ (or $h^{-3/2}$); and the repetition factor $P^{-1}$. The norm ratio $\|\mathbf{u}(0)\|/\|\mathbf{u}_1(T)\|$ (and its inhomogeneous analogue) is a repeat-until-success factor rather than amplitude amplification: the LCU flag succeeds with probability $P_1$ bounded in Eq.~\eqref{eq:probability of homo}, and coherent amplitude amplification could improve $P_1^{-1}$ to $P_1^{-1/2}$, but the theorem reports the conservative repetition cost. The factor $e^{2\delta}$ stems from the LCU normalization $\|\mathbf{c}\|_1^2 \leq e^{2\delta}$, while $\delta^{-1/2}$ follows from $\gamma^2 = R/(2\delta)$ in the kernel moment $I_3$ after solving the $N_t^{-2}$ bound. The algebraic factor $\varepsilon^{-1/2}$ (rather than only $\log(1/\varepsilon)$) and the scaling $T^{3/2}$ (rather than $T$) both originate from the global $\mathcal{O}(N_t^{-2})$ error of the second-order product formula, while the LCHS truncation contributes only logarithmically through $R = \mathcal{O}(\log(1/\varepsilon))$ and the register size $m$. Finally, in the inhomogeneous case, $N_t$ contributes $T^2$ and the outer-LCU success bound $P_2^{-1}$ in Eq.~\eqref{eq:probability of inhomo} contributes another $T^2$, which explains the $T^4$ scaling.}

\subsection{Conditional complexity advantage}

\replythr{Throughout this subsection, the complexity comparison concerns the preparation of the solution state, or the evaluation of global observables on it, under efficient state preparation for structured data; full readout of the $2^{nd}$ solution entries is excluded from the comparison, as it is equally expensive for classical and quantum algorithms.}

First, we analyze the computational complexity of solving a $d$-dimensional advection-diffusion equation via classical computing.
Consider a $k$-th order explicit time-stepping scheme with an additive error bounded by $\mathcal{O}((T/\replythr{N_t^{\mathrm{cc}}})^{k+1})$, where $\replythr{N_t^{\mathrm{cc}}}$ denotes the total number of classical time steps\replytwo{ (in contrast to the number $N_t$ of Trotter steps in the quantum circuits)}.
To achieve a classical simulation up to time $T$ with additive error $\varepsilon$, we require $\replythr{N_t^{\mathrm{cc}}} = \mathcal{O}(k T^{(k+1)/k} / \varepsilon^{1/k})$.
Combining this with the Courant-Friedrichs-Lewy (CFL) constraint $\replythr{N_t^{\mathrm{cc}}} = \mathcal{O}( aT/h \replytwo{+} bT/h^2 )$, the total number of time steps satisfies
\begin{equation}
    \replythr{N_t^{\mathrm{cc}}} = \mathcal{O} \left( k T^{\frac{k+1}{k}}/ \varepsilon^{\frac{1}{k}} \replytwo{+} aT/h \replytwo{+} bT/h^2 \right).
\end{equation}

At each time step, classical computing requires $\mathcal{O}(s2^{nd})$ arithmetic operations, where $s$ denotes the sparsity of the coefficient matrix of the ODE system ($s\leq3$ in this work). Consequently, the total classical computational complexity is
\begin{equation}
    \mathcal{C}_{\mathrm{cc}} = \mathcal{O}\left(s 2^{nd} \replythr{N_t^{\mathrm{cc}}} \right)
    = s 2^{nd} \mathcal{O} \left( k T^{\frac{k+1}{k}}/ \varepsilon^{\frac{1}{k}} \replytwo{+} aT/h \replytwo{+} bT/h^2 \right).
\end{equation}

\replytwo{For a spatial discretization of order $p$, the spatial error scales as $\mathcal{O}(h^p)$ under standard regularity assumptions, and balancing it against the spatial error tolerance $\varepsilon_{\mathrm{sp}} = \Theta(\varepsilon)$ gives $h = \Theta(\varepsilon^{1/p})$. Combined with $h = \Theta(2^{-n})$ (as $N = 2^n$),} we can derive the quantum speedup for the time evolution of the homogeneous term.

For advection-diffusion equations with non-vanishing diffusion, the speedup is given by
\begin{equation}
    \mathcal{S}
    = \frac{\mathcal{C}_{\mathrm{cc}}}{\mathcal{C}_{\mathrm{qc,homo}}}
    = \frac{\left\| \mathbf{u}_1(T) \right\|^2 s \delta^{1/2}}{e^{2\delta} \left\| \mathbf{u}(0) \right\|^2 d^{3/2}mn^2} \mathcal{O} \left( \frac{ k \varepsilon^{\replytwo{(3-d)/p+1/2-1/k}} }{b^{3/2} T^{1/2-1/k} \log^{1/2}(1/\varepsilon)} \replytwo{+} \frac{\varepsilon^{\replytwo{(1-d)/p+1/2}} }{b^{1/2} T^{1/2} \log^{1/2}(1/\varepsilon)} \right),
\end{equation}
\replytwo{which yields a theoretical speedup whenever $(1-d)/p + 1/2 < 0$ or $(3-d)/p + 1/2 - 1/k < 0$; in particular, for the second-order schemes ($p = 2$) adopted in this work, the former condition holds for $d \geq 3$.}

For pure advection equations (diffusion-free), the speedup reads
\begin{equation}
    \mathcal{S}
    = \frac{\mathcal{C}_{\mathrm{cc}}}{\mathcal{C}_{\mathrm{qc,homo}}}
    = \frac{\left\| \mathbf{u}_1(T) \right\|^2 s \delta^{1/2}}{e^{2\delta} \left\| \mathbf{u}(0) \right\|^2 d^{3/2}mn^2} \mathcal{O} \left( \frac{ k \varepsilon^{\replytwo{(3/2-d)/p+1/2-1/k}} }{a^{3/2} T^{1/2-1/k} \log^{1/2}(1/\varepsilon)} \replytwo{+} \frac{\varepsilon^{\replytwo{(1/2-d)/p+1/2}} }{a^{1/2} T^{1/2} \log^{1/2}(1/\varepsilon)} \right),
\end{equation}
\replytwo{which yields a theoretical speedup whenever $(1/2-d)/p + 1/2 < 0$ or $(3/2-d)/p + 1/2 - 1/k < 0$; in particular, for $p \leq 2$ the former condition holds for $d \geq 2$.}

Similarly, recalling $R = \mathcal{O} \left( \log \frac{1}{\varepsilon} \replytwo{+} \log \frac{T \|\mathbf{f}\|_T}{\varepsilon} \right)$, we obtain the quantum speedup for the time evolution of the inhomogeneous term.

For advection-diffusion equations with non-vanishing diffusion, we have
\begin{equation}
    \mathcal{S}
    = \frac{\mathcal{C}_{\mathrm{cc}}}{\mathcal{C}_{\mathrm{qc,inhomo}}}
    = \frac{\left\| \mathbf{u}_2(T) \right\|^2 s \delta^{1/2}}{d^{3/2}m_o mn^2 \|\mathbf{f}\|_T^{5/2}} \mathcal{O} \left( \frac{ k \varepsilon^{\replytwo{(3-d)/p+1/2-1/k}} }{b^{3/2} T^{3-1/k} R^{1/2}} \replytwo{+} \frac{\varepsilon^{\replytwo{(1-d)/p+1/2}} }{b^{1/2} T^3 R^{1/2}} \right),
\end{equation}
\replytwo{which also yields a theoretical speedup whenever $(1-d)/p + 1/2 < 0$ or $(3-d)/p + 1/2 - 1/k < 0$; in particular, for $p = 2$ the former condition holds for $d \geq 3$.}

For pure advection equations (diffusion-free), we obtain
\begin{equation}
    \mathcal{S}
    = \frac{\mathcal{C}_{\mathrm{cc}}}{\mathcal{C}_{\mathrm{qc,inhomo}}}
    = \frac{\left\| \mathbf{u}_2(T) \right\|^2 s \delta^{1/2}}{d^{3/2}m_o mn^2 \|\mathbf{f}\|_T^{5/2}} \mathcal{O} \left( \frac{ k \varepsilon^{\replytwo{(3/2-d)/p+1/2-1/k}} }{a^{3/2} T^{3-1/k} R^{1/2}} \replytwo{+} \frac{\varepsilon^{\replytwo{(1/2-d)/p+1/2}} }{a^{1/2} T^3 R^{1/2}} \right),
\end{equation}
\replytwo{which also yields a theoretical speedup whenever $(1/2-d)/p + 1/2 < 0$ or $(3/2-d)/p + 1/2 - 1/k < 0$; in particular, for $p \leq 2$ the former condition holds for $d \geq 2$.}

These results demonstrate that quantum computing offers significant advantages over classical computing for high-dimensional problems, with the computational complexity gap widening as the dimensionality $d$ increases\replytwo{, and with the threshold dimension depending on the spatial discretization order $p$}.
\replythr{The dimensional threshold originates from the product-formula cost, which scales as $h^{-3}$ up to logarithmic factors, whereas the selector contributes only a polylogarithmic dependence on $1/h$; the classical $h^{-d}$ state-space cost is the competing factor.}

\section{Numerical experiments}\label{sec:numerical experiments}

In this section, we conduct numerical experiments to validate the quantum circuits proposed in Section~\ref{sec:quantum circuits for different boundary conditions}.
The simulations are executed on a C++-based local simulator leveraging the QPanda2 package \cite{dou2022_qpanda, qpanda2_repo}.
\replythr{Specifically, the CPUQVM state-vector backend is used: it maintains the full complex state vector and applies each elementary or controlled gate of the compiled circuit directly to the affected amplitudes in program order, without assembling the dense unitary matrix of the complete circuit; final amplitudes are read out exactly rather than estimated from finite-shot sampling.}
Moreover, to verify the circuit effectiveness, we assume an ideal noise-free environment: no decoherence in quantum states and perfect, noise-free logic gates.
\replythr{For reproducibility, the C++ source code for the simulations is available at \url{https://github.com/WangKunGroup/Quantum-circuits-ADE-LCHS}.}

For simplicity, we denote $n$ as the number of qubits encoding the system state (for multidimensional problems, $n_i$ for the $i$-th dimension), and $m$ and $m_o$ as \replythr{the numbers of ancilla qubits} for LCHS and LCU coefficients, respectively.
$T$ denotes the evolution time, $\replythr{N_t}$ \replythr{the number of Trotter steps} for the select oracle (defined in Remark~\ref{remark:time steps}), and $R$ the truncation boundary of the LCHS method.
For the LCHS framework, we adopt the parameter settings for $\replythr{\hat{\kappa}}$ as: \replytwo{$\delta = 0.4$, $\varepsilon_\text{lchs} = 0.001$, and $\gamma = \frac{1}{\delta}\sqrt{\delta + \log \frac{1 + 1/(2\pi)}{\varepsilon_{\text{lchs}}}} = 6.8261\ldots$.}

\subsection{$1$-dimensional homogeneous advection-diffusion equation}

This section considers the $1$-dimensional advection-diffusion equation given by
\begin{equation}\label{eq:1d_adv_diff}
    \left\{
    \begin{aligned}
    \frac{\partial u}{\partial t}  + a\frac{\partial u}{\partial x} + cu &= b\frac{\partial^2 u}{\partial x^2} + f(t,x), \\
    u(0, x) &= u_0(x),
    \end{aligned}
    \right.
     \quad x\in\Omega=[0,l].
\end{equation}
We set the source term $f(t,x) = 0$ and impose homogeneous boundary conditions to obtain a homogeneous ODE system after spatial discretization.
This setup allows us to evaluate the effectiveness of both the LCHS method and the proposed quantum circuits for homogeneous terms across different flux construction schemes and varying values of $n$ and $m$.
For the quantum circuits, we adopt the circuit in Eq.~\eqref{eq:inner LCHS}, with the select oracle implemented for the boundary conditions introduced in Sections~\ref{subsubsec:quantum circuit under the robin boundary conditions}, \ref{subsubsec:quantum circuit under the periodic boundary conditions} and Appendix~\ref{subsec:periodic with alpha=0}.

\begin{test}\label{experiment:diff_homo_dirichlet}
The first numerical experiment investigates the pure diffusion equation ($a = c = 0$) with Dirichlet BCs $u(t,x_L) = u(t,x_R) = 0$.
The initial condition is
\begin{equation}
    u_0(x) = \sin \frac{\pi x}{l}, \quad f(t,x)=0,
\end{equation}
yielding the exact solution
\begin{equation}
    u(t,x) = \exp\left(-\frac{b\pi^2}{l^2} t\right) \sin \frac{\pi x}{l}.
\end{equation}
\end{test}

For the numerical implementation of Experiment~\ref{experiment:diff_homo_dirichlet}, we set parameters as: $b = 1$, $l = 1$, $T = 0.01$, $\replythr{N_t} = 1$, and $R = 0.89375 \times 2^m$.
Applying the central scheme and the select oracle implemented as in Eq.~\eqref{eq:select under robin}, numerical results are presented in Fig.~\ref{fig:diff_homo_dirichlet_1d}, with relative errors summarized in Tables~\ref{tab:diff_homo_dirichlet_error_change_n} and \ref{tab:diff_homo_dirichlet_error_change_m}.

\replythr{When $m$ is fixed at 4, the relative errors increase slightly as $n$ increases. The reason is that the LCHS truncation error is independent of $n$, whereas refining the spatial grid enlarges $\|L\| = \mathcal{O}(b/h^{2})$; in this setting neither the quadrature resolution $\Delta r$ nor the Trotter step size $\tau = T/N_t$ is refined, so the inner-quadrature and Trotter contributions grow and overtake the decreasing spatial discretization error. This reflects an accuracy requirement rather than a stability condition, and it can be controlled by refining $\Delta r$ and increasing the number of Trotter steps $N_t$.}
For fixed $n=9$, the relative errors decrease significantly as $m$ increases from $3$ to $4$, and continue to decrease as $m$ increases to $5$.

\begin{figure}[htbp]
    \centering
    \subfigure[Fixed $m=4$]{\includegraphics[width=0.32\textwidth]{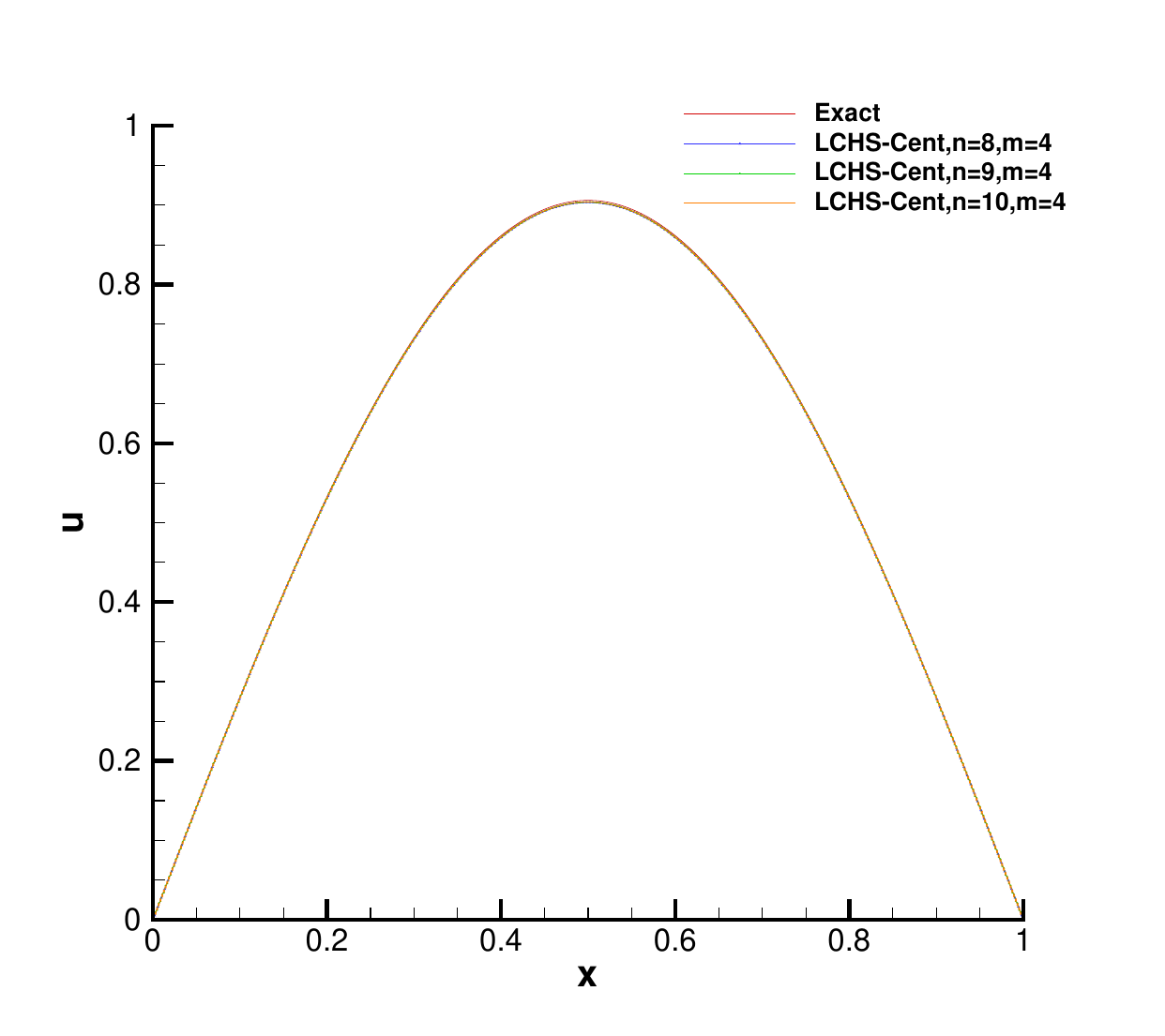}\label{fig:diff_homo_dirichlet_change_n}}
	\subfigure[Fixed $n=8$]{\includegraphics[width=0.32\textwidth]{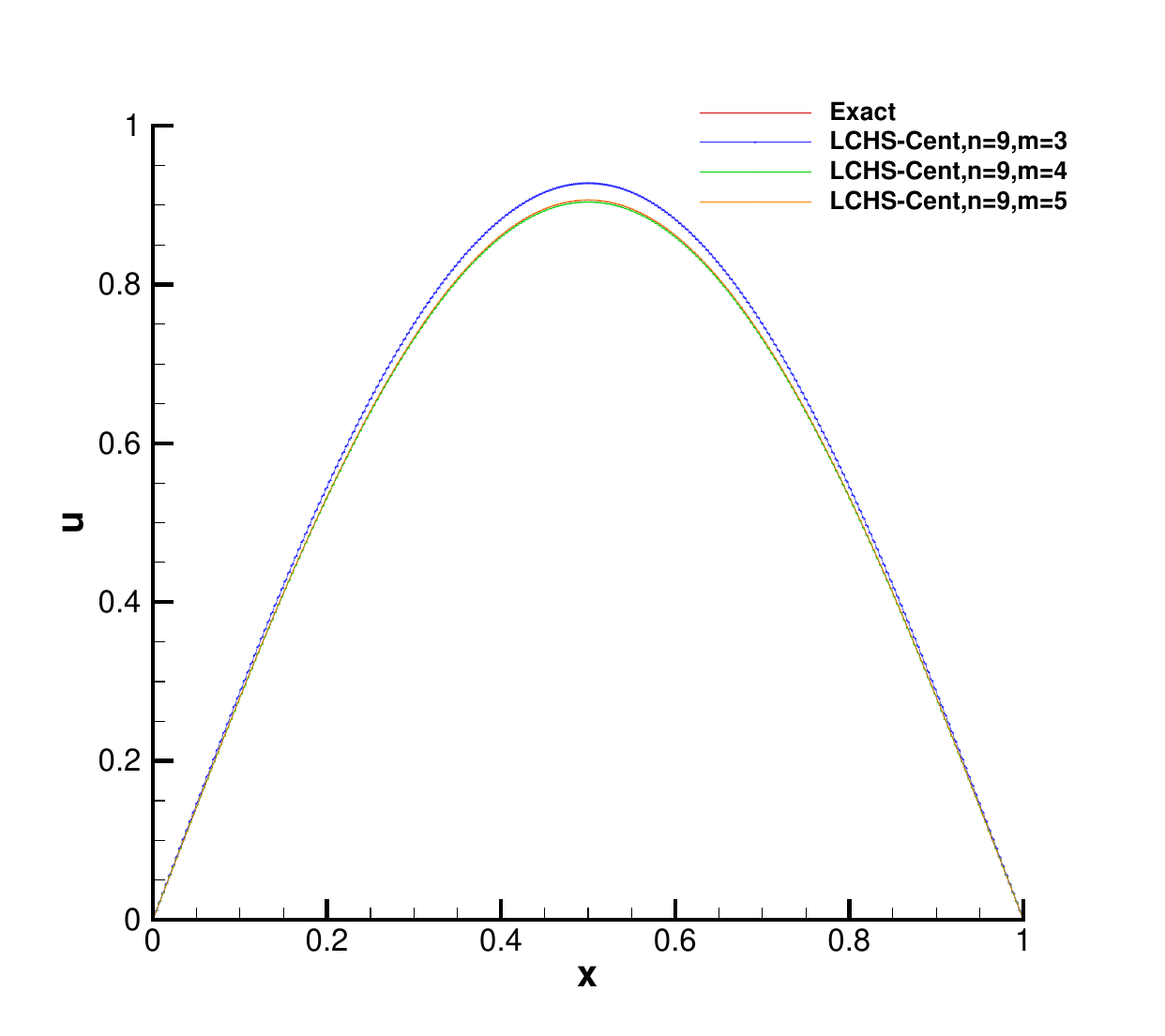}\label{fig:diff_homo_dirichlet_change_m}}
	\caption{Numerical results for Experiment~\ref{experiment:diff_homo_dirichlet} (homogeneous $1$D diffusion equation with Dirichlet BCs).}
	\label{fig:diff_homo_dirichlet_1d}
\end{figure}

\begin{table}[htbp]
\centering
\begin{minipage}[t]{0.48\textwidth}
    \centering
    \caption{Relative errors for Experiment~\ref{experiment:diff_homo_dirichlet} (fixed $m=4$)}
    \begin{tabular}{c c c c}
        \toprule
        $n$ & $L_1$-norm & $L_2$-norm & $L_\infty$-norm \\
        \midrule
        8 & 2.6835e-3 & 2.6835e-3 & 2.6835e-3 \\
        9 & 2.9274e-3 & 2.9274e-3 & 2.9274e-3 \\
        10 & 3.6440e-3 & 3.6440e-3 & 3.6439e-3 \\
        \bottomrule
    \end{tabular}
    \label{tab:diff_homo_dirichlet_error_change_n}
\end{minipage}
\hfill
\begin{minipage}[t]{0.48\textwidth}
    \centering
    \caption{Relative errors for Experiment~\ref{experiment:diff_homo_dirichlet} (fixed $n=9$)}
    \begin{tabular}{c c c c}
        \toprule
        $m$ & $L_1$-norm & $L_2$-norm & $L_\infty$-norm \\
        \midrule
        3 & 2.3781e-2 & 2.3782e-2 & 2.3782e-2 \\
        4 & 2.9274e-3 & 2.9274e-3 & 2.9274e-3 \\
        5 & 1.9847e-3 & 1.9847e-3 & 1.9847e-3 \\
        \bottomrule
    \end{tabular}
    \label{tab:diff_homo_dirichlet_error_change_m}
\end{minipage}
\end{table}

\begin{test}\label{experiment:diff_homo_neumann}
This experiment examines the pure diffusion equation ($a = c = 0$) subject to Neumann  BCs \replythr{$\frac{\partial u}{\partial x}(t, x_L) = \frac{\partial u}{\partial x}(t, x_R) = 0$}.
The initial condition is
\begin{equation}
    u_0(x) = \cos \frac{\pi x}{l}, \quad f(t,x)=0,
\end{equation}
admitting the exact solution
\begin{equation}
    u(t,x) = \exp\left(-\frac{b\pi^2}{l^2} t\right) \cos \frac{\pi x}{l}.
\end{equation}
\end{test}

For the numerical implementation of Experiment~\ref{experiment:diff_homo_neumann}, we adopt parameter as: $b = 1$, $l = 1$, $T = 0.01$, $\replythr{N_t} = 1$, and $R = 0.89375 \times 2^m$.
Employing the central scheme and the select oracle implemented as in Eq.~\eqref{eq:select under robin}, numerical results are reported in Fig.~\ref{fig:diff_homo_neumann_1d}, with relative errors compiled in Tables~\ref{tab:diff_homo_neumann_error_change_n} and \ref{tab:diff_homo_neumann_error_change_m}.

Consistent with Experiment~\ref{experiment:diff_homo_dirichlet}, the relative errors increase slightly as $n$ increases, and decrease with $m$ for fixed $n=8$.
We further observe that the relative errors exactly match those in Tables~\ref{tab:diff_homo_dirichlet_error_change_n} and \ref{tab:diff_homo_dirichlet_error_change_m}. This consistency arises because: (1) $s_0$ and $s_1$ are additive inverses across the two experiments, yielding identical error estimations by Theorem~\ref{th:circuit error under robin}; (2) the initial values correspond to the single-modality solutions of each experiment.

\begin{figure}[htbp]
    \centering
    \subfigure[Fixed $m=4$]{\includegraphics[width=0.32\textwidth]{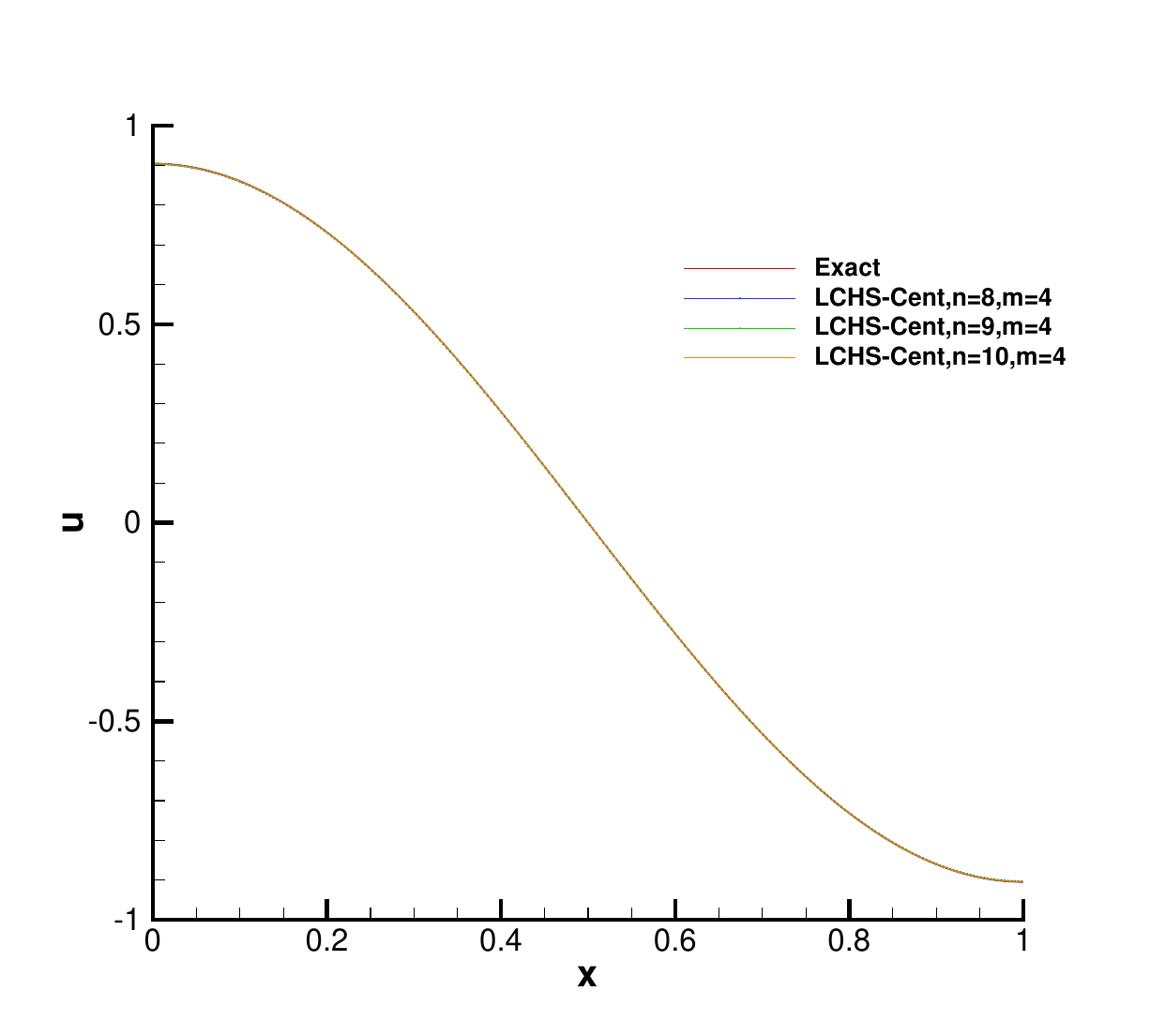}\label{fig:diff_homo_neumann_change_n}}
	\subfigure[Fixed $n=8$]{\includegraphics[width=0.32\textwidth]{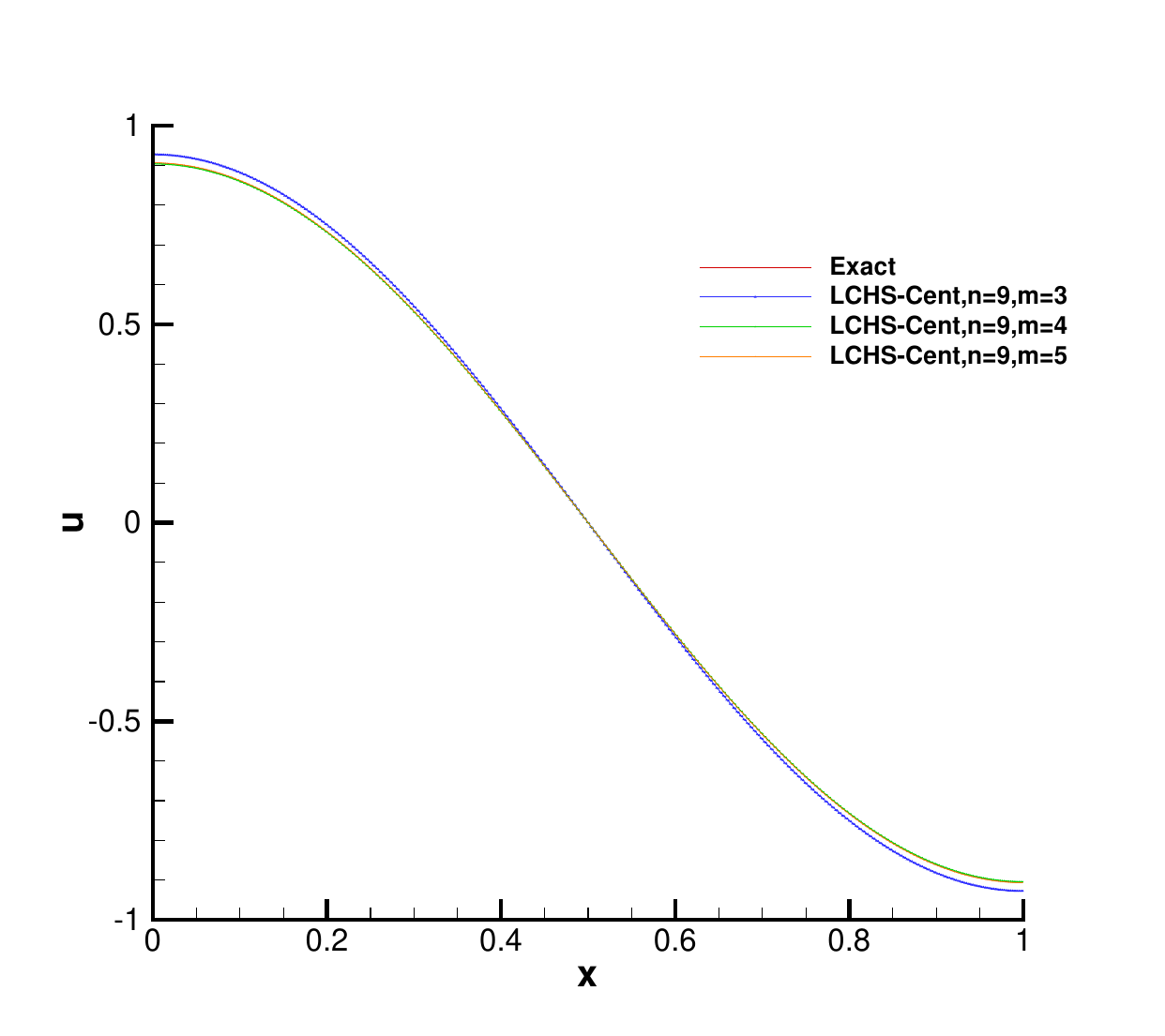}\label{fig:diff_homo_neumann_change_m}}
	\caption{Numerical results for Experiment~\ref{experiment:diff_homo_neumann} (homogeneous $1$D diffusion equation with Neumann  BCs).}
    \label{fig:diff_homo_neumann_1d}
\end{figure}

\begin{table}[htbp]
\centering
\begin{minipage}[t]{0.48\textwidth}
    \centering
    \caption{Relative errors for Experiment~\ref{experiment:diff_homo_neumann} (fixed $m=4$)}
    \begin{tabular}{c c c c}
        \toprule
        $n$ & $L_1$-norm & $L_2$-norm & $L_\infty$-norm \\
        \midrule
        8 & 2.6835e-3 & 2.6835e-3 & 2.6835e-3 \\
        9 & 2.9274e-3 & 2.9274e-3 & 2.9274e-3 \\
        10 & 3.6440e-3 & 3.6440e-3 & 3.6439e-3 \\
        \bottomrule
    \end{tabular}
    \label{tab:diff_homo_neumann_error_change_n}
  \end{minipage}
  \hfill
  \begin{minipage}[t]{0.48\textwidth}
    \centering
    \caption{Relative errors for Experiment~\ref{experiment:diff_homo_neumann} (fixed $n=8$)}
    \begin{tabular}{c c c c}
        \toprule
        $m$ & $L_1$-norm & $L_2$-norm & $L_\infty$-norm \\
        \midrule
        3 & 2.3781e-2 & 2.3782e-2 & 2.3782e-2 \\
        4 & 2.9274e-3 & 2.9274e-3 & 2.9274e-3 \\
        5 & 1.9847e-3 & 1.9847e-3 & 1.9847e-3 \\
        \bottomrule
    \end{tabular}
    \label{tab:diff_homo_neumann_error_change_m}
  \end{minipage}
\end{table}

\begin{test}\label{experiment:adv_homo_periodic}
This experiment considers the pure advection equation ($b = c = 0$) subject to periodic BCs $u(t,x_L) = u(t,x_R)$ and \replythr{$\frac{\partial u}{\partial x}(t,x_L) = \frac{\partial u}{\partial x}(t,x_R)$}.
The initial condition is specified as
\begin{equation}
    u_0(x) = \sin\frac{2\pi x}{l}, \quad f(t,x)=0,
\end{equation}
and the corresponding exact solution is
\begin{equation}
    u(t,x) = \sin \frac{2\pi (x-at)}{l}.
\end{equation}
\end{test}

For the numerical implementation of Experiment~\ref{experiment:adv_homo_periodic}, we set parameters as $a = 1$, $l = 1$, $T = 0.01$, and $\replythr{N_t} = 8$.
For the central scheme, the resulting ODE system is a Schr\"{o}dinger equation, eliminating the need for the LCHS method.
Using the "select" oracle implemented as in Eq.~\eqref{eq:select under periodic with alpha=0}, we present the numerical results in Fig.~\ref{fig:adv_homo_periodic_central_change_n}, with relative errors summarized in Table~\ref{tab:adv_homo_periodic_central_error_change_n}.
For the upwind scheme, we set $R = 1.6\times 2^m$ and adopt the select oracle implemented as in Eq.~\eqref{eq:select under periodic}, with numerical results plotted in Figs.~\ref{fig:adv_homo_periodic_upwind_change_n} and \ref{fig:adv_homo_periodic_upwind_change_m} and relative errors documented in Tables~\ref{tab:adv_homo_periodic_upwind_error_change_n} and \ref{tab:adv_homo_periodic_upwind_error_change_m}.

Evidently, the central scheme outperforms the upwind scheme due to its pure quantum simulation of the Schr\"{o}dinger equation.
Moreover, as $n$ increases, the relative errors of the central scheme grow at nearly $\mathcal{O}(h^{-2})$, milder than the $\mathcal{O}(h^{-3})$ bound estimated in Appendix~\ref{subsec:periodic with alpha=0}.
For the upwind scheme, the relative errors also increase with rising $n$; they decrease significantly as $m$ rises from $3$ to $4$, but increase slightly as $m$ increases from $4$ to $5$\replythr{, reflecting a balance between the reduced LCHS truncation error from the enlarged radius $R=\mathcal{O}(2^m)$ and the growing Trotter contribution from large-$|r_j|$ quadrature nodes}.

\begin{figure}[htbp]
    \centering
    \subfigure[Central scheme]{\includegraphics[width=0.32\textwidth]{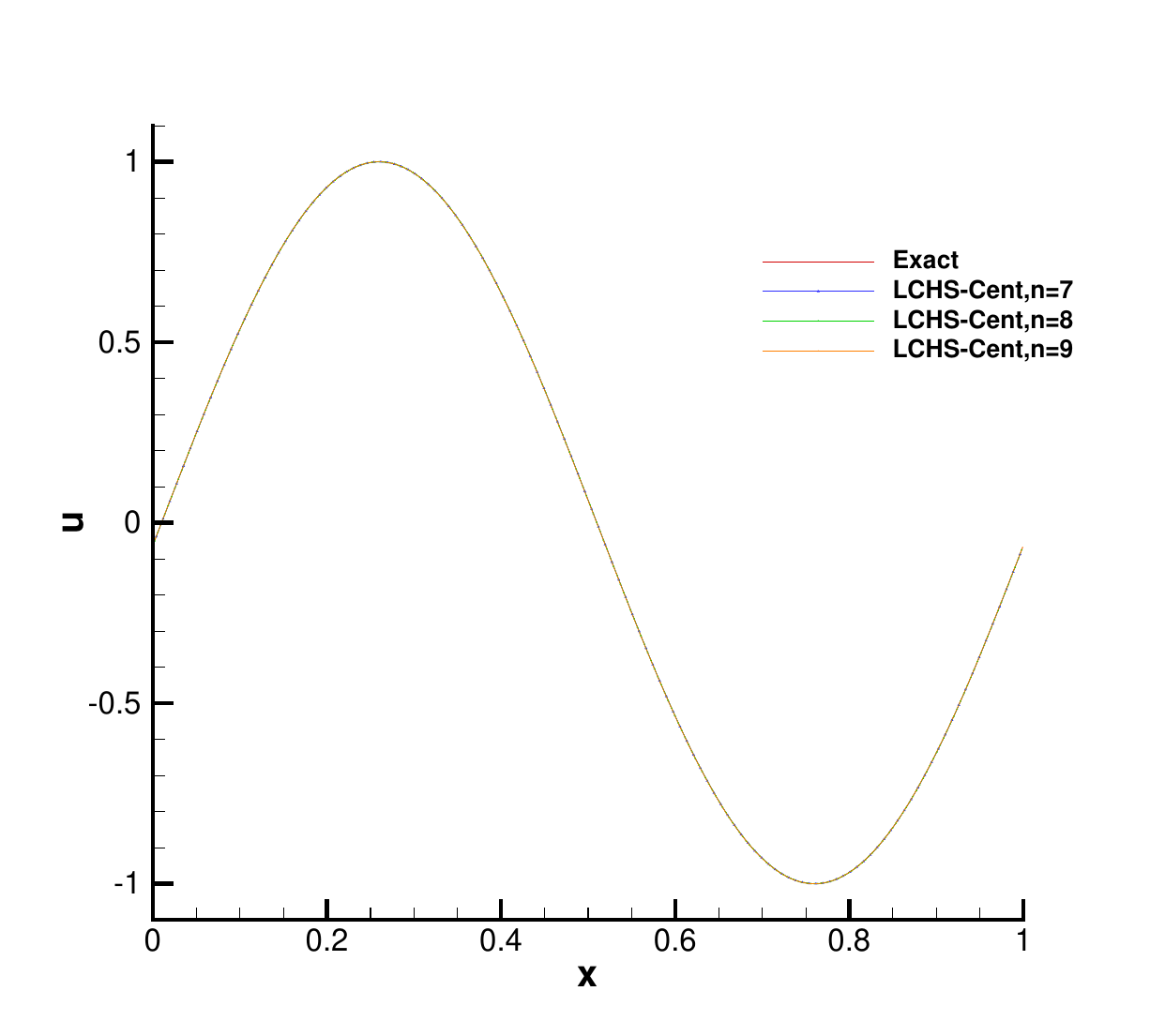}\label{fig:adv_homo_periodic_central_change_n}}
	\subfigure[Upwind scheme, fixed $m=4$]{\includegraphics[width=0.32\textwidth]{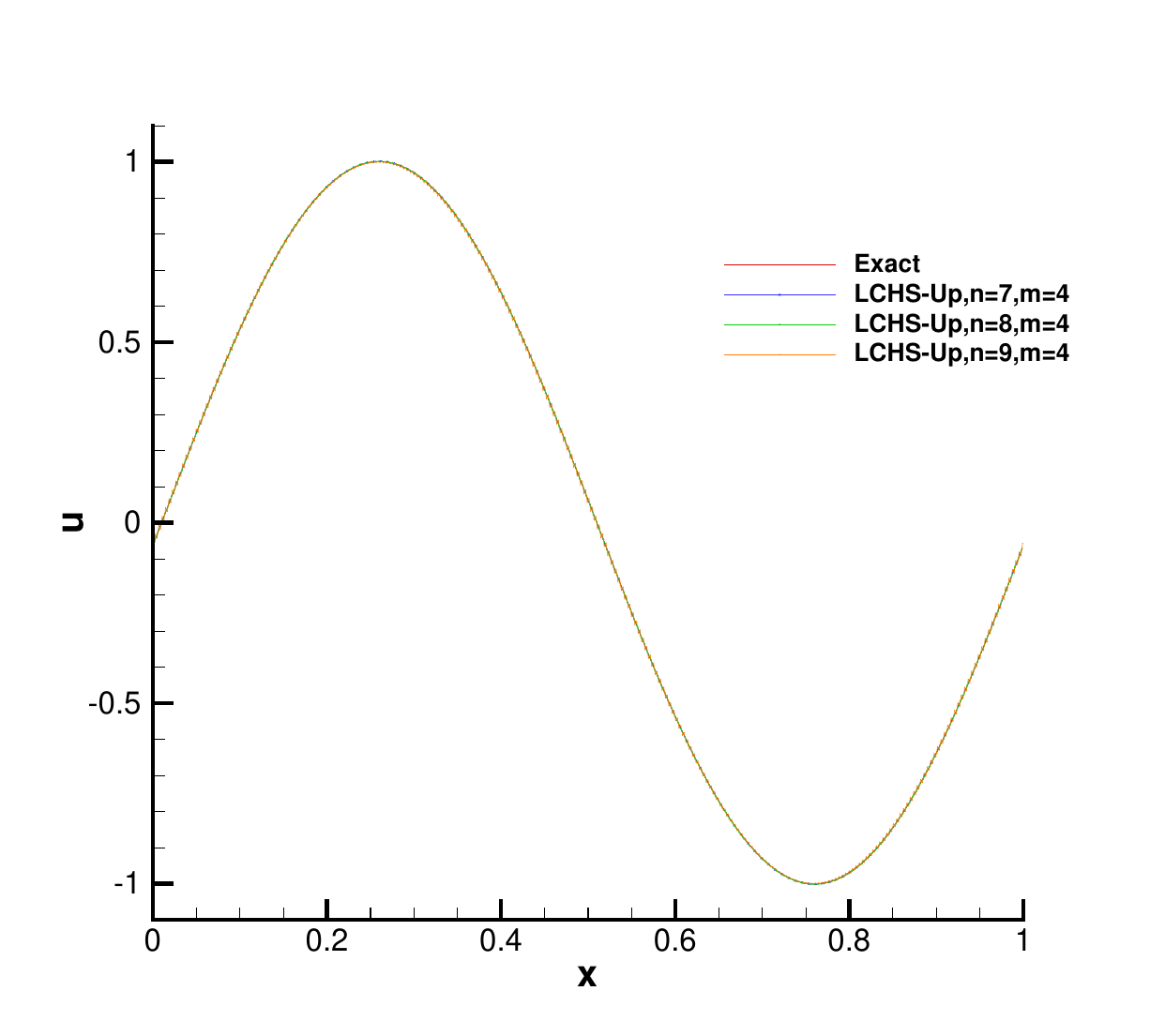}\label{fig:adv_homo_periodic_upwind_change_n}}
    \subfigure[Upwind scheme, fixed $n=8$]{\includegraphics[width=0.32\textwidth]{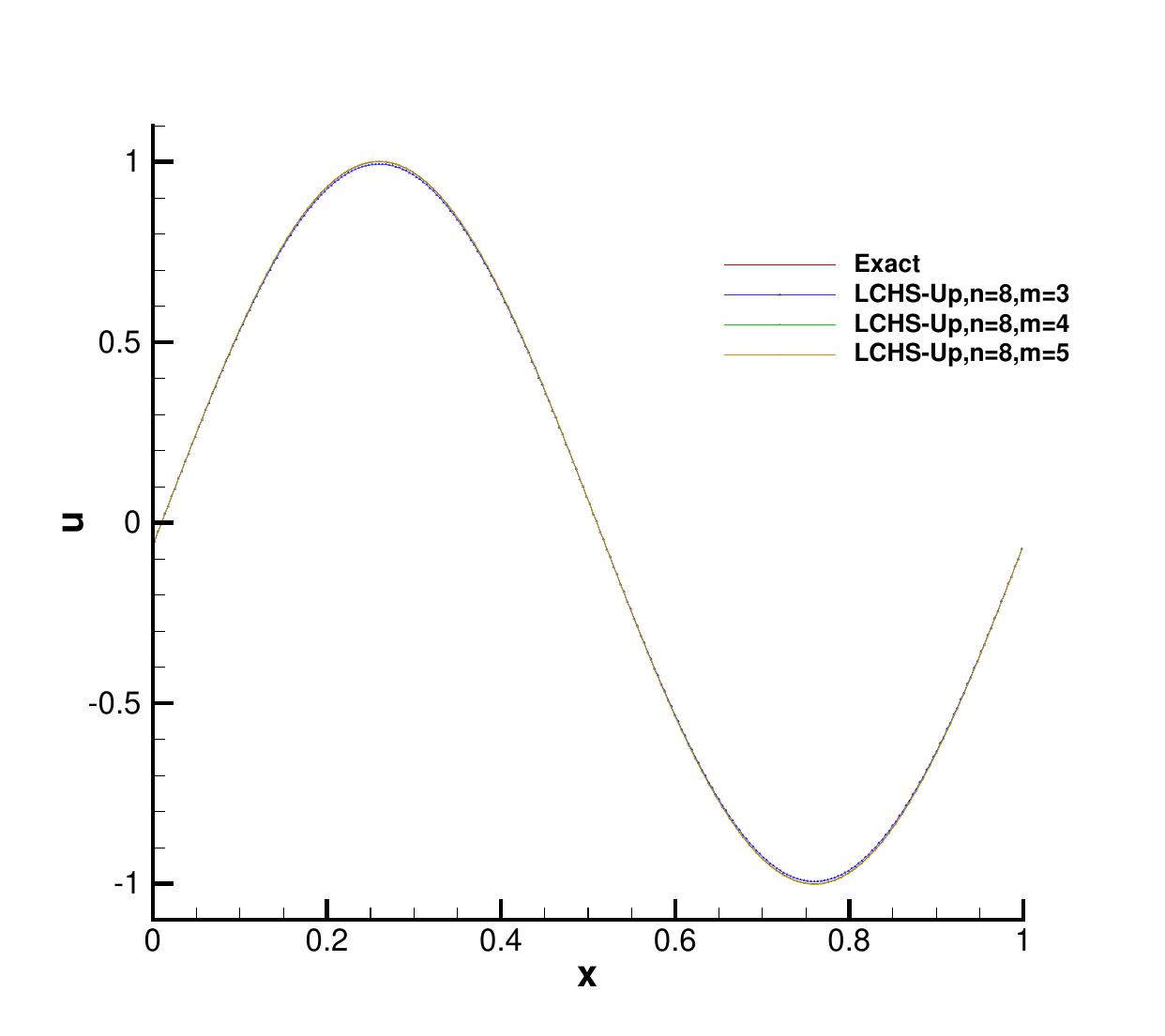}\label{fig:adv_homo_periodic_upwind_change_m}}
	\caption{Numerical results for Experiment~\ref{experiment:adv_homo_periodic} (homogeneous $1$D advection equation with periodic BCs).}
\end{figure}

\begin{table}[htbp]
\centering
\begin{minipage}[t]{0.48\textwidth}
    \centering
    \caption{Relative errors for Experiment~\ref{experiment:adv_homo_periodic} (central scheme)}
    \begin{tabular}{c c c c}
        \toprule
        $n$ & $L_1$-norm & $L_2$-norm & $L_\infty$-norm \\
        \midrule
        7 & 9.2580e-5 & 9.2565e-5 & 9.3085e-5 \\
        8 & 2.7457e-4 & 2.7456e-4 & 2.7571e-4 \\
        9 & 1.0690e-3 & 1.0690e-3 & 1.0713e-3 \\
        \bottomrule
    \end{tabular}
    \label{tab:adv_homo_periodic_central_error_change_n}
\end{minipage}
\end{table}

\begin{table}[htbp]
\centering
\begin{minipage}[t]{0.48\textwidth}
    \centering
    \caption{Relative errors for Experiment~\ref{experiment:adv_homo_periodic} (upwind scheme, fixed $m=4$)}
    \begin{tabular}{c c c c}
        \toprule
        $n$ & $L_1$-norm & $L_2$-norm & $L_\infty$-norm \\
        \midrule
        7 & 1.3747e-3 & 1.3769e-3 & 1.4580e-3 \\
        8 & 2.4183e-3 & 2.4230e-3 & 2.5701e-3 \\
        9 & 1.1064e-2 & 1.1113e-2 & 1.2106e-2 \\
        \bottomrule
    \end{tabular}
    \label{tab:adv_homo_periodic_upwind_error_change_n}
\end{minipage}
\hfill
\begin{minipage}[t]{0.48\textwidth}
    \centering
    \caption{Relative errors for Experiment~\ref{experiment:adv_homo_periodic} (upwind scheme, fixed $n=8$)}
    \begin{tabular}{c c c c}
        \toprule
        $m$ & $L_1$-norm & $L_2$-norm & $L_\infty$-norm \\
        \midrule
        3 & 6.6498e-3 & 6.6541e-3 & 6.9051e-3 \\
        4 & 2.4183e-3 & 2.4230e-3 & 2.5701e-3 \\
        5 & 2.4779e-3 & 2.4818e-3 & 2.6182e-3 \\
        \bottomrule
    \end{tabular}
    \label{tab:adv_homo_periodic_upwind_error_change_m}
\end{minipage}
\end{table}

\begin{test}\label{experiment:adv_diff_homo_dirichlet}
This numerical experiment investigates the advection-diffusion equation with Dirichlet BCs $u(t,x_L) = u(t,x_R) = 0$, where $c = 0$.
The initial condition is given by
\begin{equation}
    u_0(x) = \exp \left( \frac{a}{2b}x \right) \sin \frac{\pi x}{l}, \quad f(t,x) = 0.
\end{equation}
This initial-boundary value problem yields the exact solution
\begin{equation}
    u(t,x) = \exp \left( \frac{a}{2b}x - \left( \frac{a^2}{4b} + \frac{\pi^2 b}{l^2} \right)t \right) \sin \frac{\pi x}{l}.
\end{equation}
\end{test}

For the numerical implementation of Experiment~\ref{experiment:adv_diff_homo_dirichlet}, we set the parameters as follows: $a=1$, $b=0.5$, $l=1$, $T=0.01$, $\replythr{N_t}=1$, and $R=0.7625 \times 2^m$, and adopt the select oracle implemented as in Eq.~\eqref{eq:select under robin}.
For the central scheme, numerical results are displayed in Fig.~\ref{fig:adv_diff_homo_dirichlet_central_1d}, with relative errors reported in Tables~\ref{tab:adv_diff_homo_dirichlet_central_error_change_n} and \ref{tab:adv_diff_homo_dirichlet_central_error_change_m}.
For the exponential scheme, numerical results are shown in Fig.~\ref{fig:adv_diff_homo_dirichlet_exp_1d}, with relative errors compiled in Tables~\ref{tab:adv_diff_homo_dirichlet_exp_error_change_n} and \ref{tab:adv_diff_homo_dirichlet_exp_error_change_m}.

The relative errors of the two schemes exhibit no significant differences, as both belong to second-order spatial discretization schemes ($\mathcal{O}(h^2)$).
As $n$ increases, the relative errors decrease, as the spatial discretization error still dominates the total error.
The relative errors decrease significantly as $m$ increases from $3$ to $4$ and continue to decrease up to $m=5$.

\begin{figure}[htbp]
    \centering
    \subfigure[Fixed $m=4$]{\includegraphics[width=0.32\textwidth]{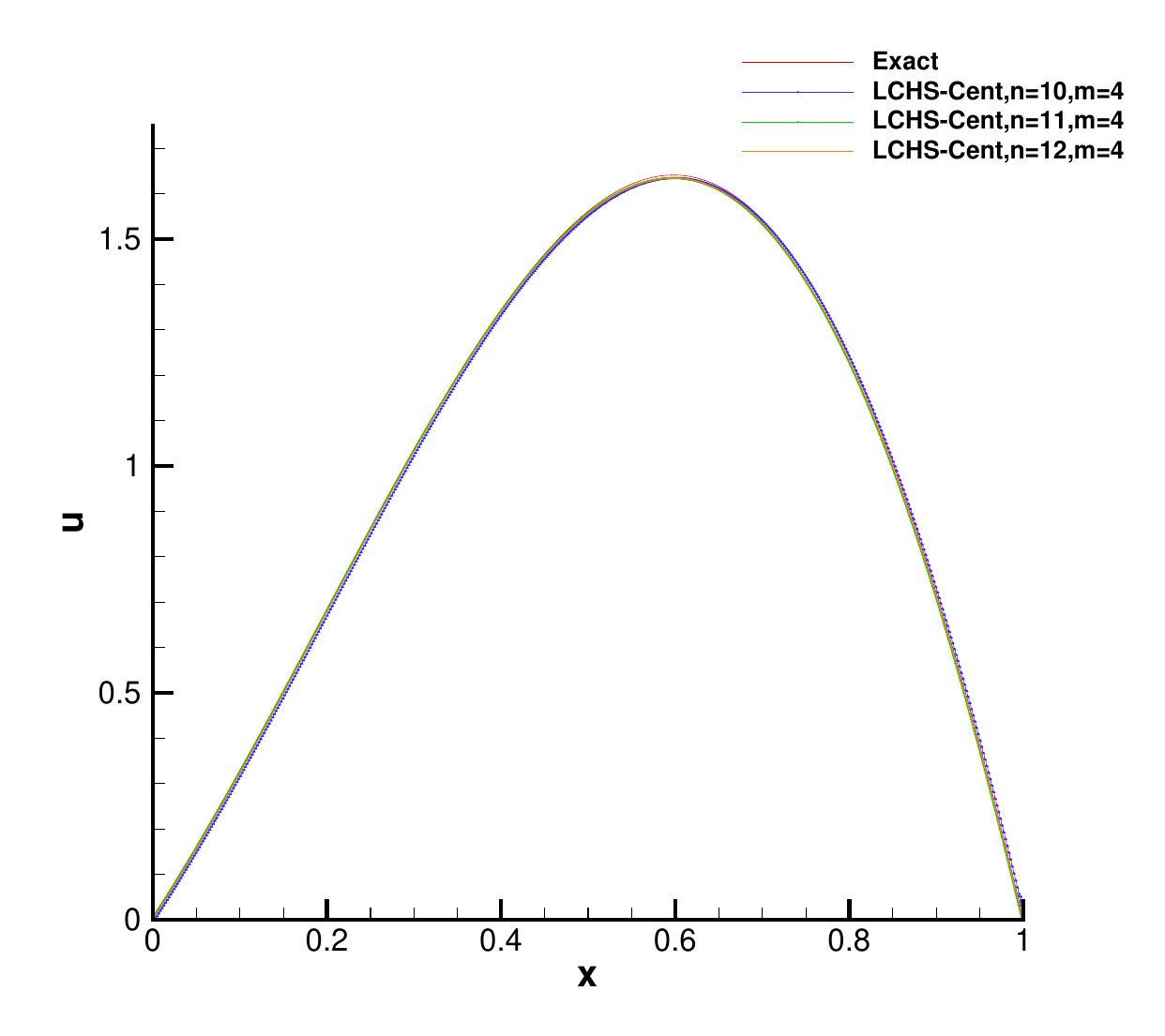}\label{fig:adv_diff_homo_dirichlet_central_change_n}}
	\subfigure[Fixed $n=11$]{\includegraphics[width=0.32\textwidth]{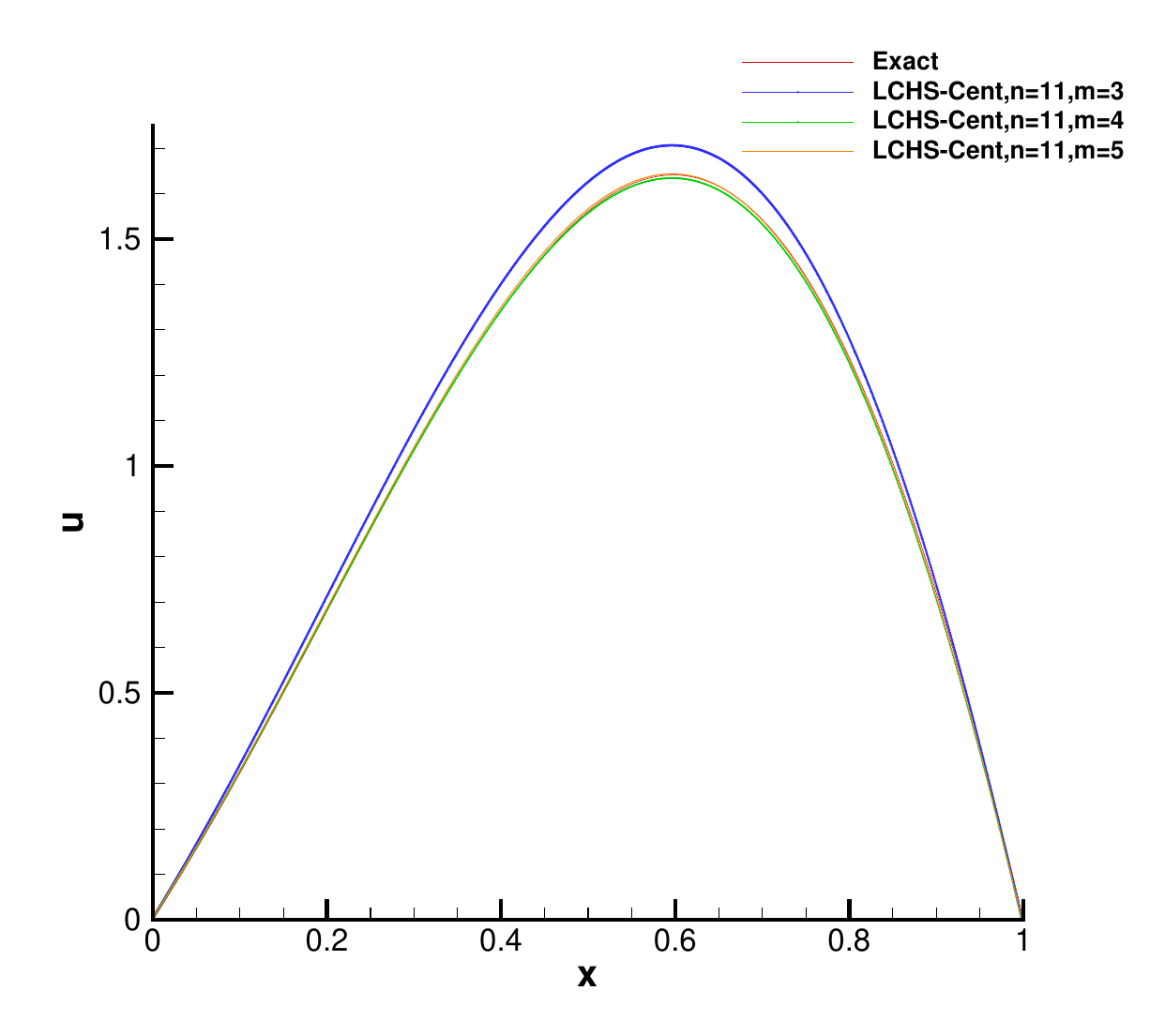}\label{fig:adv_diff_homo_dirichlet_central_change_m}}
	\caption{Numerical results for Experiment~\ref{experiment:adv_diff_homo_dirichlet} (homogeneous $1$D advection-diffusion equation with Dirichlet BCs, central scheme).}
    \label{fig:adv_diff_homo_dirichlet_central_1d}
\end{figure}

\begin{figure}[htbp]
    \centering
    \subfigure[Fixed $m=4$]{\includegraphics[width=0.32\textwidth]{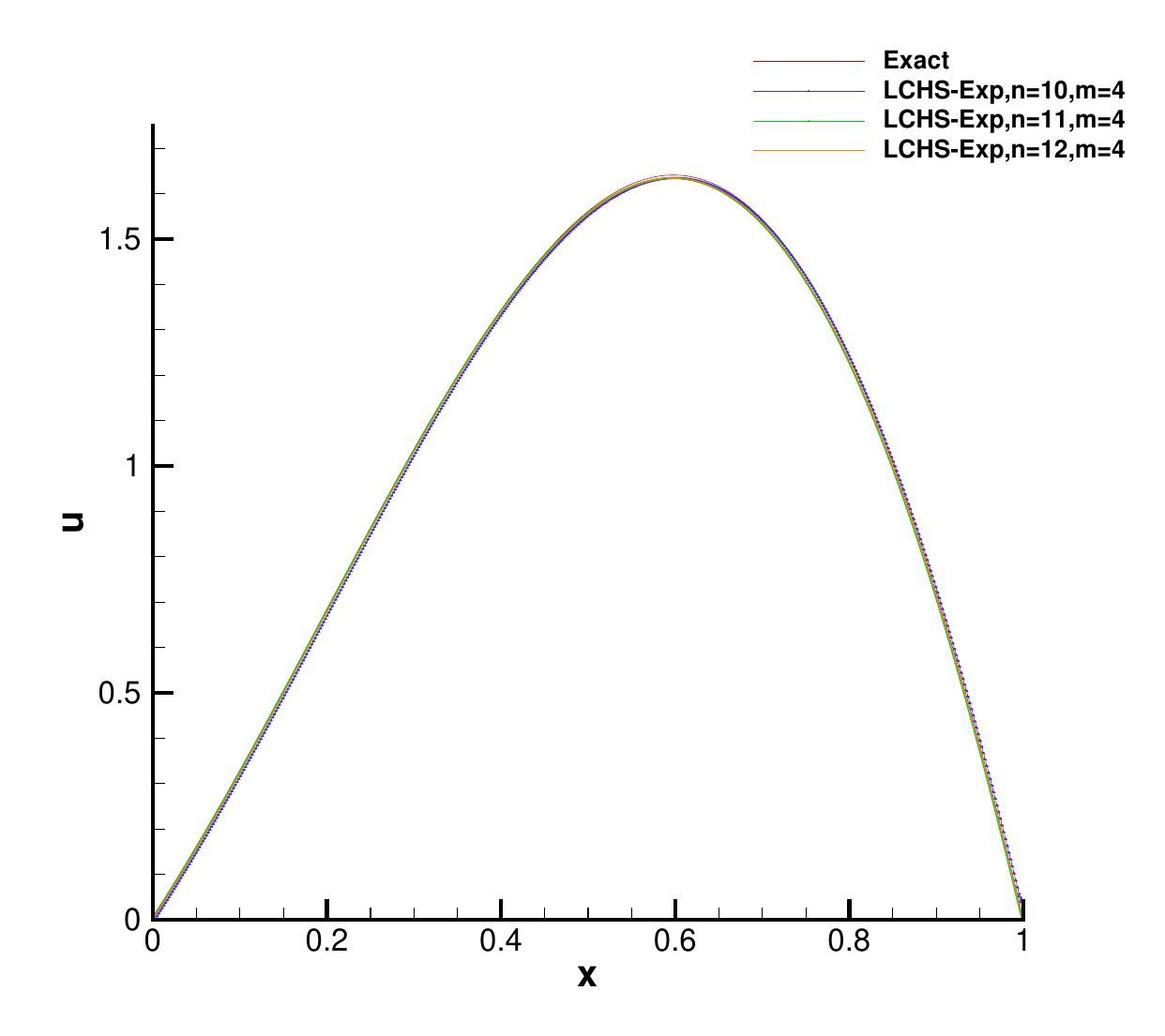}\label{fig:adv_diff_homo_dirichlet_exp_change_n}}
	\subfigure[Fixed $n=11$]{\includegraphics[width=0.32\textwidth]{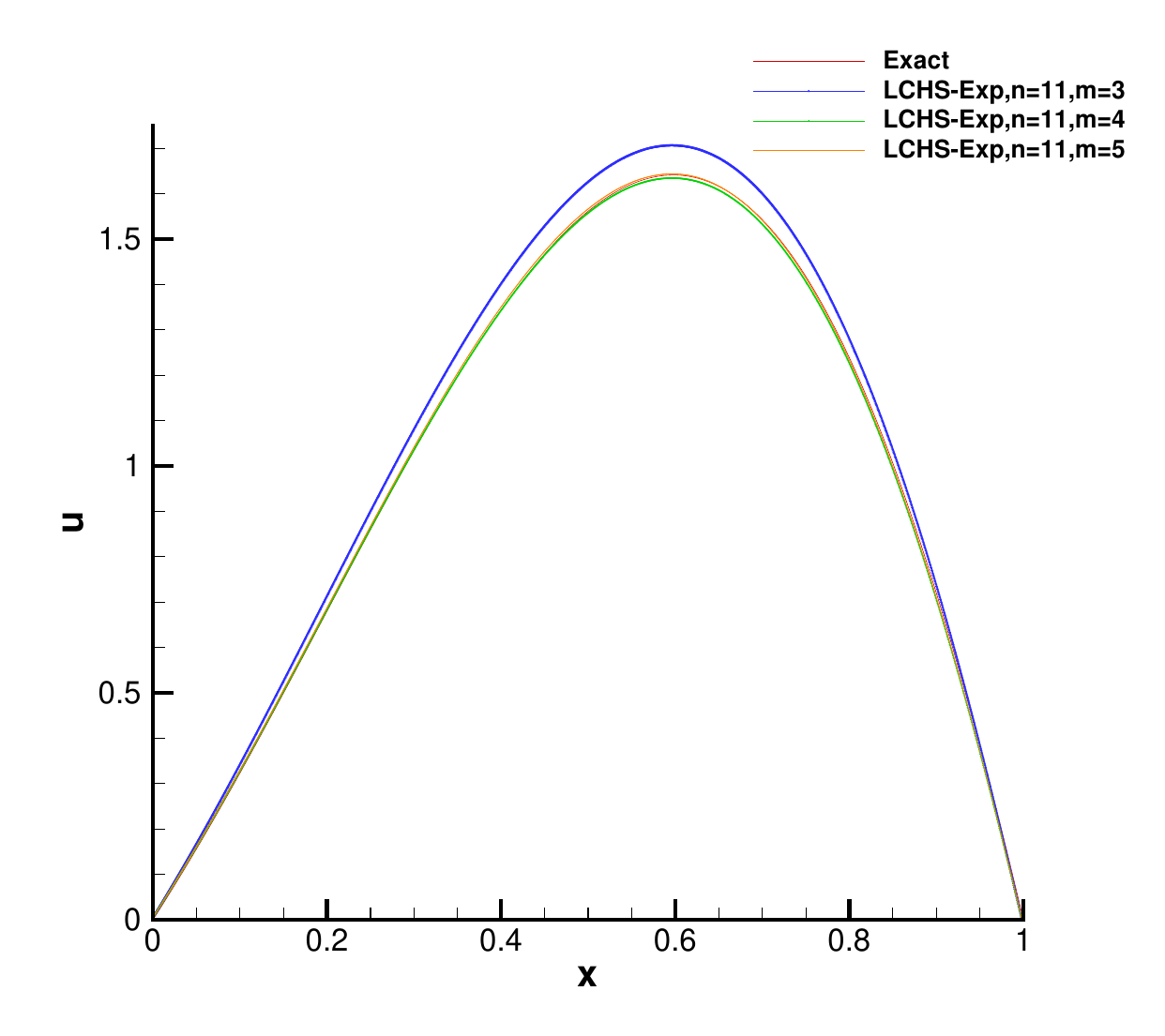}\label{fig:adv_diff_homo_dirichlet_exp_change_m}}
	\caption{Numerical results for Experiment~\ref{experiment:adv_diff_homo_dirichlet} (homogeneous $1$D advection-diffusion equation with Dirichlet BCs, exponential scheme).}
    \label{fig:adv_diff_homo_dirichlet_exp_1d}
\end{figure}

\begin{table}[htbp]
\centering
\begin{minipage}[t]{0.48\textwidth}
    \centering
    \caption{Relative errors for Experiment~\ref{experiment:adv_diff_homo_dirichlet} (central scheme, fixed $m=4$)}
    \begin{tabular}{c c c c}
        \toprule
        $n$ & $L_1$-norm & $L_2$-norm & $L_\infty$-norm \\
        \midrule
        10 & 8.4446e-3 & 8.3141e-3 & 2.1171e-2 \\
        11 & 6.6631e-3 & 7.1931e-3 & 8.8714e-3 \\
        12 & 4.5424e-3 & 4.7422e-3 & 4.9266e-3 \\
        \bottomrule
    \end{tabular}
    \label{tab:adv_diff_homo_dirichlet_central_error_change_n}
\end{minipage}
\hfill
\begin{minipage}[t]{0.48\textwidth}
    \centering
    \caption{Relative errors for Experiment~\ref{experiment:adv_diff_homo_dirichlet} (central scheme, fixed $n=11$)}
    \begin{tabular}{c c c c}
        \toprule
        $m$ & $L_1$-norm & $L_2$-norm & $L_\infty$-norm \\
        \midrule
        3 & 4.0368e-2 & 4.0201e-2 & 4.0182e-2 \\
        4 & 6.6631e-3 & 7.1931e-3 & 8.8714e-3 \\
        5 & 6.1391e-3 & 6.0865e-3 & 8.7803e-3 \\
        \bottomrule
    \end{tabular}
    \label{tab:adv_diff_homo_dirichlet_central_error_change_m}
\end{minipage}
\end{table}

\begin{table}[htbp]
\centering
\begin{minipage}[t]{0.48\textwidth}
    \centering
    \caption{Relative errors for Experiment~\ref{experiment:adv_diff_homo_dirichlet} (exponential scheme, fixed $m=4$)}
    \begin{tabular}{c c c c}
        \toprule
        $n$ & $L_1$-norm & $L_2$-norm & $L_\infty$-norm \\
        \midrule
        10 & 8.4521e-3 & 8.2909e-3 & 1.7769e-2 \\
        11 & 6.6630e-3 & 7.1940e-3 & 8.8713e-3 \\
        12 & 4.5436e-3 & 4.7428e-3 & 4.9268e-3 \\
        \bottomrule
    \end{tabular}
    \label{tab:adv_diff_homo_dirichlet_exp_error_change_n}
\end{minipage}
\hfill
\begin{minipage}[t]{0.48\textwidth}
    \centering
    \caption{Relative errors for Experiment~\ref{experiment:adv_diff_homo_dirichlet} (exponential scheme, fixed $n=11$)}
    \begin{tabular}{c c c c}
        \toprule
        $m$ & $L_1$-norm & $L_2$-norm & $L_\infty$-norm \\
        \midrule
        3 & 4.0368e-2 & 4.0201e-2 & 4.0182e-2 \\
        4 & 6.6630e-3 & 7.1940e-3 & 8.8713e-3 \\
        5 & 6.1390e-3 & 6.0877e-3 & 8.7804e-3 \\
        \bottomrule
    \end{tabular}
    \label{tab:adv_diff_homo_dirichlet_exp_error_change_m}
\end{minipage}
\end{table}

\begin{test}\label{experiment:adv_diff_homo_periodic}
This numerical experiment considers the advection-diffusion equation subject to periodic BCs $u(t,x_L) = u(t,x_R)$ and \replythr{$\frac{\partial u}{\partial x}(t,x_L) = \frac{\partial u}{\partial x}(t,x_R)$}, where $c = 0$.
The initial condition is specified as
\begin{equation}
    u_0(x) =  \sin \frac{2\pi x}{l}, \quad f(t,x)=0.
\end{equation}
This initial-boundary value problem admits an exact solution given by
\begin{equation}
    u(t,x) = \exp\left( - \frac{4\pi^2b}{l^2}  t\right)  \sin \frac{2\pi (x - at)}{l},
\end{equation}
\end{test}

For the numerical implementation of Experiment~\ref{experiment:adv_diff_homo_periodic}, we set parameters as: $a=1, b=0.25, l=1$, $T = 0.01$, $\replythr{N_t} = 16$, and $R = 0.88875 \times 2^m$, and adopt the select oracle implemented as in Eq.~\eqref{eq:select under periodic}.
For the central scheme, numerical results are presented in Fig.~\ref{fig:adv_diff_homo_periodic_central_1d}, with relative errors summarized in Tables~\ref{tab:adv_diff_homo_periodic_central_error_change_n} and \ref{tab:adv_diff_homo_periodic_central_error_change_m}.
For the exponential scheme, numerical results are displayed in Fig.~\ref{fig:adv_diff_homo_periodic_exp_1d}, with relative errors listed in Tables~\ref{tab:adv_diff_homo_periodic_exp_error_change_n} and \ref{tab:adv_diff_homo_periodic_exp_error_change_m}.

These two flux construction schemes exhibit similar behavior with nearly identical relative errors, as observed in Experiment~\ref{experiment:adv_diff_homo_dirichlet}.
For fixed $m=4$, the relative errors decrease as $n$ increases from $9$ to $10$, but increase from $10$ to $11$ as the circuit error becomes dominant.
For fixed $n=10$, the relative errors decrease significantly as $m$ increases from $3$ to $4$, and increase slightly as $m$ rises to $5$\replythr{, consistent with the balance observed in Experiment~\ref{experiment:adv_homo_periodic}}.

\begin{figure}[htbp]
    \centering
    \subfigure[Fixed $m=4$]{\includegraphics[width=0.32\textwidth]{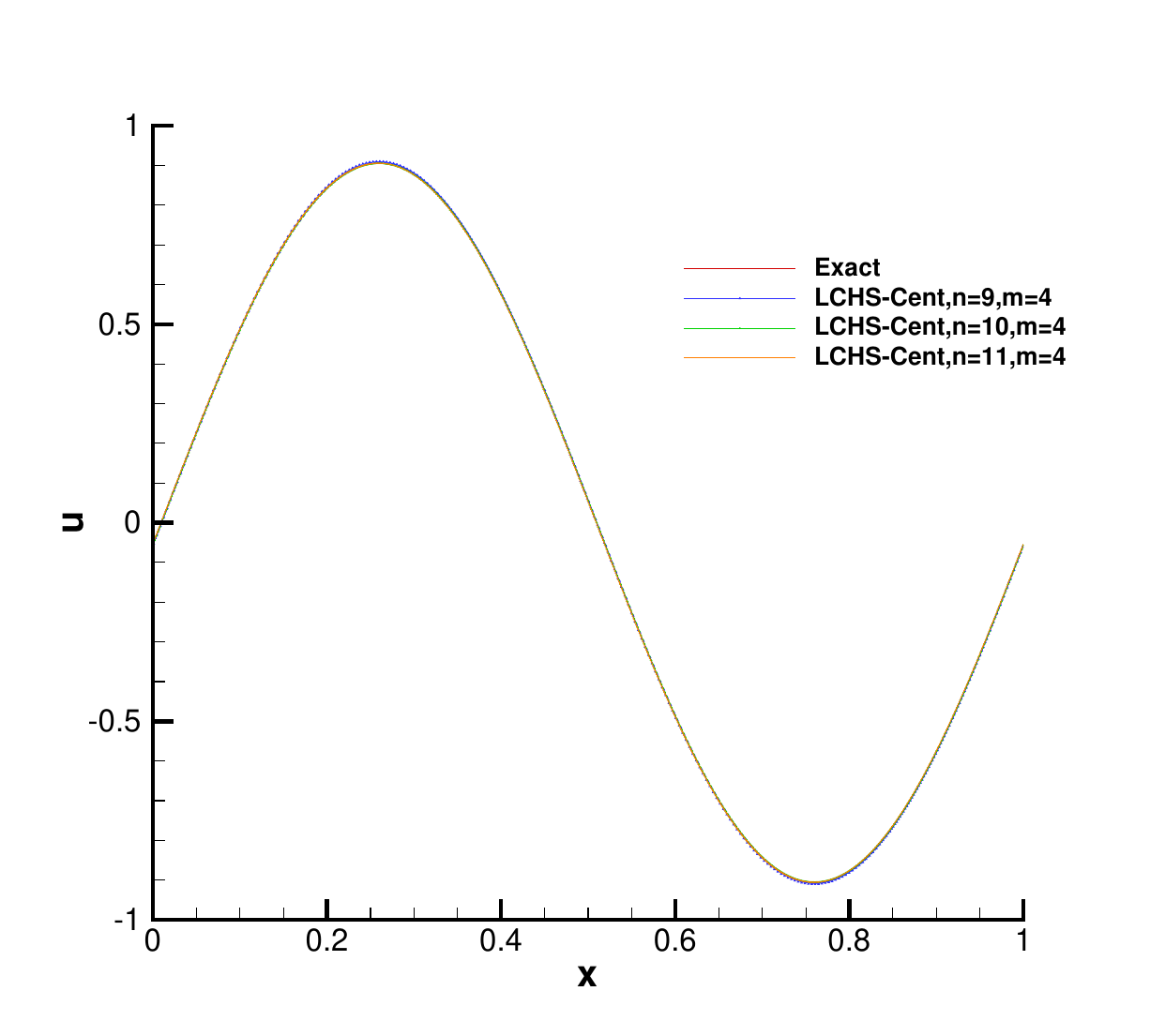}\label{fig:adv_diff_homo_periodic_central_change_n}}
	\subfigure[Fixed $n=11$]{\includegraphics[width=0.32\textwidth]{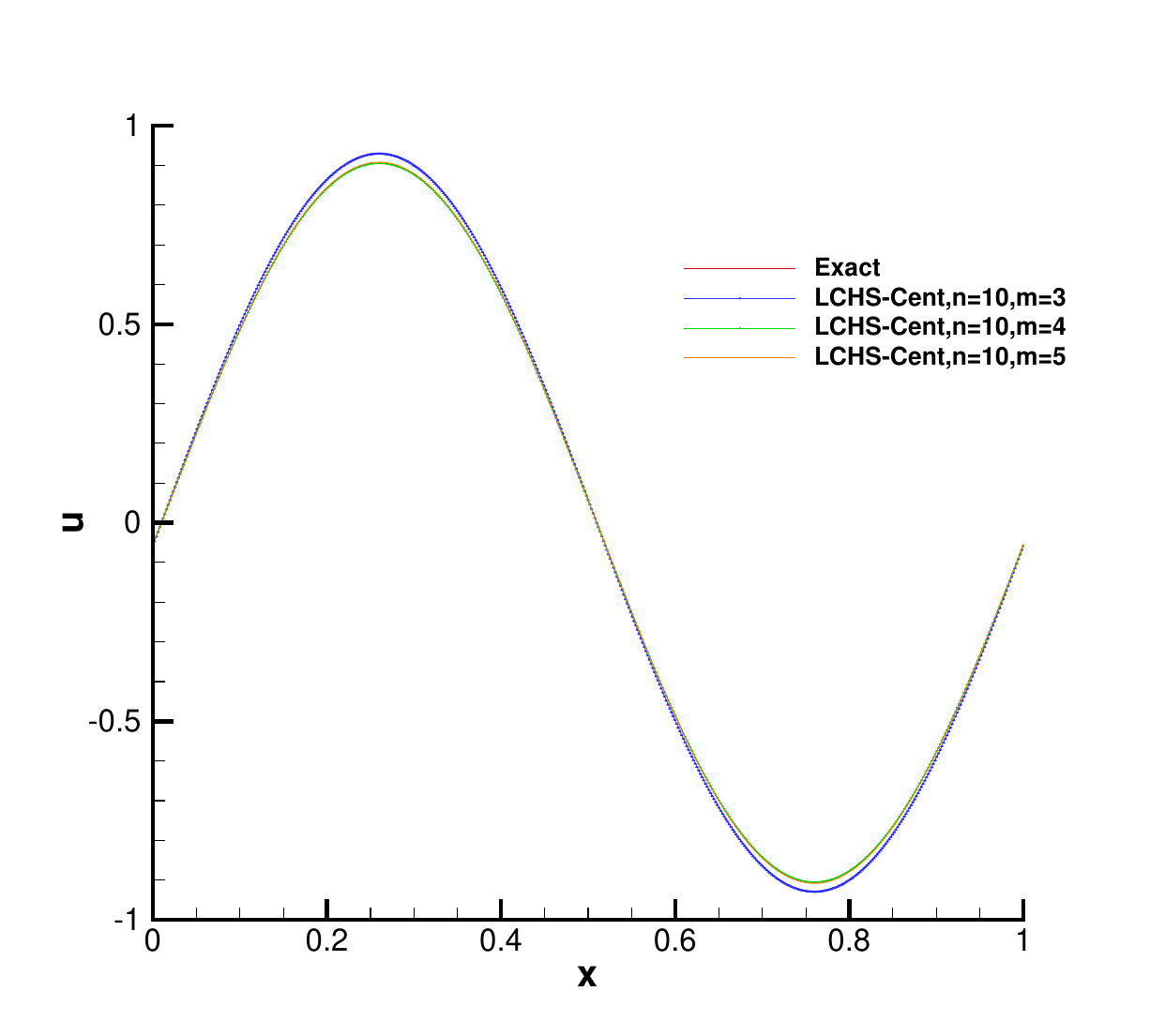}\label{fig:adv_diff_homo_periodic_central_change_m}}
	\caption{Numerical results for Experiment~\ref{experiment:adv_diff_homo_periodic} (homogeneous $1$D advection-diffusion equation with periodic BCs, central scheme).}
    \label{fig:adv_diff_homo_periodic_central_1d}
\end{figure}

\begin{figure}[htbp]
    \centering
    \subfigure[Fixed $m=4$]{\includegraphics[width=0.32\textwidth]{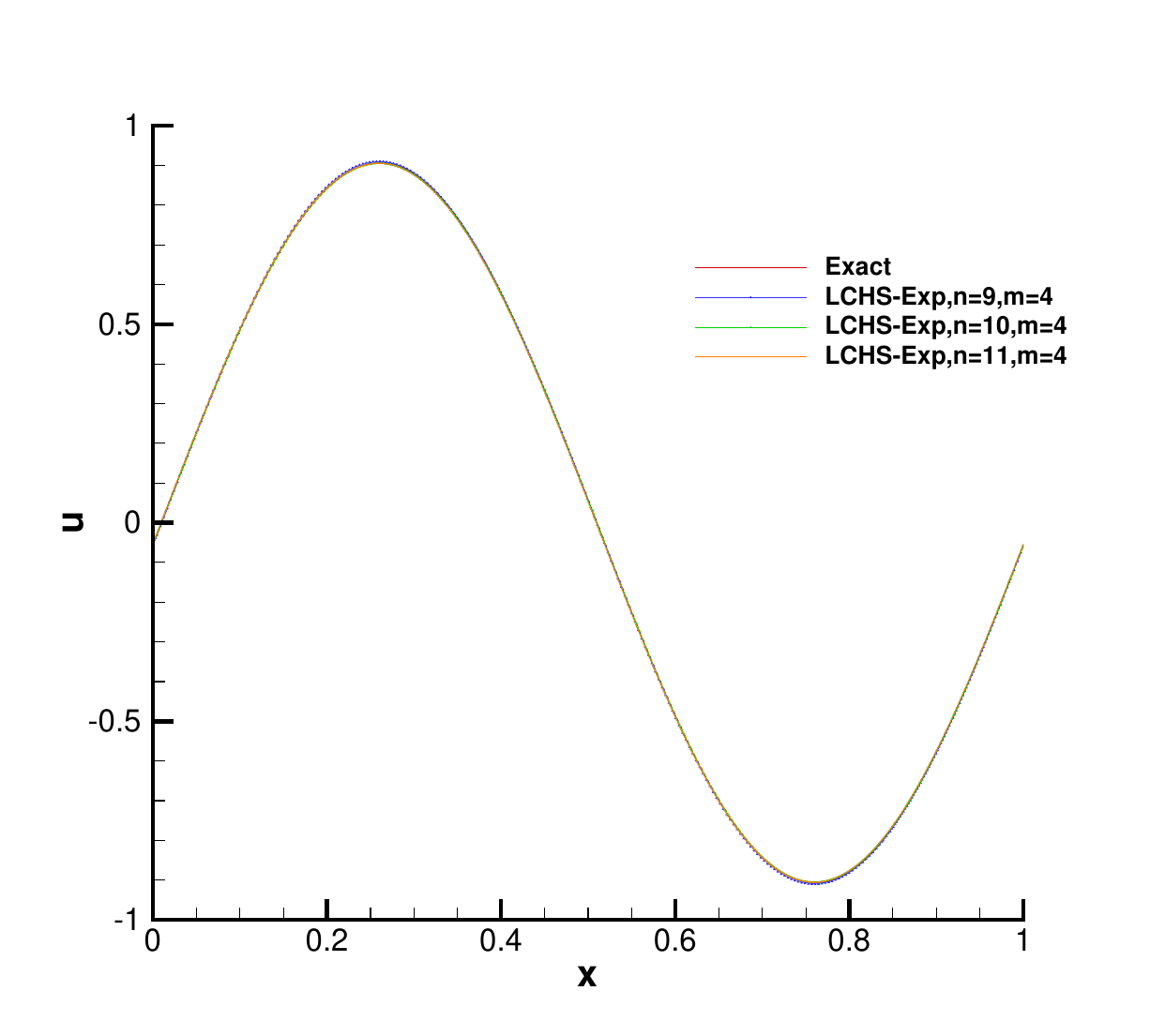}\label{fig:adv_diff_homo_periodic_exp_change_n}}
	\subfigure[Fixed $n=11$]{\includegraphics[width=0.32\textwidth]{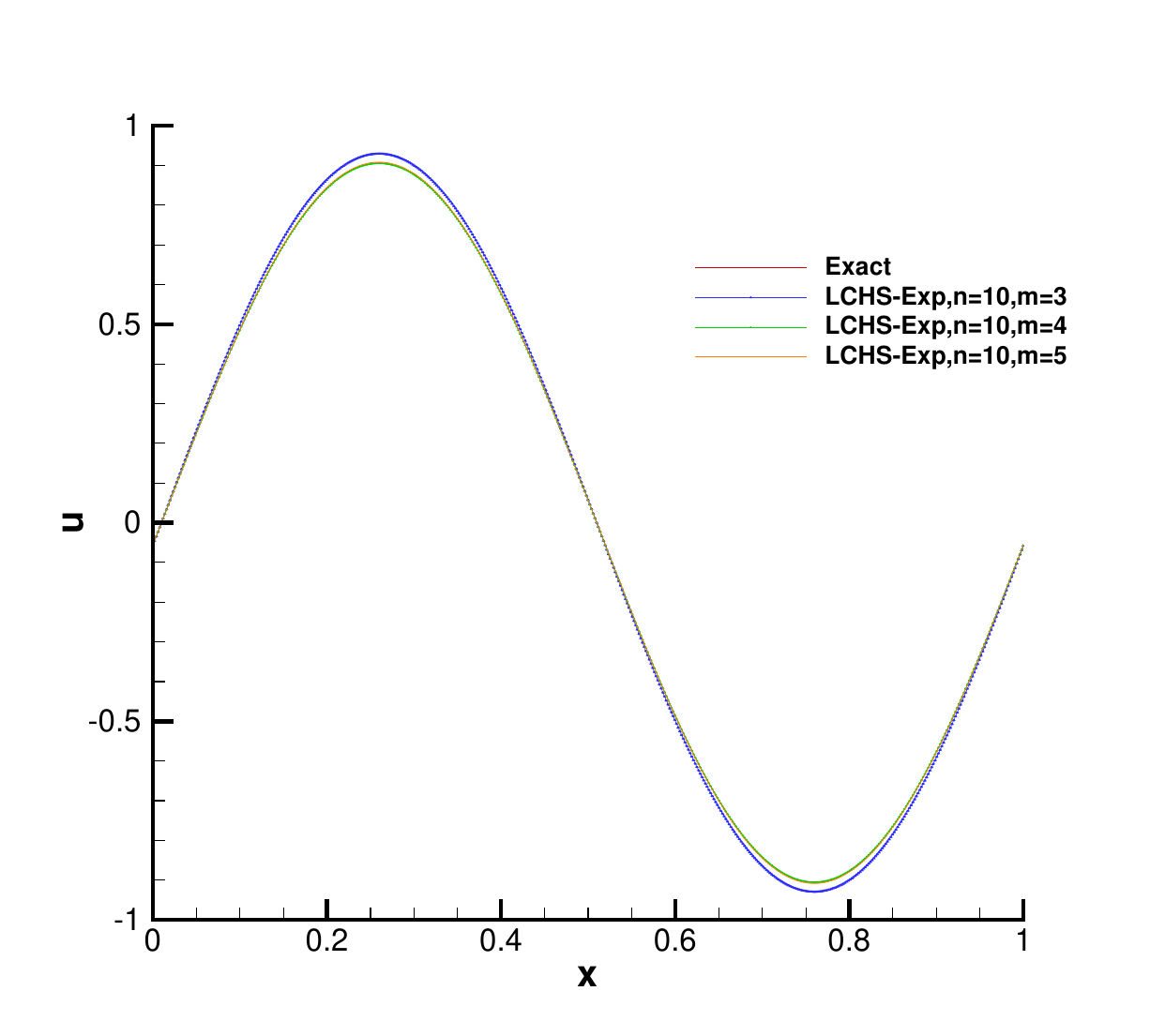}\label{fig:adv_diff_homo_periodic_exp_change_m}}
	\caption{Numerical results for Experiment~\ref{experiment:adv_diff_homo_periodic} (homogeneous $1$D advection-diffusion equation with periodic BCs, exponential scheme).}
    \label{fig:adv_diff_homo_periodic_exp_1d}
\end{figure}

\begin{table}[htbp]
\centering
\begin{minipage}[t]{0.48\textwidth}
    \centering
    \caption{Relative errors for Experiment~\ref{experiment:adv_diff_homo_periodic} (central scheme, fixed $m=4$)}
    \begin{tabular}{c c c c}
        \toprule
        $n$ & $L_1$-norm & $L_2$-norm & $L_\infty$-norm \\
        \midrule
        9 & 3.5575e-3 & 3.9240e-3 & 5.2133e-3 \\
        10 & 2.7758e-3 & 2.9517e-3 & 3.7794e-3 \\
        11 & 4.3413e-3 & 4.7290e-3 & 6.2177e-3 \\
        \bottomrule
    \end{tabular}
    \label{tab:adv_diff_homo_periodic_central_error_change_n}
\end{minipage}
\hfill
\begin{minipage}[t]{0.48\textwidth}
    \centering
    \caption{Relative errors for Experiment~\ref{experiment:adv_diff_homo_periodic} (central scheme, fixed $n=10$)}
    \begin{tabular}{c c c c}
        \toprule
        $m$ & $L_1$-norm & $L_2$-norm & $L_\infty$-norm \\
        \midrule
        3 & 2.5924e-2 & 2.5924e-2 & 2.5965e-2 \\
        4 & 2.7758e-3 & 2.9517e-3 & 3.7794e-3 \\
        5 & 3.5250e-3 & 3.6027e-3 & 4.2693e-3 \\
        \bottomrule
    \end{tabular}
    \label{tab:adv_diff_homo_periodic_central_error_change_m}
\end{minipage}
\end{table}

\begin{table}[htbp]
\centering
\begin{minipage}[t]{0.48\textwidth}
    \centering
    \caption{Relative errors for Experiment~\ref{experiment:adv_diff_homo_periodic} (exponential scheme, fixed $m=4$)}
    \begin{tabular}{c c c c}
        \toprule
        $n$ & $L_1$-norm & $L_2$-norm & $L_\infty$-norm \\
        \midrule
        9 & 3.4606e-3 & 3.8580e-3 & 5.1659e-3 \\
        10 & 2.7910e-3 & 2.9667e-3 & 3.7967e-3 \\
        11 & 4.3402e-3 & 4.7065e-3 & 6.1612e-3 \\
        \bottomrule
    \end{tabular}
    \label{tab:adv_diff_homo_periodic_exp_error_change_n}
\end{minipage}
\hfill
\begin{minipage}[t]{0.48\textwidth}
    \centering
    \caption{Relative errors for Experiment~\ref{experiment:adv_diff_homo_periodic} (exponential scheme, fixed $n=10$)}
    \begin{tabular}{c c c c}
        \toprule
        $m$ & $L_1$-norm & $L_2$-norm & $L_\infty$-norm \\
        \midrule
        3 & 2.5929e-2 & 2.5929e-2 & 2.5966e-2 \\
        4 & 2.7910e-3 & 2.9667e-3 & 3.7967e-3 \\
        5 & 3.5381e-3 & 3.6158e-3 & 4.2835e-3 \\
        \bottomrule
    \end{tabular}
    \label{tab:adv_diff_homo_periodic_exp_error_change_m}
\end{minipage}
\end{table}

\subsection{$1$-dimensional inhomogeneous advection-diffusion equation}

This section considers the $1$-dimensional advection-diffusion equation (Eq.~\eqref{eq:1d_adv_diff}) with the initial condition $u_0(x)=0$ and source term $f(t,x) \neq 0$, which yields an inhomogeneous ODE system after spatial discretization.
This setup is used to evaluate the performance of the LCHS method and the proposed quantum circuits for inhomogeneous terms under different flux construction schemes and varying $m_o$.
For the quantum circuits, we adopt the circuit in Eq.~\eqref{eq:outer LCHS}, with the outer select oracle formulated as in Eq.~\eqref{eq:implementation of sel-U_kj} and the inner select oracle implemented under the boundary conditions introduced in Sections~\ref{subsubsec:quantum circuit under the robin boundary conditions}, \ref{subsubsec:quantum circuit under the periodic boundary conditions} and Appendix~\ref{subsec:periodic with alpha=0}.

\begin{test}\label{experiment:adv_inhomo_periodic}
This experiment investigates the pure advection equation ($b = c = 0$) subject to periodic BCs $u(t,x_L) = u(t,x_R)$ and \replythr{$\frac{\partial u}{\partial x}(t,x_L) = \frac{\partial u}{\partial x}(t,x_R)$}.
The initial condition is specified as
\begin{equation}
    u_0(x) = 0, \quad f(t,x) = \cos \frac{2\pi x}{l},
\end{equation}
yielding the exact solution
\begin{equation}
    u(t,x) = \frac{l}{2\pi a} \left( \sin \frac{2\pi x}{l} - \sin \frac{2\pi (x-at)}{l} \right).
\end{equation}
\end{test}

For the numerical implementation of Experiment~\ref{experiment:adv_inhomo_periodic}, we set parameters as: $a = 1$, $l = 1$, $T = 0.05$, $\replythr{N_t} = 8$, $n=8$ and $R=8$.
For the central scheme, we employ the inner select oracle implemented as in Eq.~\eqref{eq:select under periodic with alpha=0}, with numerical results presented in Fig.~\ref{fig:adv_inhomo_periodic_central_change_mo} and relative errors summarized in Table~\ref{tab:adv_inhomo_periodic_central_error_change_mo}.
For the upwind scheme, we set $R = 34$, $m=4$ and adopt the inner select oracle implemented as in Eq.~\eqref{eq:select under periodic}, with numerical results plotted in Fig.~\ref{fig:adv_inhomo_periodic_upwind_change_mo} and relative errors listed in Table~\ref{tab:adv_inhomo_periodic_upwind_error_change_mo}.

Numerical results show that the relative errors of these two flux construction schemes are of the same order of magnitude, as the error introduced by the outer quadrature dominates; accordingly, both errors decrease slightly as $m_o$ increases.
Owing to the employment of the LCHS method in the upwind scheme, its relative errors are slightly larger than those of the central scheme.

\begin{figure}[htbp]
    \centering
    \subfigure[Central scheme, fixed $n=8$]{\includegraphics[width=0.32\textwidth]{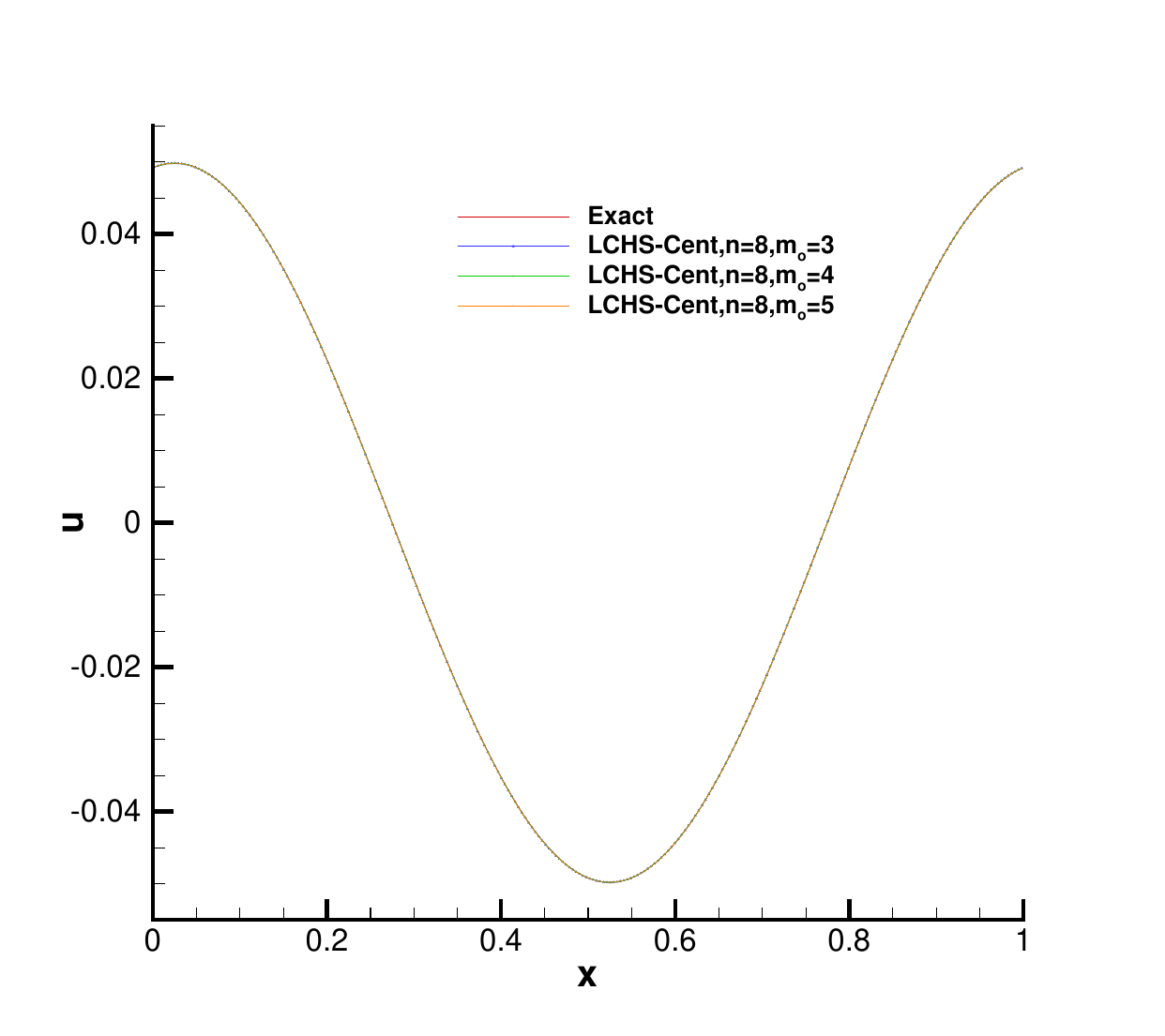}\label{fig:adv_inhomo_periodic_central_change_mo}}
	\subfigure[Upwind scheme, fixed $n=8,m=4$]{\includegraphics[width=0.32\textwidth]{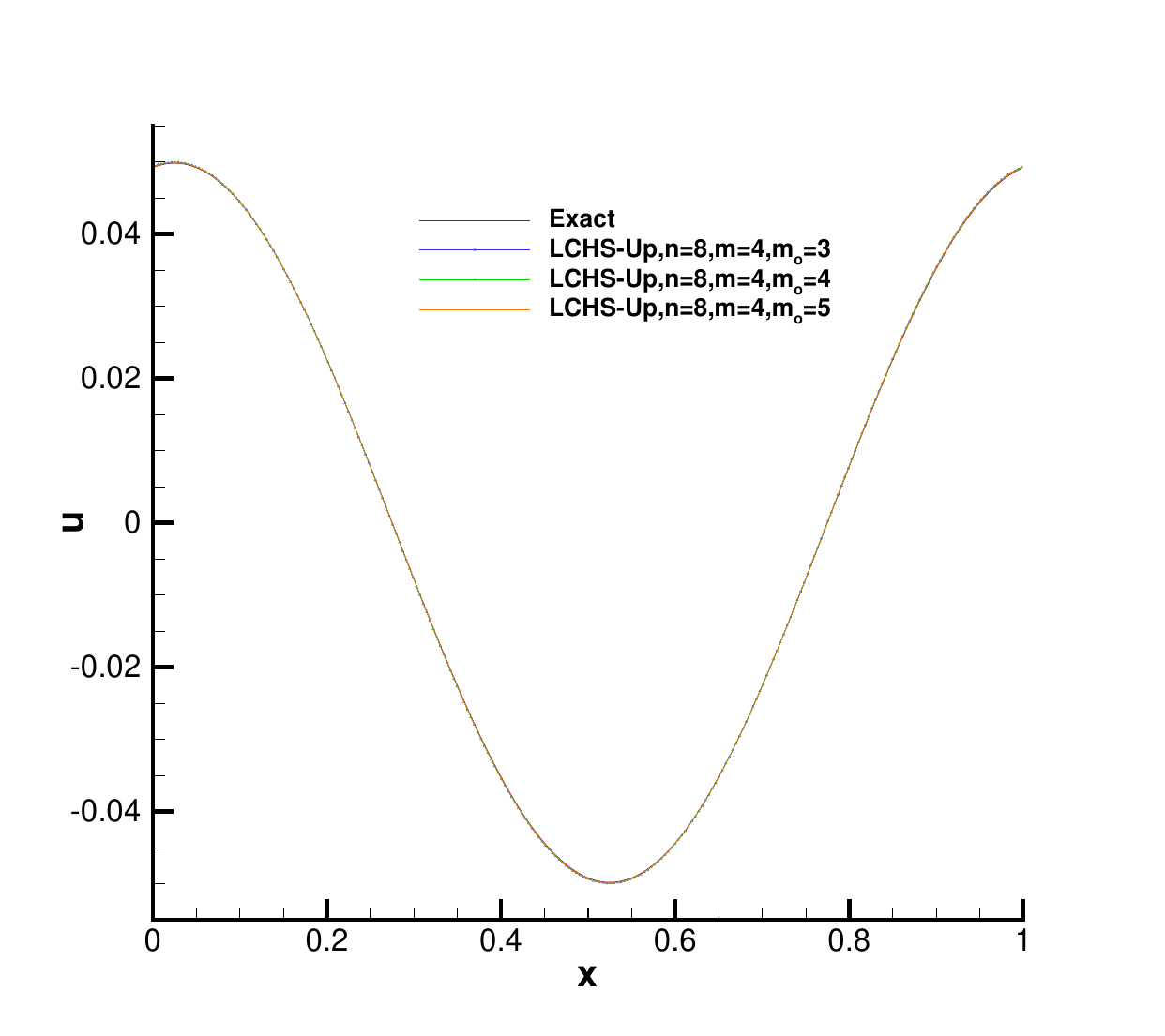}\label{fig:adv_inhomo_periodic_upwind_change_mo}}
	\caption{Numerical results for Experiment~\ref{experiment:adv_inhomo_periodic} (inhomogeneous $1$D advection equation with periodic BCs).}
\end{figure}

\begin{table}[htbp]
\centering
\begin{minipage}[t]{0.48\textwidth}
    \centering
    \caption{Relative errors for Experiment~\ref{experiment:adv_inhomo_periodic} (central scheme, fixed $n=8$)}
    \begin{tabular}{c c c c}
        \toprule
        $m_o$ & $L_1$-norm & $L_2$-norm & $L_\infty$-norm \\
        \midrule
        3 & 2.4119e-3 & 2.4122e-3 & 2.4441e-3 \\
        4 & 2.4046e-3 & 2.4048e-3 & 2.4337e-3 \\
        5 & 2.4035e-3 & 2.4035e-3 & 2.4315e-3 \\
        \bottomrule
    \end{tabular}
    \label{tab:adv_inhomo_periodic_central_error_change_mo}
\end{minipage}
\hfill
\begin{minipage}[t]{0.48\textwidth}
    \centering
    \caption{Relative errors for Experiment~\ref{experiment:adv_inhomo_periodic} (upwind scheme, fixed $n=8,m=4$)}
    \begin{tabular}{c c c c}
        \toprule
        $m_o$ & $L_1$-norm & $L_2$-norm & $L_\infty$-norm \\
        \midrule
        3 & 3.7480e-3 & 3.7580e-3 & 4.0244e-3 \\
        4 & 3.7127e-3 & 3.7223e-3 & 3.9848e-3 \\
        5 & 3.7039e-3 & 3.7135e-3 & 3.9753e-3 \\
        \bottomrule
    \end{tabular}
    \label{tab:adv_inhomo_periodic_upwind_error_change_mo}
\end{minipage}
\end{table}

\begin{test}\label{experiment:adv_diff_inhomo_dirichlet}
This numerical experiment considers the advection-diffusion equation with Dirichlet BCs $u(t,x_L) = u(t,x_R) = 0$, where $c = 0$.
The initial condition is given by
\begin{equation}
    u_0(x) = 0, \quad
    f(t,x) = \exp\left( \frac{a}{2b}x \right) \left[ e^{-t} + \left( 1-e^{-t} \right) \left( \frac{a^2}{4b} + \frac{\pi^2 b}{l^2} \right) \right]  \sin \frac{\pi x}{l}.
\end{equation}
This initial-boundary value problem admits the exact solution
\begin{equation}
    u(t,x) = \left( 1-e^{-t} \right) \exp\left( \frac{a}{2b}x \right) \sin \frac{\pi x}{l},
\end{equation}
\end{test}

For the numerical implementation of Experiment~\ref{experiment:adv_diff_inhomo_dirichlet}, we set parameters as: $a = 1$, $b=0.25$, $l = 1$, $T = 0.01$, $\replythr{N_t} = 1$, $R = 9.6$, $n=11$ and $m=4$, and adopt the inner select oracle implemented as in Eq.~\eqref{eq:select under robin}.
For the central scheme, numerical results are illustrated in Fig.~\ref{fig:adv_diff_inhomo_dirichlet_central_change_mo}, while the relative errors are recorded in Table~\ref{tab:adv_diff_inhomo_dirichlet_central_error_change_mo}.
For the exponential scheme, numerical results are depicted in Fig.~\ref{fig:adv_diff_inhomo_dirichlet_exp_change_mo}, while the relative errors are reported in Table~\ref{tab:adv_diff_inhomo_dirichlet_exp_error_change_mo}.

Consistent with the homogeneous case in Experiment~\ref{experiment:adv_diff_homo_dirichlet}, the two flux construction schemes yield nearly identical relative errors.
\replythr{The errors remain nearly constant as $m_o$ increases from $3$ to $5$, indicating that with $n=11$ and $m=4$ fixed, the outer quadrature resolution is not the bottleneck; the error is dominated by the inner-loop contributions from spatial discretization and LCHS truncation. A similar though weaker insensitivity to $m_o$ is also visible in Experiment~\ref{experiment:adv_inhomo_periodic}.}

\begin{figure}[htbp]
    \centering
    \subfigure[Central scheme]{\includegraphics[width=0.32\textwidth]{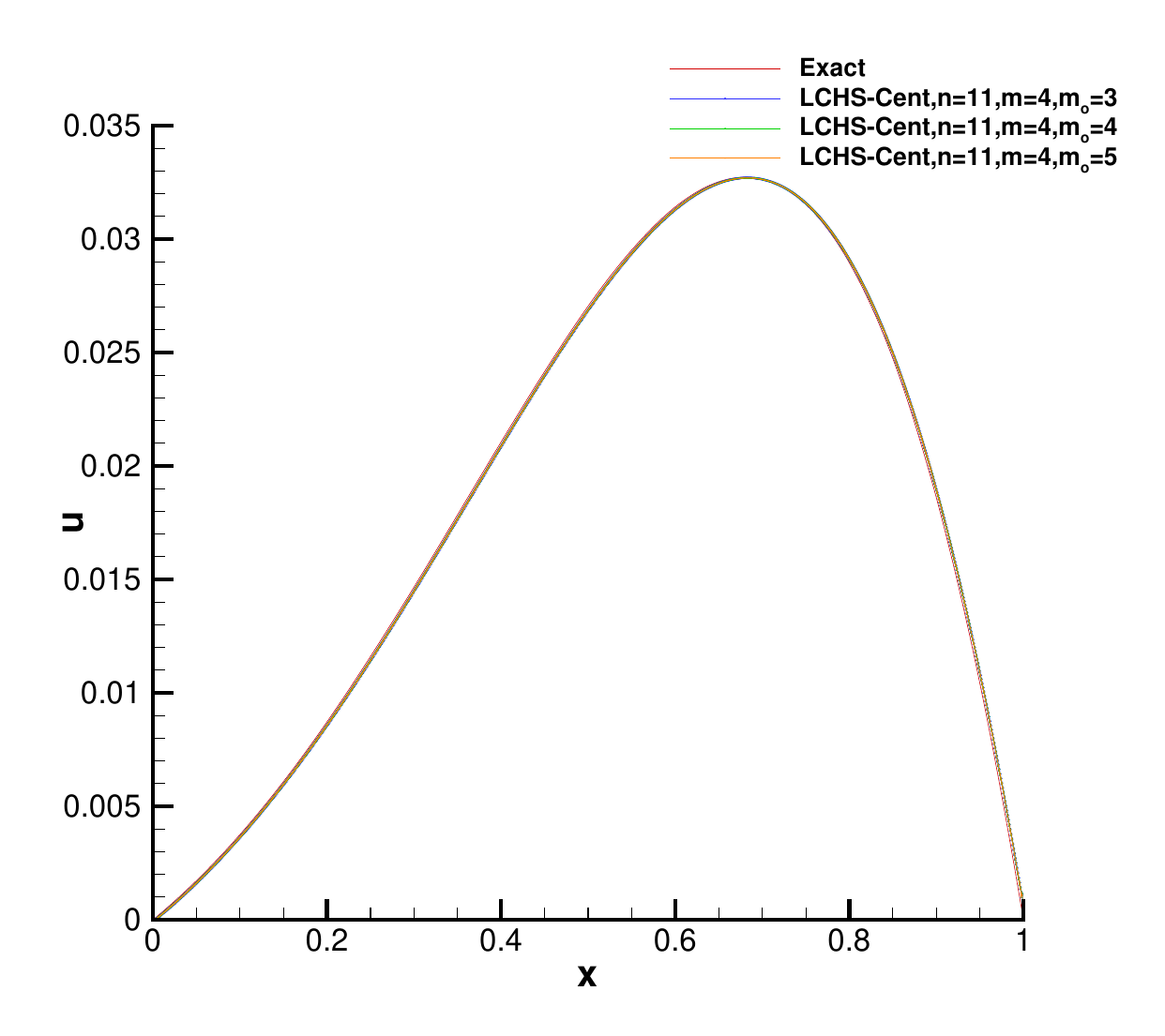}\label{fig:adv_diff_inhomo_dirichlet_central_change_mo}}
	\subfigure[Exponential scheme]{\includegraphics[width=0.32\textwidth]{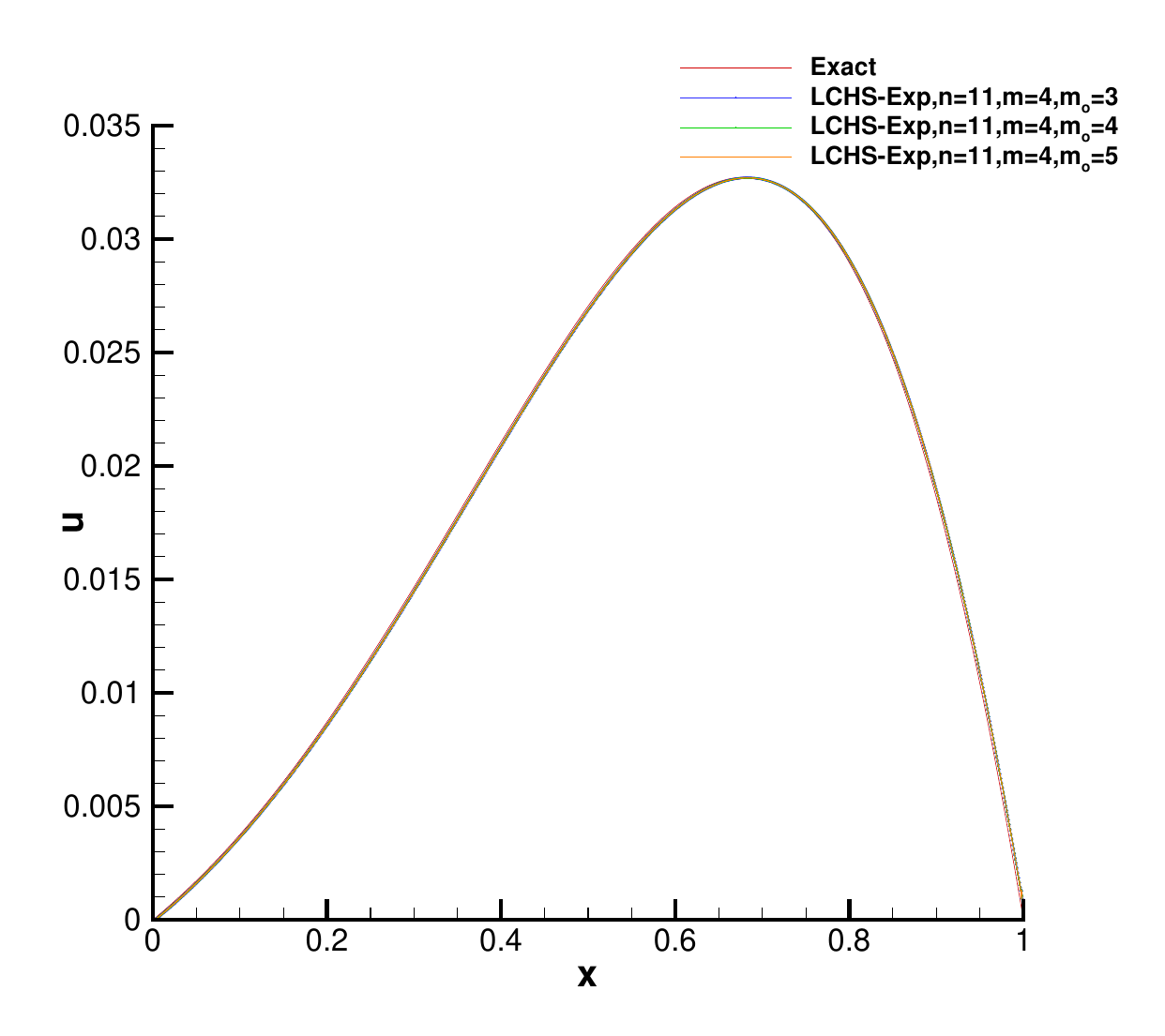}\label{fig:adv_diff_inhomo_dirichlet_exp_change_mo}}
	\caption{Numerical results for Experiment~\ref{experiment:adv_diff_inhomo_dirichlet} (inhomogeneous $1$D advection-diffusion equation with Dirichlet BCs, fixed $n=11, m=4$).}
\end{figure}

\begin{table}[htbp]
\centering
\begin{minipage}[t]{0.48\textwidth}
    \centering
    \caption{Relative errors for Experiment~\ref{experiment:adv_diff_inhomo_dirichlet} (central scheme, fixed $n=11,m=4$)}
    \begin{tabular}{c c c c}
        \toprule
        $m_o$ & $L_1$-norm & $L_2$-norm & $L_\infty$-norm \\
        \midrule
        3 & 9.4449e-3 & 1.0019e-2 & 2.6340e-2 \\
        4 & 9.4555e-3 & 1.0035e-2 & 2.6723e-2 \\
        5 & 9.4607e-3 & 1.0055e-2 & 2.5693e-2 \\
        \bottomrule
    \end{tabular}
    \label{tab:adv_diff_inhomo_dirichlet_central_error_change_mo}
\end{minipage}
\hfill
\begin{minipage}[t]{0.48\textwidth}
    \centering
    \caption{Relative errors for Experiment~\ref{experiment:adv_diff_inhomo_dirichlet} (exponential scheme, fixed $n=11,m=4$)}
    \begin{tabular}{c c c c}
        \toprule
        $m_o$ & $L_1$-norm & $L_2$-norm & $L_\infty$-norm \\
        \midrule
        3 & 9.4412e-3 & 1.0030e-2 & 2.6795e-2 \\
        4 & 9.4519e-3 & 1.0044e-2 & 2.7187e-2 \\
        5 & 9.4570e-3 & 1.0064e-2 & 2.5426e-2 \\
        \bottomrule
    \end{tabular}
    \label{tab:adv_diff_inhomo_dirichlet_exp_error_change_mo}
\end{minipage}
\end{table}

\subsection{$2$-dimensional advection-diffusion equation}

This section considers the $2$-dimensional advection-diffusion equation given by
\begin{equation}\label{eq:2d_adv_diff}
    \left\{
    \begin{aligned}
    \frac{\partial u}{\partial t} + a_1\frac{\partial u}{\partial x} + a_2\frac{\partial u}{\partial y} + cu &= b_1\frac{\partial^2 u}{\partial x^2} + b_2\frac{\partial^2 u}{\partial y^2} + f(t,x,y), \\
    u(0, x, y) &= u_0(x, y),
    \end{aligned}
    \right.
     \quad (x,y) \in \Omega=[0,l_1] \times [0,l_2].
\end{equation}
We set the source term $f(t,x,y) = 0$ and impose homogeneous boundary conditions, yielding a homogeneous ODE system after spatial discretization. This setup enables testing the effectiveness of the LCHS method and the proposed quantum circuits for multidimensional problems.
For the quantum circuit implementation, we employ the circuit in Eq.~\eqref{eq:inner LCHS} with the global select oracle formulated as in Eq.~\eqref{eq:global select}, where the select oracles for each dimension are implemented according to the boundary conditions detailed in Sections~\ref{subsubsec:quantum circuit under the robin boundary conditions}, \ref{subsubsec:quantum circuit under the periodic boundary conditions}.

\begin{test}\label{experiment:diff_homo_robin_2d}
This numerical experiment examines the pure diffusion equation ($a_1 = a_2 = c = 0$) with homogeneous Robin BCs:
\replytwo{for both spatial directions, the parameters are chosen as $(\alpha_L,\beta_L)=(0,1)$ on the left ($x=x_L$) and bottom ($y=y_L$) boundaries, and $(\alpha_R,\beta_R)=(1,0)$ on the right ($x=x_R$) and top ($y=y_R$) boundaries.}
\replythr{Consequently, the boundary conditions read $\frac{\partial u}{\partial x}(t,x_L,y) = \frac{\partial u}{\partial y}(t,x,y_L) = 0$ and $u(t,x_R,y) = u(t,x,y_R) = 0$.}

The initial condition is
\begin{equation}
    u_0(x,y) = \cos \left( \frac{\pi x}{2l_1} \right) \cos \left( \frac{\pi y}{2l_2} \right), \quad f(t,x,y) = 0,
\end{equation}
admitting the exact solution
\begin{equation}
    u(t,x,y) = \exp \left( -\left( \frac{b_1}{4l_1^2} + \frac{b_2}{4l_2^2} \right)\pi^2 t \right) \cos \left( \frac{\pi x}{2l_1} \right) \cos \left( \frac{\pi y}{2l_2} \right).
\end{equation}
\end{test}

For the numerical implementation of Experiment~\ref{experiment:diff_homo_robin_2d}, we set parameters as: $b_1=b_2=1$, $l_1 = l_2 = 1$, $T = 0.04$, $\replythr{N_t} = 1$, $R=17.8$, $n_1=n_2=7$ and $m=4$.
Applying the central scheme and the select oracle implemented as in Eq.~\eqref{eq:select under robin} for each dimension, numerical results are presented in Fig.~\ref{fig:diff_homo_robin_2d}, with relative errors summarized in Table~\ref{tab:diff_homo_2d_error}.

Numerical results demonstrate that the absolute errors are pronounced at the peaks presented in Fig.~\ref{fig:diff_homo_robin_2d_error}, whereas they remain negligible in regions where the solution values are close to zero.
\replythr{The relative error is comparable to those of the $1$-dimensional diffusion tests (Experiments~\ref{experiment:diff_homo_dirichlet} and~\ref{experiment:diff_homo_neumann}), confirming that the multidimensional assembly via the global select oracle in Eq.~\eqref{eq:global select} preserves the per-dimension accuracy.}

\begin{figure}[htbp]
    \centering
    \subfigure[Exact solution]{\includegraphics[width=0.32\textwidth]{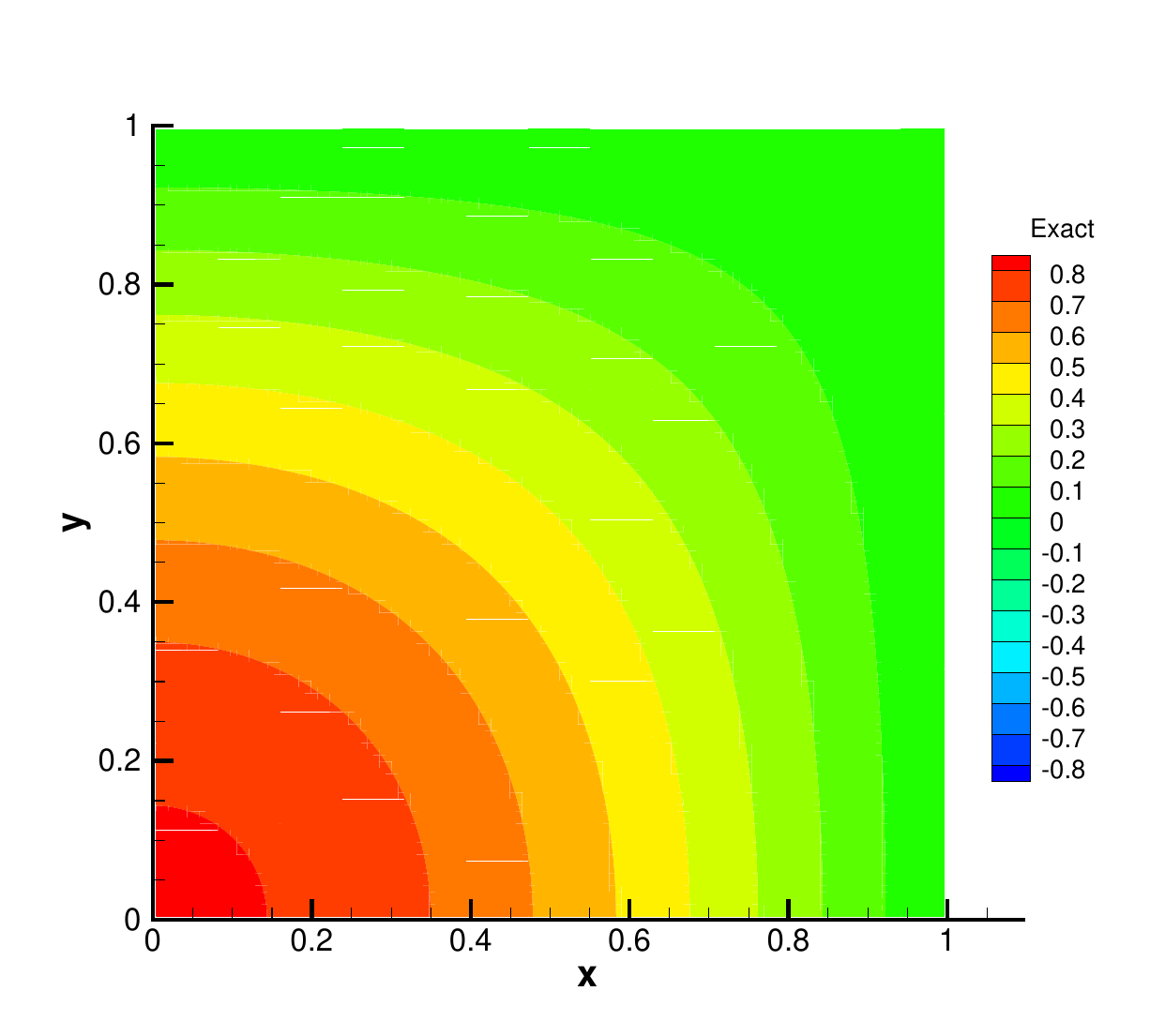}\label{fig:diff_homo_robin_2d_exact}}
	\subfigure[Central scheme]{\includegraphics[width=0.32\textwidth]{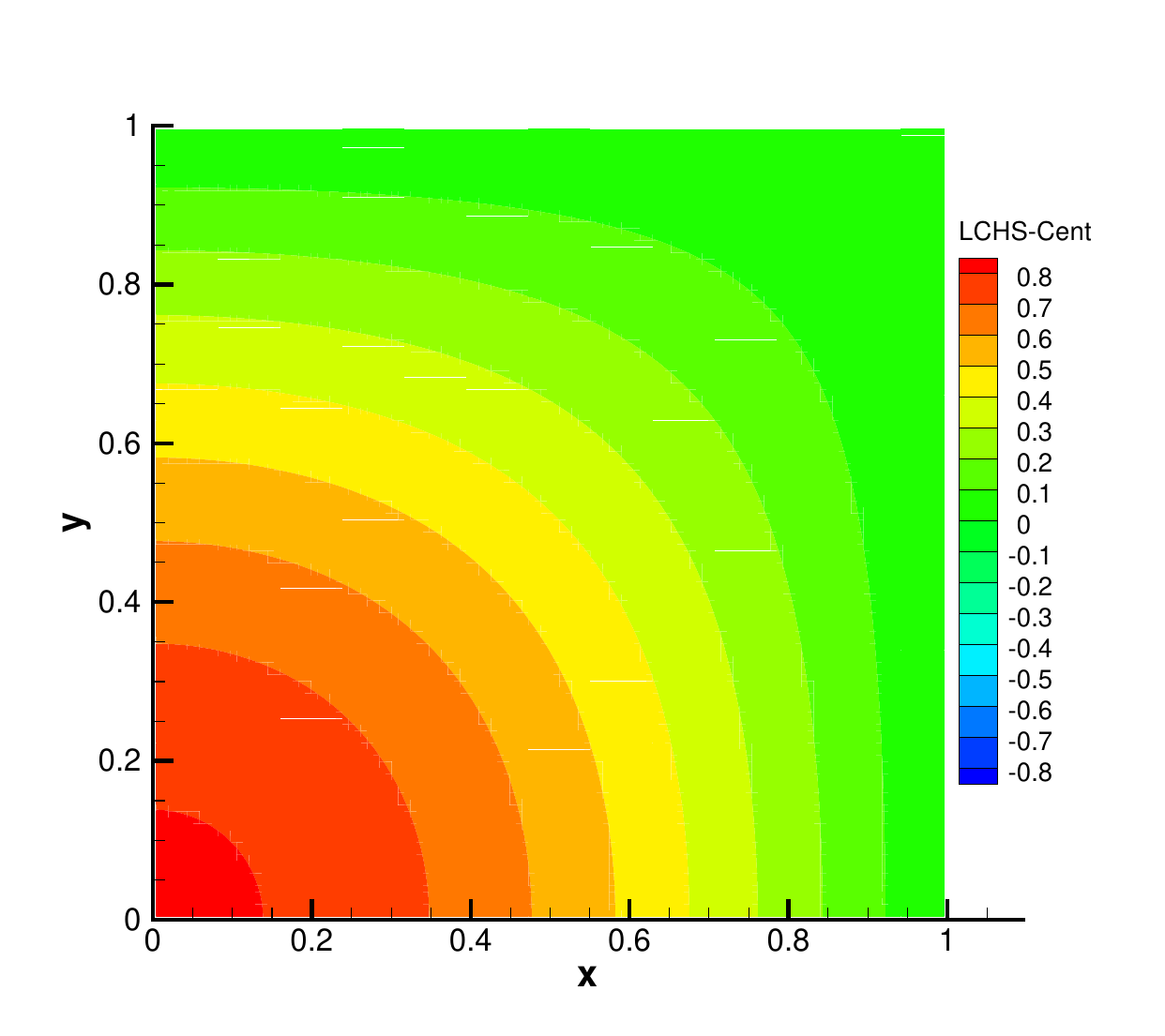}\label{fig:diff_homo_robin_2d_lchs}}
    \subfigure[Absolute error]{\includegraphics[width=0.32\textwidth]{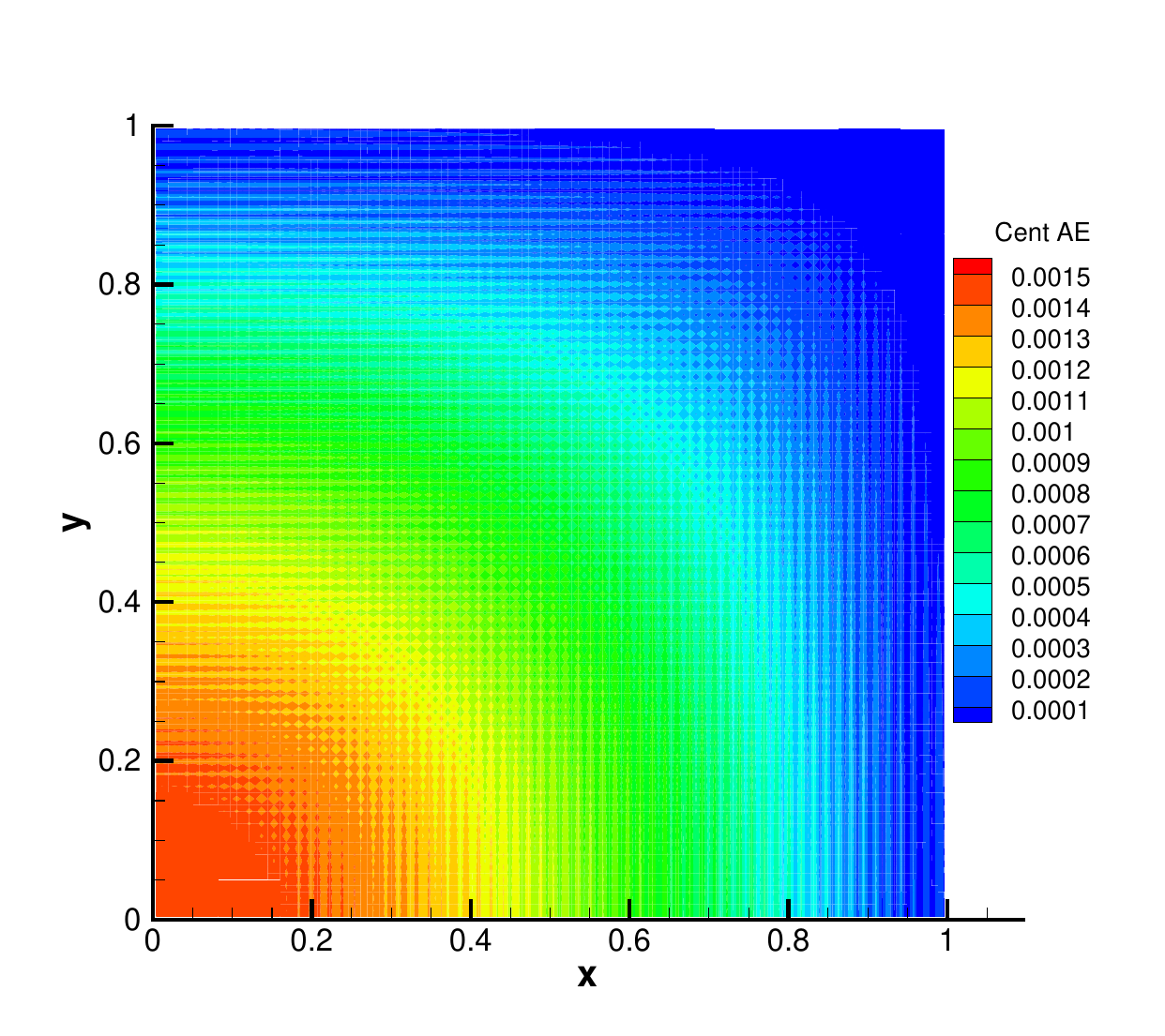}\label{fig:diff_homo_robin_2d_error}}
	\caption{Numerical results for Experiment~\ref{experiment:diff_homo_robin_2d} (homogeneous $2$D diffusion equation with Robin BCs).}
    \label{fig:diff_homo_robin_2d}
\end{figure}

\begin{table}[htbp]
\centering
\begin{minipage}[t]{0.48\textwidth}
    \centering
    \caption{Relative errors for Experiment~\ref{experiment:diff_homo_robin_2d}}
    \begin{tabular}{c c c c}
        \toprule
        Boundary conditions & $L_1$-norm & $L_2$-norm & $L_\infty$-norm \\
        \midrule
        Robin & 1.8134e-3 & 1.8133e-3 & 1.8132e-3 \\
        \bottomrule
    \end{tabular}
    \label{tab:diff_homo_2d_error}
\end{minipage}
\end{table}

\begin{test}\label{experiment:adv_diff_homo_periodic_2d}
The final numerical experiment considers the advection-diffusion equation subject to periodic BCs $u(t,x_L,y) = u(t,x_R,y), \replythr{\frac{\partial u}{\partial x}(t,x_L,y) = \frac{\partial u}{\partial x}(t,x_R,y)},  u(t,x,y_L) = u(t,x,y_R), \replythr{\frac{\partial u}{\partial y}(t,x,y_L) = \frac{\partial u}{\partial y}(t,x,y_R)}$ (where $c = 0$).
The initial condition is
\begin{equation}
    u_0(x,y) = \sin\left(\frac{2\pi x}{l_1}\right) \sin\left(\frac{2\pi y}{l_2}\right), \quad f(t,x,y) = 0,
\end{equation}
yielding the exact solution
\begin{equation}
    u(t,x,y) = \exp\left( -\left( \frac{4\pi^2 b_1}{l_1^2} + \frac{4\pi^2 b_2}{l_2^2} \right) t \right) \sin\left(\frac{2\pi (x - a_1 t)}{l_1}\right) \sin\left(\frac{2\pi (y - a_2 t)}{l_2}\right).
\end{equation}
\end{test}

For the numerical implementation of the problem in Experiment~\ref{experiment:adv_diff_homo_periodic_2d}, we set parameters as: $a_1=a_2=1$, $b_1=b_2=0.25$, $l_1=l_2=1$, $T = 0.01$, $\replythr{N_t} = 8$, $R = 17.61$, $n_1=n_2=7$, and $m=4$, and apply the select oracle implemented as in Eq.~\eqref{eq:select under periodic} for each dimension.
By applying the central scheme and the exponential scheme, numerical results are displayed in Fig.~\ref{fig:adv_diff_homo_periodic_2d}, while the relative errors are listed in Table~\ref{tab:adv_diff_homo_periodic_2d_error}.

Consistent with the $1$-dimensional case in Experiment~\ref{experiment:adv_diff_homo_periodic}, both flux construction schemes maintain comparable performance.
While the central scheme outperforms the exponential scheme slightly in the $L_\infty$-norm, the exponential scheme is marginally better in the $L_1$- and $L_2$-norms; overall, their differences remain negligible.

\begin{figure}[htbp]
    \centering
    \begin{tabular}{m{0.32\textwidth} m{0.32\textwidth} m{0.32\textwidth}}
        \multirow{2}{*}{\subfigure[Exact solution]{\includegraphics[width=\linewidth]{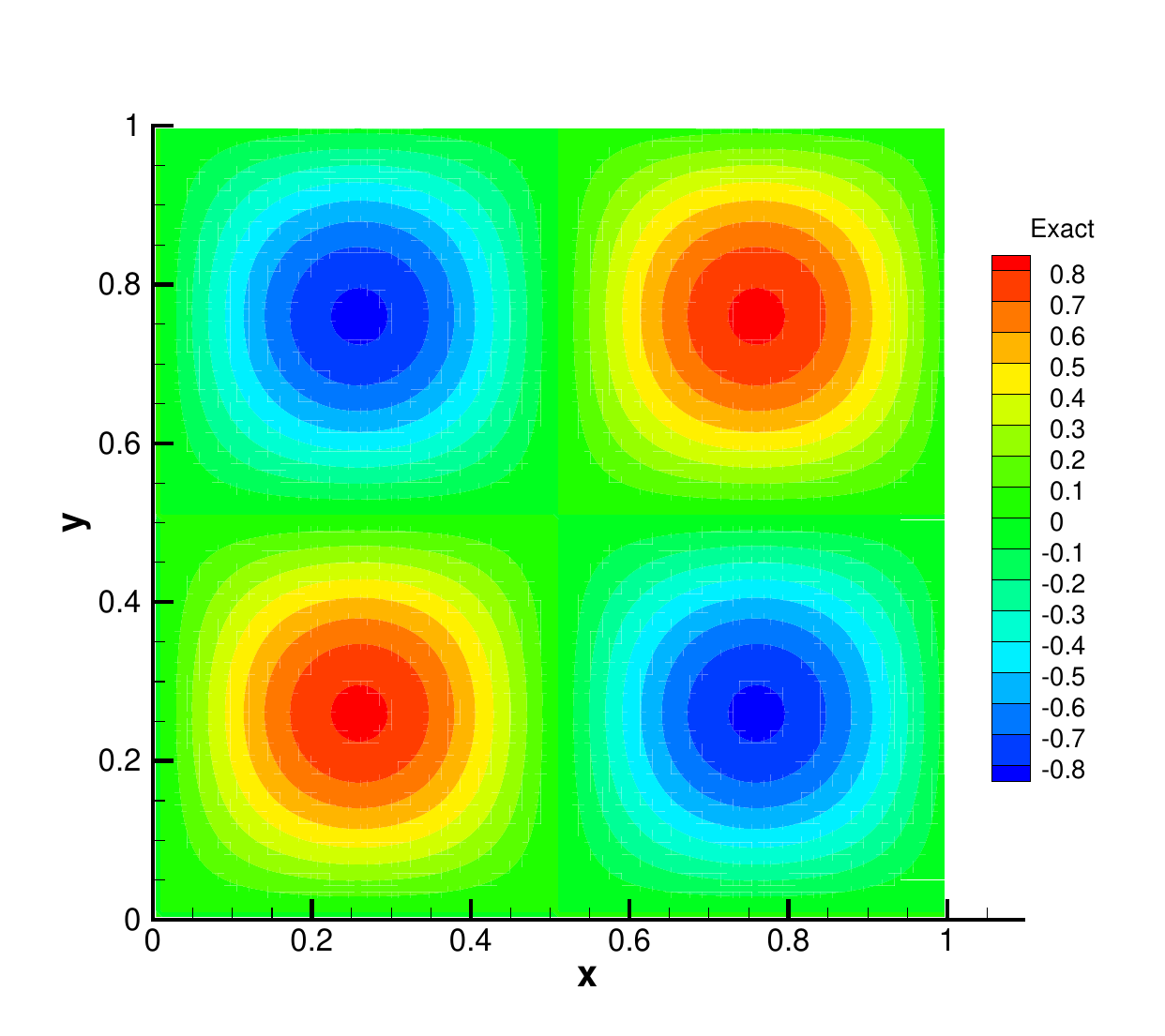}}} &
        \subfigure[Central scheme]{\includegraphics[width=\linewidth]{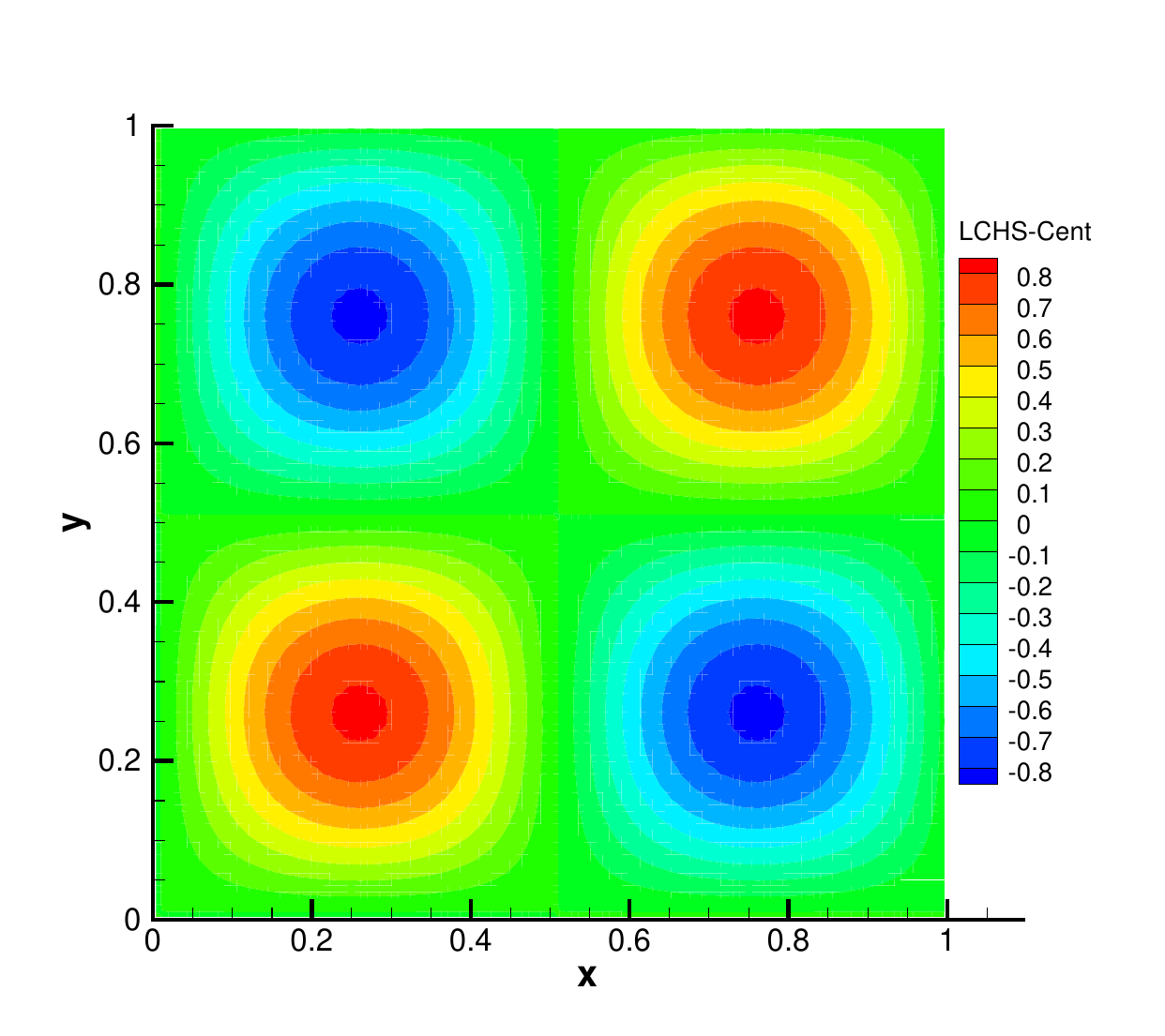}} &
        \subfigure[Absolute error of central scheme]{\includegraphics[width=\linewidth]{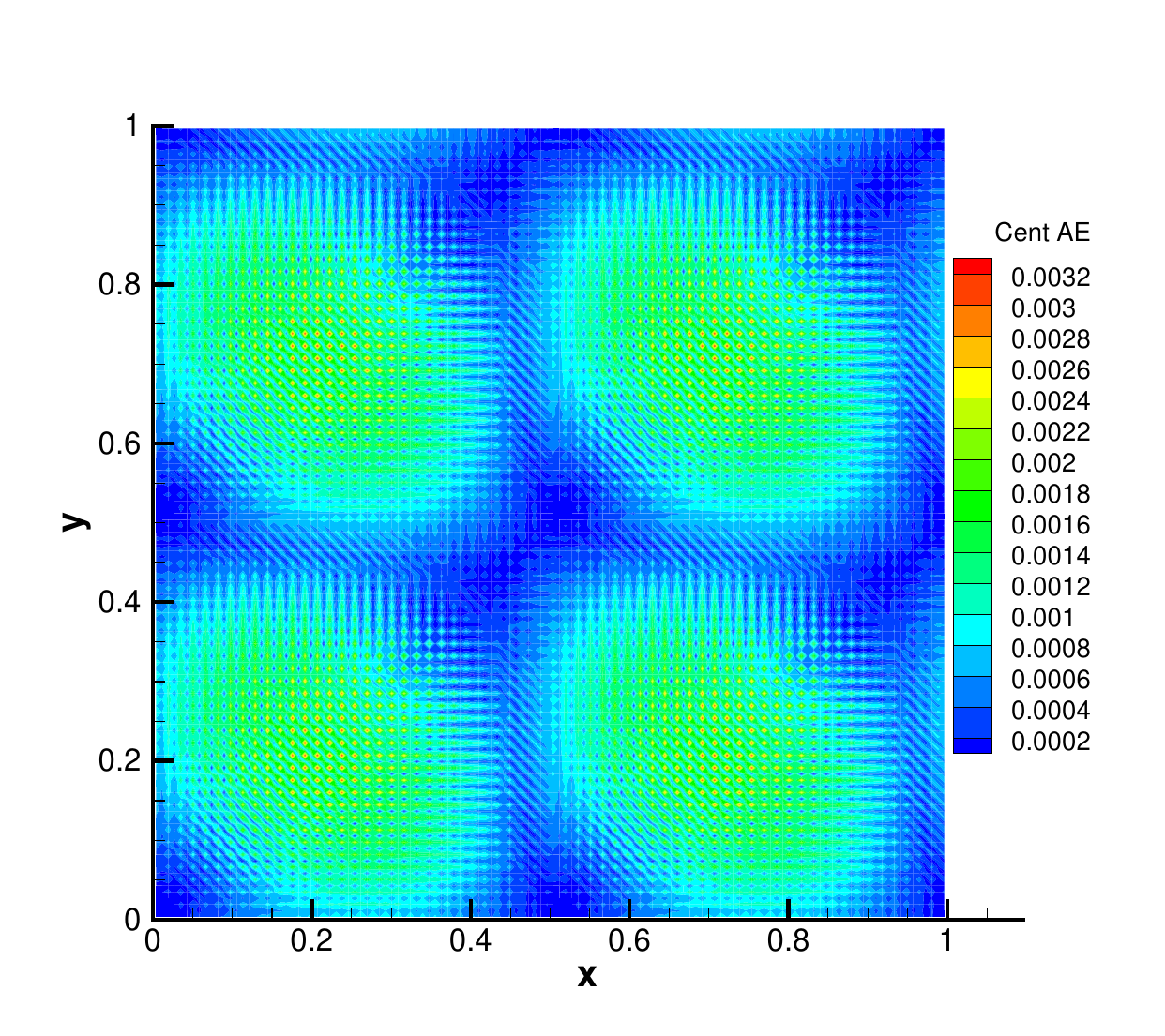}} \\
        &
        \subfigure[Exponential scheme]{\includegraphics[width=\linewidth]{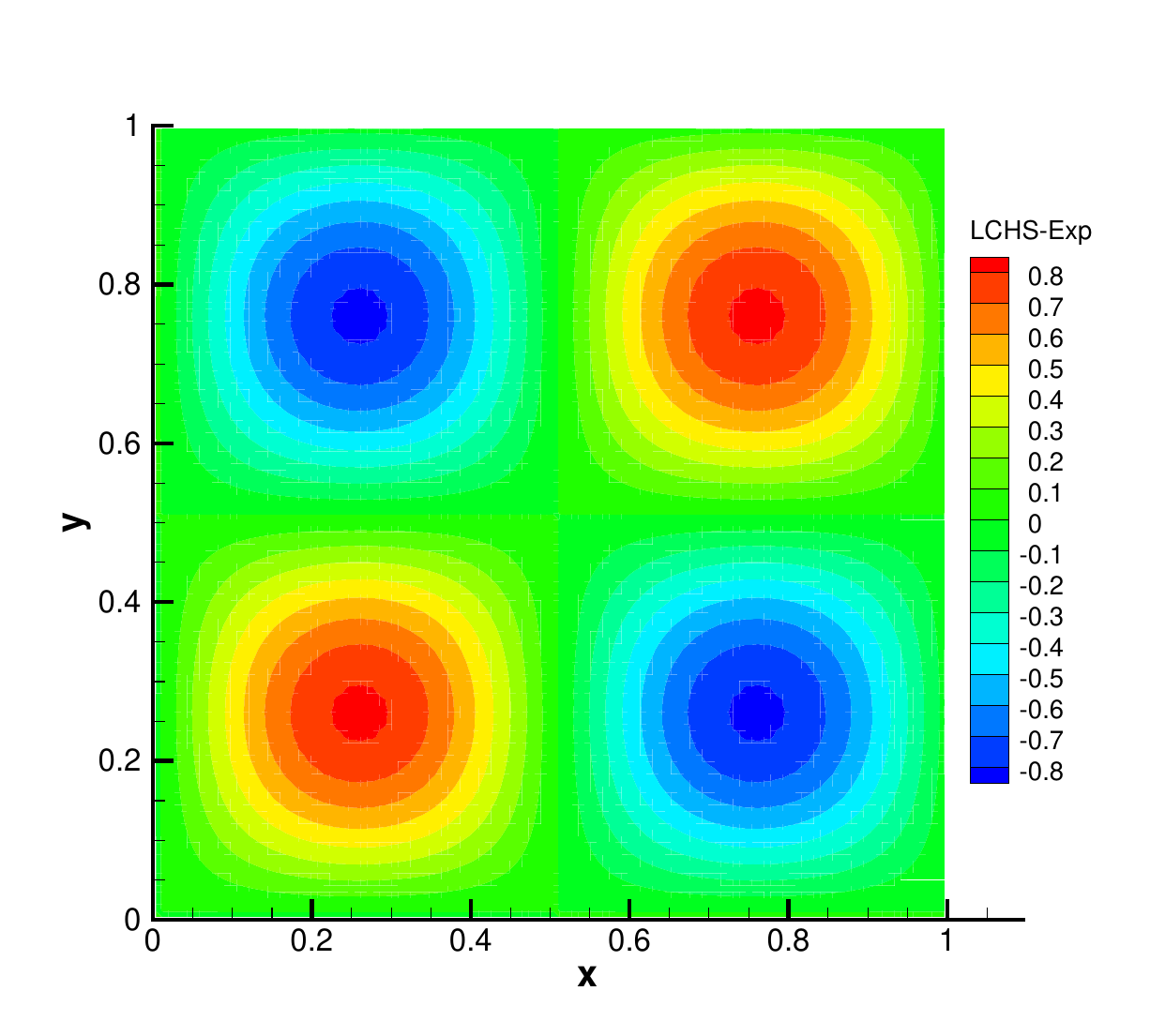}} &
        \subfigure[Absolute error of exponential scheme]{\includegraphics[width=\linewidth]{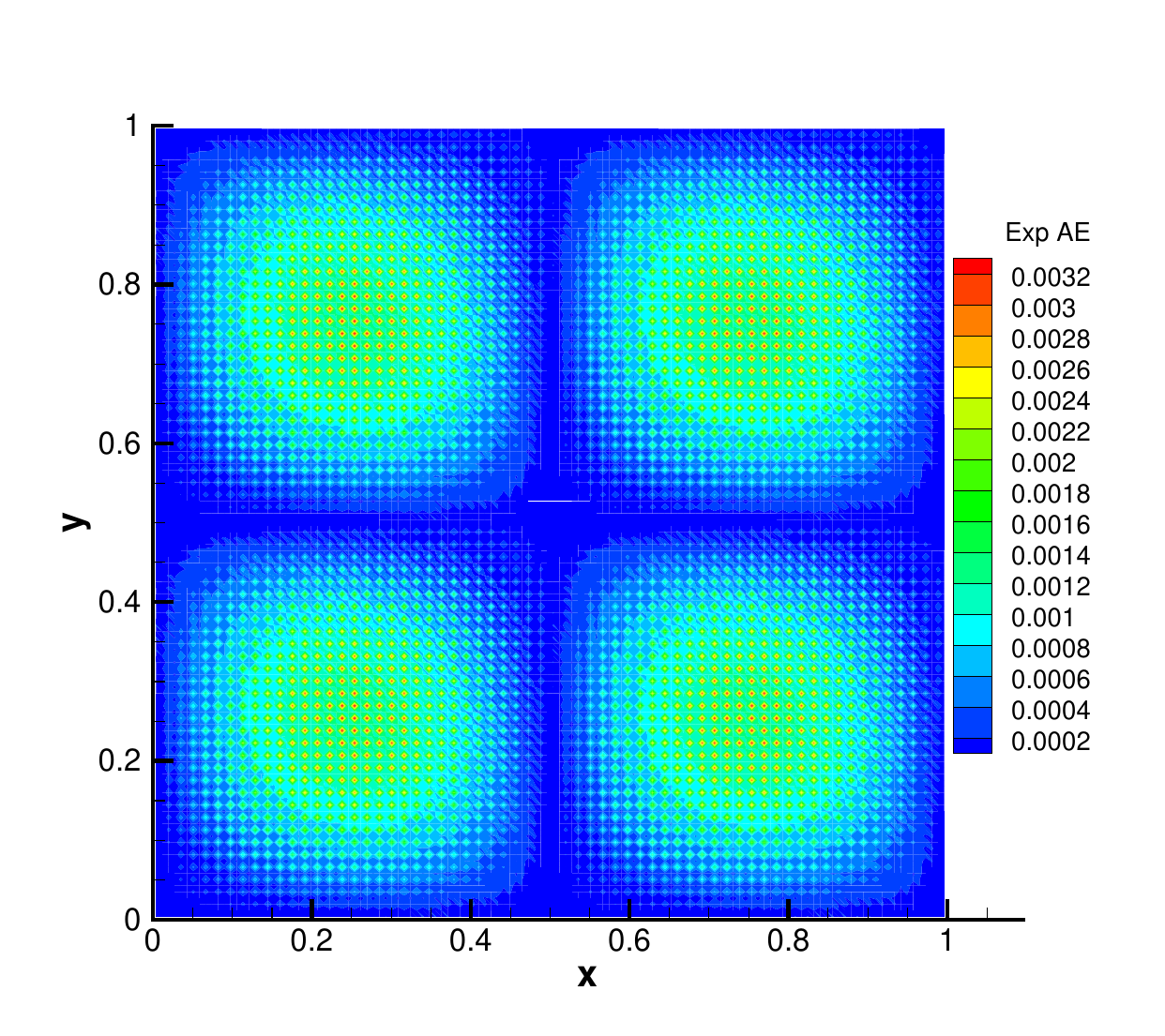}}
    \end{tabular}
    \caption{Numerical results for Experiment~\ref{experiment:adv_diff_homo_periodic_2d} (homogeneous $2$D advection-diffusion equation with periodic BCs).}
    \label{fig:adv_diff_homo_periodic_2d}
\end{figure}

\begin{table}[htbp]
\centering
\begin{minipage}[t]{0.48\textwidth}
    \centering
    \caption{Relative errors for Experiment~\ref{experiment:adv_diff_homo_periodic_2d}}
    \begin{tabular}{c c c c}
        \toprule
        Scheme & $L_1$-norm & $L_2$-norm & $L_\infty$-norm \\
        \midrule
        Central & 2.4825e-3 & 2.5783e-3 & 3.9348e-3 \\
        Exponential & 2.1343e-3 & 2.4494e-3 & 4.2411e-3 \\
        \bottomrule
    \end{tabular}
    \label{tab:adv_diff_homo_periodic_2d_error}
\end{minipage}
\end{table}

\subsection{Discussion}

Based on the comprehensive numerical results presented above, we distill the key observations and insights as follows.

First, numerical results demonstrate that the relative error is on the order of $10^{-3}$ in most test cases, which validates the effectiveness of quantum circuits designed based on the LCHS method.
For homogeneous cases, encoding the system state with $n=8$ achieves satisfactory performance across most scenarios.
Notably, an excessive number of ancilla qubits is unnecessary for encoding the coefficients of the LCHS method, and $m=4$ suffices in most cases.
For inhomogeneous cases, we similarly find that an excessive number of ancilla qubits is not required for encoding the outer LCU coefficients, with $m_o=3$ achieving reliable performance.

\replythr{Beyond the resource observations, four cross-experiment findings are worth noting. First, the two-dimensional tests yield relative errors of the same order as their one-dimensional counterparts, confirming that the multidimensional assembly via the global select oracle preserves the per-dimension accuracy. Second, across all advection-diffusion tests the central and exponential flux schemes produce nearly identical errors, indicating that the framework is insensitive to the choice of flux construction. Third, the mild error growth with increasing $n$ observed in the diffusion-dominated tests originates from the growing quadrature and Trotter contributions under the enlarged $\|L\|$, rather than from any instability, and is controllable by refining $\Delta r$ and increasing $N_t$ (cf.~Experiment~\ref{experiment:diff_homo_dirichlet}). Fourth, the number of Trotter steps $N_t$ does not need to be set excessively large as suggested by the theoretical analysis: $N_t \leq 16$ suffices for the tested parameter ranges, and $N_t=1$ is used in most experiments.}
\replytwo{This gap between theory and practice is expected: the bounds in Theorems~\ref{th:r of homo term}--\ref{th:r of inhomo term} are sufficient worst-case operator-norm estimates that allocate an $\mathcal{O}(\varepsilon)$ budget to the Trotter contribution alone, whereas the reported tests use short evolution times, smooth single-mode data, and moderate LCHS truncation boundaries, for which the spatial-discretization and LCHS quadrature errors reach their floors before the worst-case Trotter bound becomes active. We emphasize that this is an empirical observation restricted to the tested parameter range: increasing $R$ and $m$ reduces the LCHS truncation and quadrature errors, after which the $\mathcal{O}(N_t^{-2})$ error of the second-order product formula becomes visible and $N_t$ must be increased; the same occurs for stronger advection, longer evolution times, finer grids (the Trotter error bounds grow as $b^3/h^6$ or $a^3/h^3$), or less smooth data. Consistently, in the advection-dominated tests (Experiments~\ref{experiment:adv_homo_periodic}, \ref{experiment:adv_diff_homo_periodic}, \ref{experiment:adv_inhomo_periodic} and~\ref{experiment:adv_diff_homo_periodic_2d}), where the $a^3/h^3$ contribution dominates the Trotter bound, $N_t = 8$--$16$ is already required.}

However, \replythr{oscillatory behavior is prevalent in the $1$-dimensional experiments; it} most likely originates \replythr{from the Trotter error, whose bounds grow as $b^3/h^6$ or $a^3/h^3$,} and can be suppressed by increasing $N_t$.
Overshoot and dissipation are also observed in several experiments\replythr{; both originate from the choice of the truncation boundary $R$: a smaller $R$ concentrates the quadrature nodes on the large-amplitude region of the kernel, yielding results larger than the reference solution (overshoot), whereas a larger $R$ spreads the nodes toward the low-amplitude tail, yielding smaller ones (dissipation)}.
In practice, the choice $R = \mathcal{O}(2^m)$\replythr{, which balances these two opposing deviations,} is an empirical observation rather than a theoretical result; this setting delivers strong performance, particularly for $m\geq4$.
\replythr{A node distribution adapted to the kernel decay---retaining the leading integration nodes with the largest weights---may further improve the numerical performance.}

\section{Conclusion}\label{sec:conclusion}

We present a \replytwo{systematic and explicit} quantum circuit framework for solving advection-diffusion equations with boundary conditions, based on the LCHS method.
We elaborate on the complete workflow of the proposed framework: commencing with the transformation of advection-diffusion equations into a source-term-included ODE system via the FVM and diverse flux construction schemes, and proceeding to the design of tailored quantum circuits for different boundary conditions.
\replytwo{The design features new endpoint-projector circuits for Robin boundary elimination and a boundary- and flux-dependent selector assembly, all specified at the elementary-gate level.}
The error analysis and gate complexity of the proposed quantum circuits are discussed in detail\replytwo{, with LCU-weighted product-formula bounds preserving the LCHS kernel decay and a boundary-dependent dissipativity condition}, and these analyses reveal the quantum advantages of our approach over classical methods in high-dimensional scenarios.
The constructed circuits have been verified through simulation on the \replythr{QPanda2} digital emulator of fault-tolerant quantum computers, utilizing 7--12 qubits for spatial encoding and 3--5 qubits for inner and outer LCU coefficients.

Numerical results from the proposed quantum circuits\replythr{, covering $1$- and $2$-dimensional problems under Dirichlet, Neumann, Robin, and periodic boundary conditions with comparable accuracy across dimensions,} exhibit excellent agreement with analytical solutions (achieving $\leq 2\%$ relative error in almost all cases), thereby validating the efficacy of our circuit design.
In addition, the proposed framework is broadly applicable to various spatial discretizations and numerical schemes, and it generalizes naturally to other linear PDEs that incorporate boundary conditions.
This flexibility stems from the fact that our proposed circuit architecture focuses solely on the coefficient matrix of the ODE system, rather than imposing any constraints on the underlying discretization.

Nevertheless, several inherent limitations of the current quantum circuits merit further improvement.
First, the boundary conditions investigated herein are limited to time-independent cases, which restricts the framework’s applicability to dynamic scenarios.
Second, although multi-controlled gates can be decomposed into single-qubit gates and CNOT gates, the gate complexity (quantified in Theorem~\ref{th:gate complexity}) renders practical execution of these circuits on near-term NISQ~\cite{preskill2018quantum} devices highly challenging.
Third, while oscillation, overshoot, and dissipation phenomena are observed in the numerical experiments\replythr{~and have been preliminarily attributed to the Trotter error and the truncation boundary $R$,} effective strategies for suppressing or mitigating these phenomena await further exploration.
\replythr{Fourth, the coefficients considered are spatially constant and the grids are restricted to Cartesian coordinates, which limits the applicability of the proposed framework.}

To address the aforementioned limitations and further advance the practicality of quantum PDE simulations, future research directions will focus on:
(1) extending the framework to accommodate time-dependent boundary conditions;
(2) simplifying the quantum circuit architecture and compressing its depth to reduce gate complexity; (3) \replythr{developing targeted suppression or mitigation strategies for the oscillation, overshoot, and dissipation phenomena based on the mechanisms identified above; and (4) extending the present constant-coefficient Cartesian construction to spatially varying coefficients and curvilinear grids via position-dependent oracles or sparse block encodings}.

\section*{Acknowledgments}

We would like to express our gratitude to Prof. Jin-Peng Liu and Dr. Zhaoyuan Meng for their valuable suggestions related to this article.

\appendix

\section{Maximum and minimum eigenvalues of matrix $B$}\label{sec:eigenvalue_of_B}
Define the tridiagonal matrix
\begin{equation}
    B(\phi, \mu_0, \mu_1) :=
    \begin{bmatrix}
    \mu_0 & e^{i \phi} & & \\
    e^{-i \phi} & 0 & e^{i \phi} & & \\
    & \ddots & \ddots & \ddots & \\
    & & e^{-i \phi} & 0 & e^{i \phi} & \\
    & & & e^{-i \phi} & \mu_1
    \end{bmatrix}_{N\times N},
\end{equation}
and consider its maximum and minimum eigenvalues.
A diagonal phase similarity transformation shows that $B(\phi, \mu_0, \mu_1)$ shares the spectrum of $B(0, \mu_0, \mu_1)$. Hence, without loss of generality, we denote $B = B(0, \mu_0, \mu_1)$ and focus on its spectrum.

Let $\lambda$ be an eigenvalue of $B$ with corresponding eigenvector $\mathbf{x} = (x_1, x_2, \dots, x_N)^\mathrm{T}$.
The component-wise expansion of $B\mathbf{x} = \lambda\mathbf{x}$ yields
\begin{equation}\label{eq:recur_sys}
\begin{aligned}
    x_2 &= (\lambda - \mu_0)x_1, \quad k = 1; \\
    x_{k+1} &= \lambda x_k - x_{k-1}, \quad 2 \leq k \leq N-1; \\
    x_{N-1} &= (\lambda - \mu_1)x_N, \quad k = N.
\end{aligned}
\end{equation}
The middle relation is a second-order linear homogeneous recurrence, whose characteristic equation reads
\begin{equation}\label{eq:r_char}
    r^2 - \lambda r + 1 = 0.
\end{equation}
The form of solutions is determined by the discriminant $\lambda^2 - 4$:
\begin{enumerate}
    \item $\lambda^2 < 4$: complex conjugate roots $r_{1,2} = e^{\pm i\theta}$ with $\theta = \arccos(\lambda/2)$, giving trigonometric solutions;
    \item $\lambda^2 = 4$: repeated roots $r_{1,2} = \lambda/2 = \pm 1$, leading to polynomial solutions;
    \item $\lambda^2 > 4$: real roots $r_{1,2} = e^{\pm \xi}$ with $\xi = \text{arccosh}(\lambda/2)$, yielding hyperbolic solutions.
\end{enumerate}

\subsection{Trigonometric case}
Let $\lambda = 2\cos\theta$ ($\theta \in (0,\pi)$) and assume $x_k = C\cos((k-1)\theta) + D\sin((k-1)\theta)$.
Substituting these into the boundary conditions and solving the characteristic equation yields
\begin{equation}\label{eq:trigo_char}
    \sin((N+1)\theta) - (\mu_0+\mu_1) \sin(N\theta) + \mu_0\mu_1 \sin((N-1)\theta) = 0.
\end{equation}

Solving Eq.~\eqref{eq:trigo_char} yields the eigenvalues of matrix $B$. For specific boundary parameter pairs $(\mu_0, \mu_1)$, the eigenvalues are explicitly given by
\begin{equation}
    \lambda_k = 2\cos\left( \frac{M_k \cdot \pi}{L_N} \right), \quad k=1,\cdots,N,
\end{equation}
where $M_k$ and $L_N$ depend on $(\mu_0, \mu_1)$ as summarized in Table~\ref{tab:eigenvalue for trigonometric case}.
For all cases listed in Table~\ref{tab:eigenvalue for trigonometric case}, we have $\lambda_{\max}(B) \leq 2$ and $\lambda_{\min}(B) \geq -2$.

\begin{table}[htbp]
\centering
\caption{Parameters $M_k$ and $L_N$ for different boundary parameter pairs $(\mu_0, \mu_1)$.}
\label{tab:eigenvalue for trigonometric case}
\begin{tabular}{ccccccc}
    \toprule
    $(\mu_0,\mu_1)$ & $(-1,-1)$ & $(0,-1)$ & $(0,0)$ & $(1,-1)$ & $(1,0)$ & $(1,1)$ \\ \midrule
    $M_k$ & $k$ & $2k$ & $k$ & $2k-1$ & $2k-1$ & $k-1$ \\ \midrule
    $L_N$ & $N$ & $2N+1$ & $N+1$ & $2N$ & $2N+1$ & $N$ \\ \bottomrule
\end{tabular}
\end{table}

For the general case in which the boundary parameters $\mu_0, \mu_1 \in [-1, 1]$ are non-infinitesimal (and thus non-negligible), we leverage the perturbation method and the analytical results from \cite{willms2008_analytic, dafonseca2020_eigenpairs} to estimate the maximum and minimum eigenvalues of the matrix $B$, leading to
\begin{equation}\label{eq:trigo_eigenvalue}
\begin{aligned}
    \lambda_{\max} &= 2\cos\left( \frac{\pi}{N - 1 + \frac{1}{1-\mu_0} + \frac{1}{1-\mu_1}} \right) + \mathcal{O}\left( \frac{1}{N^4} \right), \\
    \lambda_{\min} &= -2\cos\left( \frac{\pi}{N - 1 + \frac{1}{1 + \mu_0} + \frac{1}{1 + \mu_1}} \right) + \mathcal{O}\left( \frac{1}{N^4} \right).
\end{aligned}
\end{equation}

\subsection{Polynomial case}\label{subsec:poly case}

For $\lambda = 2$, we derive the characteristic equation of $B$ as
\begin{equation}
    (N + 1) - N(\mu_0+\mu_1) + (N-1)\mu_0\mu_1 = 0,
\end{equation}
which yields
\begin{equation}\label{eq:lam=2}
    \mu_1 = \frac{(N+1) - N \mu_0}{N - (N-1)\mu_0}.
\end{equation}
If the point $(\mu_0, \mu_1)$ lies above the curve shown in Fig.~\ref{fig:mu1_vs_mu0(lambda=2)}, i.e., $\mu_1 > \frac{(N+1) - N \mu_0}{N - (N-1)\mu_0}$, then $\lambda_{\max} > 2$; otherwise, $\lambda_{\max} \leq 2$.

For $\lambda = -2$, the characteristic equation of $B$ is derived as
\begin{equation}
    (N + 1) + N(\mu_0+\mu_1) + (N-1)\mu_0\mu_1 = 0,
\end{equation}
which leads to
\begin{equation}\label{eq:lam=-2}
    \mu_1 = -\frac{(N+1) + N\mu_0}{N + (N-1)\mu_0}.
\end{equation}
If the point $(\mu_0, \mu_1)$ lies below the curve in Fig.~\ref{fig:mu1_vs_mu0(lambda=-2)}, i.e., $\mu_1 < -\frac{(N+1) + N\mu_0}{N + (N-1)\mu_0}$, then $\lambda_{\min} < -2$; otherwise, $\lambda_{\min} \geq -2$.

\begin{figure}[htbp]
\centering
\begin{minipage}[t]{0.33\textwidth}
    \centering
    \includegraphics[width=1\textwidth]{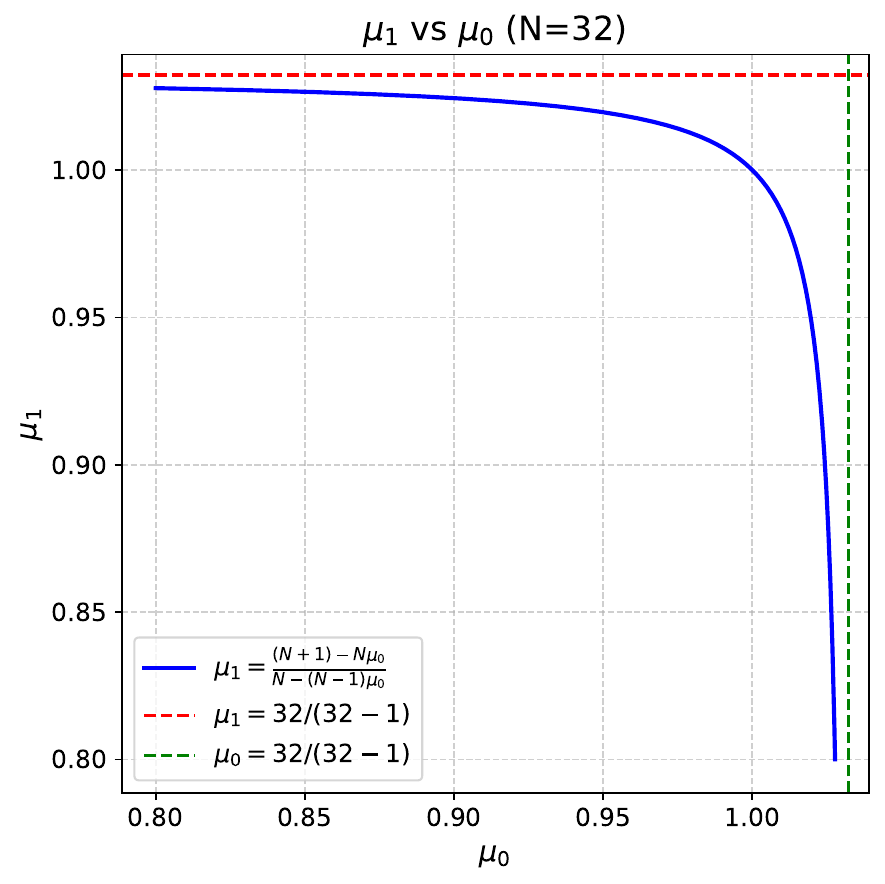}
    \caption{$\mu_1$ vs $\mu_0$($N=32$) of Eq.~\eqref{eq:lam=2}.}
    \label{fig:mu1_vs_mu0(lambda=2)}
\end{minipage}
\quad
\begin{minipage}[t]{0.33\textwidth}
    \centering
    \includegraphics[width=1\textwidth]{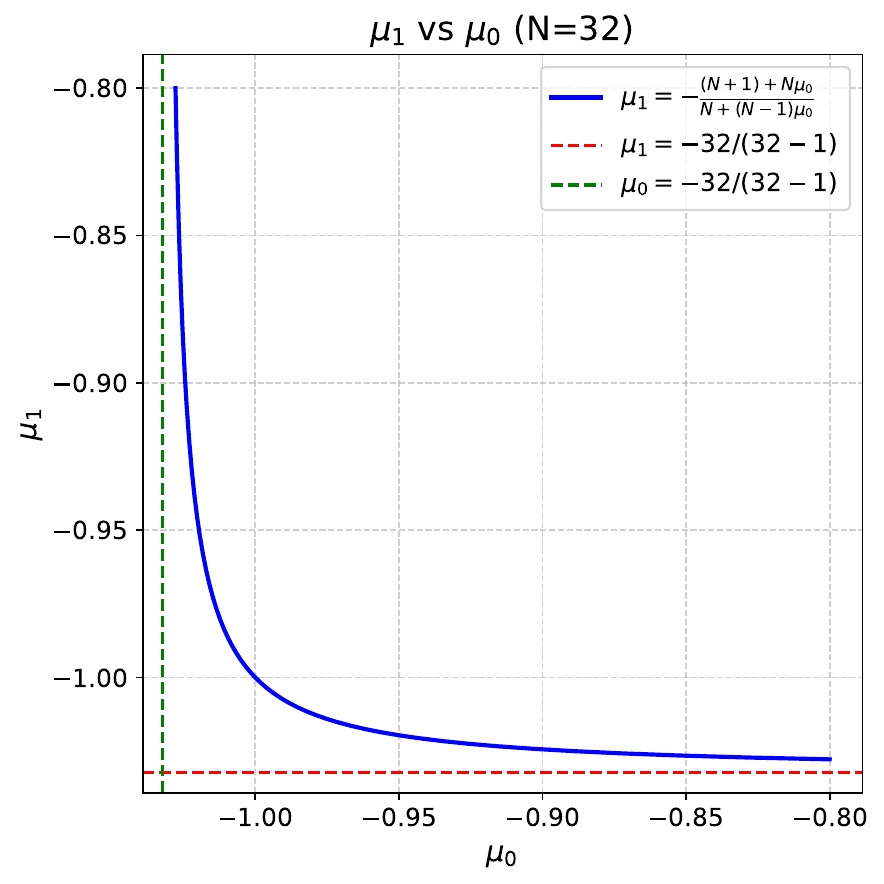}
    \caption{$\mu_1$ vs $\mu_0$($N=32$) of Eq.~\eqref{eq:lam=-2}.}
    \label{fig:mu1_vs_mu0(lambda=-2)}
\end{minipage}
\end{figure}

\subsection{Hyperbolic case}
For the maximum eigenvalue, we set $\lambda = 2\cosh\xi$ and assume $x_k = C\cosh((k-1)\xi) + D\sinh((k-1)\xi)$.
Substituting these into the boundary conditions yields the characteristic equation
\begin{equation}\label{eq:hyper_char1}
    \sinh((N+1)\xi) - (\mu_0+\mu_1) \sinh(N\xi) + \mu_0\mu_1 \sinh((N-1)\xi) = 0,
\end{equation}
which can be rewritten as
\begin{equation}\label{eq:hyper_char2}
    e^{(N-1)\xi} (e^{\xi} - \mu_0) (e^{\xi} - \mu_1) = e^{-(N-1)\xi} (e^{-\xi} - \mu_0) (e^{-\xi} - \mu_1).
\end{equation}

Dividing by $\sinh(N\xi)$ and taking the limit $N \to \infty$ (where $\frac{\sinh((N\pm1)\xi)}{\sinh(N\xi)} \approx e^{\pm \xi}$), the equation simplifies to:
\begin{equation}
    e^\xi - (\mu_0 + \mu_1) + \mu_0\mu_1 e^{-\xi} \approx 0.
\end{equation}
The solutions are $e^\xi \approx \mu_0$ or $e^\xi \approx \mu_1$.
Thus the maximum eigenvalue is determined by $\mu := \max \left\{ \mu_0, \mu_1 \right\} > 1$. Setting $e^\xi = \mu$, we find
\begin{equation}
    \lambda_{\max} = 2\cosh \xi = e^\xi + e^{-\xi}  = \mu + \frac{1}{\mu} + \mathcal{O} \left( \left( \mu + \frac{1}{\mu} \right) \mu^{-2N} \right) \approx \mu + \frac{1}{\mu},
\end{equation}
where the approximation becomes exact as $N \to \infty$.

Similarly, for the minimum eigenvalue we set $\lambda = -2\cosh\xi$
and let $\nu = \min\{\mu_0,\mu_1\} < -1$.
With $e^\xi = -\nu$, the minimum eigenvalue satisfies
\begin{equation}
    \lambda_{\min} = -2\cosh\xi = \nu + \frac{1}{\nu} + \mathcal{O} \left( \left(|\nu|+\frac1{|\nu|}\right)|\nu|^{-2N} \right)
    \approx \nu + \frac{1}{\nu},
\end{equation}
which also becomes exact in the limit $N\to\infty$.

\section{Trotter error analysis}\label{sec:trotter error}
\begin{lemma}[\cite{Childs2021_Theory}]
Given a general Hamiltonian $\mathcal{H} = \sum_{\gamma=1}^{\Gamma} H_{\gamma}$, let $\mathcal{H}_2(t) = \prod_{\gamma=\Gamma}^1 e^{-i H_\gamma \frac{t}{2}} \prod_{\gamma=1}^\Gamma e^{-i H_\gamma \frac{t}{2}}$ be the second-order Trotter-Suzuki formula. Then, the additive Trotter error can be bounded as
\begin{equation}
    \|\mathcal{H}_2(t) - e^{-i\mathcal{H}t}\|
    \leq \frac{t^3}{12} \sum_{\gamma_1=1}^\Gamma \left\| \left[ \sum_{\gamma_3=\gamma_1+1}^\Gamma H_{\gamma_3}, \left[ \sum_{\gamma_2=\gamma_1+1}^\Gamma H_{\gamma_2}, H_{\gamma_1} \right] \right] \right\|
    + \frac{t^3}{24} \sum_{\gamma_1=1}^\Gamma \left\| \left[ H_{\gamma_1}, \left[ H_{\gamma_1}, \sum_{\gamma_2=\gamma_1+1}^\Gamma H_{\gamma_2} \right] \right] \right\|.
\end{equation}
\end{lemma}

Recall
\begin{equation}
    s_{j,n}^- = I^{\otimes (n-j)} \otimes \sigma_{01} \otimes \sigma_{10}^{\otimes (j-1)},\quad
    s_{j,n}^+ = I^{\otimes (n-j)} \otimes \sigma_{10} \otimes \sigma_{01}^{\otimes (j-1)}.
\end{equation}
Direct calculation of the first-order commutators gives
\begin{equation}
\begin{aligned}
    \left[s_{j,n}^-, s_{j',n}^+ \right] &= \delta_{j j'} I^{\otimes(n-j)} \otimes \left( \sigma_{00} \otimes \sigma_{11}^{\otimes(j-1)} - \sigma_{11} \otimes \sigma_{00}^{\otimes(j-1)} \right) , \\
    \left[s_{j,n}^-, s_{j',n}^- \right] &= 0, \quad j\geq j'>1, \quad \left[s_{1,n}^-, s_{1,n}^- \right] = 0, \\
    \left[s_{1,n}^-, s_{j,n}^- \right] &= I^{\otimes(n-j)} \otimes \sigma_{01} \otimes \sigma_{10}^{\otimes(j-2)} \otimes Z, \quad j > 1, \\
    \left[s_{j,n}^+, s_{j',n}^+ \right] &= 0, \quad j\geq j'>1, \quad \left[s_{1,n}^+, s_{1,n}^+ \right] = 0, \\
    \left[s_{1,n}^+, s_{j,n}^+ \right] &= -I^{\otimes(n-j)} \otimes \sigma_{10} \otimes \sigma_{01}^{\otimes(j-2)} \otimes Z, \quad j > 1, \\
    \left[ \sigma_{10}^{\otimes n}, \sigma_{01}^{\otimes n} \right] &= \sigma_{11}^{\otimes n} - \sigma_{00}^{\otimes n}, \\
    \left[ s_{j,n}^-, \sigma_{01}^{\otimes n} \right] &= 0, \quad \left[ s_{j,n}^-, \sigma_{10}^{\otimes n} \right] = \delta_{j1} \sigma_{10}^{\otimes(n-1)} \otimes Z  , \\
    \left[ s_{j,n}^+, \sigma_{10}^{\otimes n} \right] &= 0, \quad \left[ s_{j,n}^+, \sigma_{01}^{\otimes n} \right] = -\delta_{j1} \sigma_{01}^{\otimes(n-1)} \otimes Z, \\
    \left[ \sigma_{00}^{\otimes n}, \sigma_{11}^{\otimes n} \right] &= 0, \\
    \left[ \sigma_{00}^{\otimes n}, s_{j,n}^- \right] &= \delta_{j1} \sigma_{00}^{\otimes(n-1)} \otimes \sigma_{01}, \quad \left[ \sigma_{00}^{\otimes n}, s_{j,n}^+ \right] = -\delta_{j1} \sigma_{00}^{\otimes(n-1)} \otimes \sigma_{10}, \\
    \left[ \sigma_{11}^{\otimes n}, s_{j,n}^- \right] &= -\delta_{j1} \sigma_{11}^{\otimes(n-1)} \otimes \sigma_{01}, \quad \left[ \sigma_{11}^{\otimes n}, s_{j,n}^+ \right] = -\delta_{j1} \sigma_{11}^{\otimes(n-1)} \otimes \sigma_{10}.
\end{aligned}
\end{equation}

Recall the Hamiltonians employed in constructing the coefficient matrix, defined as
\begin{equation}
    \mathcal{H}_1(\lambda, n) =  e^{i\lambda} s_{1,n}^- + e^{-i\lambda} s_{1,n}^+, \quad
    \mathcal{H}_2(\lambda, n) = \sum_{j=2}^{n} \left(e^{i\lambda} s_{j,n}^- + e^{-i\lambda} s_{j,n}^+\right), \quad
    \mathcal{H}_3(\lambda, n) = e^{i \lambda}\sigma_{10}^{\otimes n} + e^{-i \lambda}\sigma_{01}^{\otimes n}.
\end{equation}
Direct calculation of the first-order commutators of these operators yields
\begin{equation}\label{eq:first order commutator relationships}
\begin{aligned}
    \left[ \mathcal{H}_1(\lambda_1, n), \mathcal{H}_1(\lambda_2, n) \right]
    &= 2i \sin(\lambda_1-\lambda_2) I^{\otimes(n-1)} \otimes Z, \\
    \left[ \mathcal{H}_1(\lambda_1, n), \mathcal{H}_2(\lambda_2, n) \right]
    &= \sum_{j=2}^n I^{\otimes(n-j)} \otimes \left( e^{i(\lambda_1+\lambda_2)} \sigma_{01} \otimes \sigma_{10}^{\otimes(j-2)} - e^{-i(\lambda_1+\lambda_2)} \otimes \sigma_{10} \otimes \sigma_{01}^{\otimes(j-2)} \right) \otimes Z, \\
    \left[ \mathcal{H}_1(\lambda_1, n), \mathcal{H}_3(\lambda_2, n) \right]
    &= \left( e^{i(\lambda_1+\lambda_2)} \sigma_{10}^{\otimes(n-1)} - e^{-i(\lambda_1+\lambda_2)} \sigma_{01}^{\otimes(n-1)} \right) \otimes Z, \\
    \left[ \mathcal{H}_2(\lambda_1, n), \mathcal{H}_1(\lambda_2, n) \right]
    &= -\left[ \mathcal{H}_1(\lambda_2, n), \mathcal{H}_2(\lambda_1, n) \right]
    = -\left[ \mathcal{H}_1(\lambda_1, n), \mathcal{H}_2(\lambda_2, n) \right], \\
    \left[ \mathcal{H}_2(\lambda_1, n), \mathcal{H}_2(\lambda_2, n) \right]
    &= 2i \sin(\lambda_1-\lambda_2) \sum_{j=2}^n I^{\otimes(n-j)} \otimes \left( \sigma_{00} \otimes \sigma_{11}^{\otimes(j-1)} - \sigma_{11} \otimes \sigma_{00}^{\otimes(j-1)} \right), \\
    \left[ \mathcal{H}_2(\lambda_1, n), \mathcal{H}_3(\lambda_2, n) \right]
    &= 0, \\
    \left[ \mathcal{H}_3(\lambda_1, n), \mathcal{H}_3(\lambda_2, n) \right]
    &= 2i \sin(\lambda_1-\lambda_2) \left( \sigma_{11}^{\otimes n} - \sigma_{00}^{\otimes n} \right), \\
    \left[ \sigma_{00}^{\otimes n}, \mathcal{H}_1( \lambda, n) \right]
    &= \sigma_{00}^{\otimes(n-1)} \otimes \left( e^{i \lambda} \sigma_{01} - e^{-i \lambda} \sigma_{10} \right), \\
    \left[ \sigma_{00}^{\otimes n}, \mathcal{H}_2(\lambda, n) \right]
    &= 0, \\
    \left[ \sigma_{11}^{\otimes n},
    \mathcal{H}_1(\lambda, n) \right]
    &= -\sigma_{11}^{\otimes(n-1)} \otimes \left( e^{i \lambda} \sigma_{01} - e^{-i \lambda} \sigma_{10} \right), \\
    \left[ \sigma_{11}^{\otimes n}, \mathcal{H}_2(\lambda, n) \right]
    &= 0,
\end{aligned}
\end{equation}

Direct calculation of the second-order commutators of the operators that used in constructing the coefficient matrix under the Robin BCs gives
\begin{equation}
\begin{aligned}
    \left[ \mathcal{H}_1(\lambda_1, n), [\mathcal{H}_1(\lambda_1, n), \mathcal{H}_1(\lambda_2, n)] \right]
    =& -4i \sin(\lambda_1-\lambda_2) I^{\otimes (n-1)} \otimes \left( e^{i\lambda_1} \sigma_{01} - e^{-i\lambda_1} \sigma_{10} \right), \\
    \left[ \mathcal{H}_1(\lambda_1, n), [\mathcal{H}_1(\lambda_1, n), \mathcal{H}_2(\lambda_2, n)] \right]
    =& 2 e^{i\lambda_2} \sum_{j=2}^n I^{\otimes(n-j)} \otimes \sigma_{01} \otimes \sigma_{10}^{\otimes(j-2)} \otimes \begin{pmatrix}
        & -e^{i2\lambda_1} \\
        1 &
    \end{pmatrix} \\
    &+ 2 e^{-i\lambda_2} \sum_{j=2}^n I^{\otimes(n-j)} \otimes \sigma_{10} \otimes \sigma_{01}^{\otimes(j-2)} \otimes \begin{pmatrix}
        & 1 \\
        -e^{-i2\lambda_1} &
    \end{pmatrix}, \\
    \left[ \mathcal{H}_1(\lambda_1, n), [\mathcal{H}_2(\lambda_1, n), \mathcal{H}_1(\lambda_2, n)] \right]
    =& - \left[ \mathcal{H}_1(\lambda_1, n), [\mathcal{H}_1(\lambda_1, n), \mathcal{H}_2(\lambda_2, n)] \right], \\
    \left[ \mathcal{H}_1(\lambda_1, n), [\mathcal{H}_2(\lambda_1, n), \mathcal{H}_2(\lambda_2, n)] \right]
    =& 2i \sin(\lambda_1-\lambda_2)  \left( 2I^{\otimes (n-1)} - \sigma_{00}^{\otimes (n-1)} - \sigma_{11}^{\otimes (n-1)} \right) \otimes \left( e^{i \lambda_1} \sigma_{01} - e^{-i \lambda_1} \sigma_{10} \right),
\end{aligned}
\end{equation}

\begin{equation}
\begin{aligned}
    \left[ \mathcal{H}_1(\lambda_1, n), [\sigma_{00}^{\otimes n}, \mathcal{H}_1(\lambda_2, n)] \right]
    &= -2  \cos(\lambda_1 - \lambda_2) \sigma_{00}^{\otimes (n-1)} \otimes Z, \\
    \left[ \mathcal{H}_1(\lambda_1, n), [\sigma_{00}^{\otimes n}, \mathcal{H}_2(\lambda_2, n)] \right]
    &= 0, \\
    \left[ \mathcal{H}_1(\lambda_1, n), [\sigma_{11}^{\otimes n}, \mathcal{H}_1(\lambda_2, n)] \right]
    &= 2 \cos(\lambda_1 - \lambda_2) \sigma_{11}^{\otimes (n-1)} \otimes Z, \\
    \left[ \mathcal{H}_1(\lambda_1, n), [\sigma_{11}^{\otimes n}, \mathcal{H}_2(\lambda_2, n)] \right]
    &= 0,
\end{aligned}
\end{equation}

\begin{equation}
\begin{aligned}
    \left[ \mathcal{H}_2(\lambda_1, n), [\mathcal{H}_1(\lambda_1, n), \mathcal{H}_1(\lambda_2, n)] \right]
    =& 4i \sin(\lambda_1-\lambda_2) e^{i\lambda_1} \sum_{j=2}^{n} I^{\otimes (n-j)} \otimes \sigma_{01} \otimes \sigma_{10}^{\otimes (j-1)} \\
    &- 4i \sin(\lambda_1-\lambda_2) e^{-i\lambda_1} \sum_{j=2}^{n} I^{\otimes (n-j)} \otimes \sigma_{10} \otimes \sigma_{01}^{\otimes (j-1)}, \\
    \left[ \mathcal{H}_2(\lambda_1, n), [\mathcal{H}_1(\lambda_1, n), \mathcal{H}_2(\lambda_2, n)] \right]
    =& 2 e^{i(2\lambda_1+\lambda_2)} \sum_{j=1}^{n-2} I^{\otimes(n-2-j)} \otimes \sigma_{01} \otimes \sigma_{10}^{\otimes(j-1)} \otimes I \otimes \sigma_{10} \\
    & - e^{-i\lambda_2} \left( 2I^{\otimes (n-1)} - \sigma_{00}^{\otimes (n-1)} -  \sigma_{11}^{\otimes (n-1)} \right) \otimes \sigma_{10} \\
    & - e^{i\lambda_2} \left( 2I^{\otimes (n-1)} - \sigma_{00}^{\otimes (n-1)} -  \sigma_{11}^{\otimes (n-1)} \right) \otimes \sigma_{01} \\
    & + 2 e^{-i(2\lambda_1+\lambda_2)} \sum_{j=1}^{n-2} I^{\otimes(n-2-j)} \otimes \sigma_{10} \otimes \sigma_{01}^{\otimes(j-1)} \otimes I \otimes \sigma_{01}, \\
    \left[ \mathcal{H}_2(\lambda_1, n), [\mathcal{H}_2(\lambda_1, n), \mathcal{H}_1(\lambda_2, n)] \right]
    =& - \left[ \mathcal{H}_2(\lambda_1, n), [\mathcal{H}_1(\lambda_1, n), \mathcal{H}_2(\lambda_2, n)] \right], \\
    \left[ \mathcal{H}_2(\lambda_1, n), [\mathcal{H}_2(\lambda_1, n), \mathcal{H}_2(\lambda_2, n)] \right]
    =& -4i \sin(\lambda_1-\lambda_2) e^{i\lambda_1} \sum_{j=2}^n I^{\otimes (n-j)} \otimes \sigma_{01} \otimes \sigma_{10}^{\otimes (j-1)} \\
    &+ 4i \sin(\lambda_1-\lambda_2) e^{-i\lambda_1} \sum_{j=2}^n I^{\otimes (n-j)} \otimes \sigma_{10} \otimes \sigma_{01}^{\otimes (j-1)},
\end{aligned}
\end{equation}

\begin{equation}
\begin{aligned}
    \left[ \mathcal{H}_2(\lambda_1, n), [\sigma_{00}^{\otimes n}, \mathcal{H}_1(\lambda_2, n)] \right]
    &= 0, \\
    \left[ \mathcal{H}_2(\lambda_1, n), [\sigma_{00}^{\otimes n}, \mathcal{H}_2(\lambda_2, n)] \right]
    &= 0, \\
    \left[ \mathcal{H}_2(\lambda_1, n), [\sigma_{11}^{\otimes n}, \mathcal{H}_1(\lambda_2, n)] \right]
    &= 0, \\
    \left[ \mathcal{H}_2(\lambda_1, n), [\sigma_{11}^{\otimes n}, \mathcal{H}_2(\lambda_2, n)] \right]
    &= 0,
\end{aligned}
\end{equation}

\begin{equation}
\begin{aligned}
    \left[\sigma_{00}^{\otimes n}, [\mathcal{H}_1(\lambda_1, n), \mathcal{H}_1(\lambda_2, n)] \right]
    &= 0, \\
    \left[\sigma_{00}^{\otimes n}, [\mathcal{H}_1(\lambda_1, n), \mathcal{H}_2(\lambda_2, n)] \right]
    &= \sigma_{00}^{\otimes (n-2)} \otimes\left( e^{i(\lambda_1+\lambda_2)} \sigma_{01} + e^{-i(\lambda_1+\lambda_2)} \sigma_{10} \right) \otimes \sigma_{00}, \\
    \left[\sigma_{00}^{\otimes n}, [\mathcal{H}_2(\lambda_1, n), \mathcal{H}_1(\lambda_2, n)] \right]
    &= -\left[\sigma_{00}^{\otimes n}, [\mathcal{H}_1(\lambda_1, n), \mathcal{H}_2(\lambda_2, n)] \right], \\
    \left[\sigma_{00}^{\otimes n}, [\mathcal{H}_2(\lambda_1, n), \mathcal{H}_2(\lambda_2, n)] \right]
    &= 0,
\end{aligned}
\end{equation}

\begin{equation}
\begin{aligned}
    \left[ \sigma_{00}^{\otimes n}, [ \sigma_{00}^{\otimes n}, \mathcal{H}_1(\lambda, n) ] \right]
    &= \sigma_{00}^{\otimes (n-1)} \otimes\left( e^{i\lambda} \sigma_{01} + e^{-i\lambda} \sigma_{10} \right), \\
    \left[ \sigma_{00}^{\otimes n}, [ \sigma_{00}^{\otimes n}, \mathcal{H}_2(\lambda, n) ] \right]
    &= 0, \\
    \left[ \sigma_{00}^{\otimes n}, [ \sigma_{11}^{\otimes n}, \mathcal{H}_1(\lambda, n) ] \right]
    &= 0, \\
    \left[ \sigma_{00}^{\otimes n}, [ \sigma_{11}^{\otimes n}, \mathcal{H}_2(\lambda, n) ] \right]
    &= 0,
\end{aligned}
\end{equation}

\begin{equation}
\begin{aligned}
    \left[\sigma_{11}^{\otimes n}, [\mathcal{H}_1(\lambda_1, n), \mathcal{H}_1(\lambda_2, n)] \right]
    &= 0, \\
    \left[ \sigma_{11}^{\otimes n}, [\mathcal{H}_1(\lambda_1, n), \mathcal{H}_2(\lambda_2, n) ] \right]
    &= \sigma_{11}^{\otimes (n-2)} \otimes\left( e^{i(\lambda_1+\lambda_2)} \sigma_{01} + e^{-i(\lambda_1+\lambda_2)} \sigma_{10} \right) \otimes \sigma_{11}, \\
    \left[ \sigma_{11}^{\otimes n}, [ \mathcal{H}_2(\lambda_1, n), \mathcal{H}_1(\lambda_2, n) ] \right]
    &= -\left[ \sigma_{11}^{\otimes n}, [\mathcal{H}_1(\lambda_1, n), \mathcal{H}_2(\lambda_2, n) ] \right], \\
    \left[\sigma_{11}^{\otimes n}, [ \mathcal{H}_2(\lambda_1, n), \mathcal{H}_2(\lambda_2, n) ] \right]
    &= 0,
\end{aligned}
\end{equation}

\begin{equation}
\begin{aligned}
    \left[ \sigma_{11}^{\otimes n}, [\sigma_{00}^{\otimes n}, \mathcal{H}_1(\lambda, n)] \right]
    &= 0, \\
    \left[ \sigma_{11}^{\otimes n}, [\sigma_{00}^{\otimes n}, \mathcal{H}_2(\lambda, n)] \right]
    &= 0, \\
    \left[ \sigma_{11}^{\otimes n}, [\sigma_{11}^{\otimes n}, \mathcal{H}_1(\lambda, n)] \right]
    &= \sigma_{11}^{\otimes (n-1)} \otimes\left( e^{i\lambda} \sigma_{01} + e^{-i\lambda} \sigma_{10} \right), \\
    \left[ \sigma_{11}^{\otimes n}, [\sigma_{11}^{\otimes n}, \mathcal{H}_2(\lambda, n)] \right]
    &= 0.
\end{aligned}
\end{equation}

Direct calculation of the second-order commutators of the operators that used in constructing the coefficient matrix under the periodic BCs gives
\begin{equation}
\begin{aligned}
    \left[ \mathcal{H}_2(\lambda, n), [\mathcal{H}_3(\lambda, n), \mathcal{H}_1(\lambda, n)] \right]
    =& - e^{3i\lambda} \sigma_{10}^{\otimes(n-2)} \otimes I \otimes \sigma_{10}
    - e^{-3i\lambda} \sigma_{01}^{\otimes(n-2)} \otimes I \otimes \sigma_{01}, \\
    \left[ \mathcal{H}_3(\lambda, n), [\mathcal{H}_2(\lambda, n), \mathcal{H}_1(\lambda, n)] \right]
    =& - e^{3i\lambda} \sigma_{10}^{\otimes(n-2)} \otimes I \otimes \sigma_{10}
    - e^{-3i\lambda} \sigma_{01}^{\otimes(n-2)} \otimes I \otimes \sigma_{01}, \\
    \left[ \mathcal{H}_3(\lambda, n), [\mathcal{H}_3(\lambda, n), \mathcal{H}_1(\lambda, n)] \right]
    =& \left( \sigma_{00}^{\otimes(n-1)} + \sigma_{11}^{\otimes(n-1)} \right) \otimes \left( e^{i\lambda} \sigma_{01} + e^{-i\lambda} \sigma_{10} \right), \\
    \left[ \mathcal{H}_1(\lambda, n), [\mathcal{H}_1(\lambda, n), \mathcal{H}_3(\lambda, n)] \right]
    =& 2 e^{i\lambda} \sigma_{10}^{\otimes (n-1)} \otimes
    \begin{pmatrix}
        & -e^{i2\lambda} \\
        1 &
    \end{pmatrix}
    - 2 e^{-i\lambda} \sigma_{01}^{\otimes (n-1)} \otimes
    \begin{pmatrix}
        & 1 \\
        -e^{-i2\lambda} &
    \end{pmatrix}.
\end{aligned}
\end{equation}

\begin{remark}
To derive the calculation results presented above, we utilize the following key identity:
\begin{equation}
    \sum_{j=2}^{n} I^{\otimes (n-j)} \otimes \left( \sigma_{00} \otimes \sigma_{11}^{\otimes (j-2)} + \sigma_{11} \otimes \sigma_{00}^{\otimes (j-2)} \right)
    = 2 I^{\otimes (n-1)} - \left( \sigma_{00}^{\otimes (n-1)} + \sigma_{11}^{\otimes (n-1)} \right).
\end{equation}
\end{remark}

\subsection{Robin boundary conditions}\label{subsec: robin bc}
Denote $L = \gamma_1 \mathcal{H}_1(\lambda_1, n) + \gamma_1 \mathcal{H}_2(\lambda_1, n) + s_0 \sigma_{00}^{\otimes n} + s_1 \sigma_{11}^{\otimes n}$ with $B:=L /\gamma_1$, and $H = \gamma_2 \mathcal{H}_1(\lambda_2, n) + \gamma_2 \mathcal{H}_2(\lambda_2, n)$.
Without loss of generality, we assume that $\gamma_1, \gamma_2 \geq 0$ and define $s = \max\left\{|s_0|, |s_1|\right\}$.
Then we have
\begin{equation}
\begin{aligned}
    \left[L, H \right]
    =& \gamma_1 \gamma_2 \left[ \mathcal{H}_1(\lambda_1, n), \mathcal{H}_1(\lambda_2, n) \right]
    + \gamma_1 \gamma_2 \left[ \mathcal{H}_2(\lambda_1, n), \mathcal{H}_2(\lambda_2, n) \right] \\
    &+ \gamma_2 s_0 \left[ \sigma_{00}^{\otimes n}, \mathcal{H}_1(\lambda_2, n) \right]
    + \gamma_2  s_1 \left[ \sigma_{11}^{\otimes n}, \mathcal{H}_1(\lambda_2, n) \right] \\
    =& 2i  \gamma_1 \gamma_2 \sin(\lambda_1-\lambda_2) \sum_{j=1}^n I^{\otimes(n-j)} \otimes \left( \sigma_{00} \otimes \sigma_{11}^{\otimes(j-1)} - \sigma_{11} \otimes \sigma_{00}^{\otimes(j-1)} \right) \\
    &+ \gamma_2 \left( s_0 \sigma_{00}^{\otimes(n-1)} - s_1 \sigma_{11}^{\otimes(n-1)} \right) \otimes \left( e^{i \lambda_2} \sigma_{01} - e^{-i \lambda_2} \sigma_{10} \right),
\end{aligned}
\end{equation}
which indicates that $[L, H] \neq 0$ in general cases.

By canceling out the terms equal to zero and those that cancel each other, we obtain
\begin{equation}
\begin{aligned}
    \left[L, [L, H] \right]
    =& \gamma_1^2 \gamma_2 \left[ \mathcal{H}_1(\lambda_1, n), [\mathcal{H}_1(\lambda_1, n), \mathcal{H}_1(\lambda_2, n)] \right]
    + \gamma_1 \gamma_2 s_0 \left[ \mathcal{H}_1(\lambda_1, n), [\sigma_{00}^{\otimes n}, \mathcal{H}_1(\lambda_2, n)] \right] \\
    &+ \gamma_1^2 \gamma_2 \left[ \mathcal{H}_1(\lambda_1, n), [\mathcal{H}_2(\lambda_1, n), \mathcal{H}_2(\lambda_2, n)] \right]
    + \gamma_1 \gamma_2 s_1 \left[ \mathcal{H}_1(\lambda_1, n), [\sigma_{11}^{\otimes n}, \mathcal{H}_1(\lambda_2, n)] \right] \\
    &+ \gamma_2 s_0^2 \left[ \sigma_{00}^{\otimes n}, [\sigma_{00}^{\otimes n}, \mathcal{H}_1(\lambda_2, n)] \right]
    + \gamma_2 s_1^2 \left[ \sigma_{11}^{\otimes n}, [\sigma_{11}^{\otimes n}, \mathcal{H}_1(\lambda_2, n)] \right] \\
    =& - 2i \gamma_1^2 \gamma_2 \sin(\lambda_1-\lambda_2)  \left( \sigma_{00}^{\otimes (n-1)} + \sigma_{11}^{\otimes (n-1)} \right) \otimes \left( e^{i \lambda_1} \sigma_{01} - e^{-i \lambda_1} \sigma_{10} \right) \\
    &- 2 \gamma_1 \gamma_2 s_0 \cos(\lambda_1 - \lambda_2) \sigma_{00}^{\otimes (n-1)} \otimes Z
    + 2 \gamma_1 \gamma_2 s_1 \cos(\lambda_1 - \lambda_2) \sigma_{11}^{\otimes (n-1)} \otimes Z \\
    &+ \gamma_2 s_0^2 \sigma_{00}^{\otimes (n-1)} \otimes\left( e^{i\lambda_2} \sigma_{01} + e^{-i\lambda_2} \sigma_{10} \right)
    + \gamma_2 s_1^2 \sigma_{11}^{\otimes (n-1)} \otimes\left( e^{i\lambda_2} \sigma_{01} + e^{-i\lambda_2} \sigma_{10} \right),
\end{aligned}
\end{equation}
then direct calculation gives
\begin{equation}\label{eq:norm of [L,[L,H]]}
    \left\| \left[L, [L, H] \right] \right\|
    = \gamma_2 \sqrt{4\gamma_1^4\sin^2(\lambda_1-\lambda_2) + 4\gamma_1^2s^2 + s^4}
    \leq \gamma_2 \sqrt{4\gamma_1^4 + 4\gamma_1^2s^2 + s^4}
    =\gamma_2 \left( 2\gamma_1^2 +s^2 \right).
\end{equation}

Similarly, we obtain
\begin{equation}
\begin{aligned}
    \left[H, [H, L] \right]
    =& \gamma_1 \gamma_2^2 \left[ \mathcal{H}_1(\lambda_2, n), [\mathcal{H}_1(\lambda_2, n), \mathcal{H}_1(\lambda_1, n)] \right]
    + \gamma_2^2 s_0 \left[ \mathcal{H}_1(\lambda_2, n), [\mathcal{H}_1(\lambda_2, n), \sigma_{00}^{\otimes n}] \right] \\
    &+ \gamma_1 \gamma_2^2 \left[ \mathcal{H}_1(\lambda_2, n), [\mathcal{H}_2(\lambda_2, n), \mathcal{H}_2(\lambda_1, n)] \right]
    + \gamma_2^2 s_1 \left[ \mathcal{H}_1(\lambda_2, n), [\mathcal{H}_1(\lambda_2, n), \sigma_{11}^{\otimes n}] \right] \\
    =& 2i \gamma_1 \gamma_2^2 \sin(\lambda_1-\lambda_2)  \left( \sigma_{00}^{\otimes (n-1)} + \sigma_{11}^{\otimes (n-1)} \right) \otimes \left( e^{i \lambda_2} \sigma_{01} - e^{-i \lambda_2} \sigma_{10} \right) \\
    &+ 2 \gamma_2^2 s_0 \sigma_{00}^{\otimes (n-1)} \otimes Z - 2 \gamma_2^2 s_1 \sigma_{11}^{\otimes (n-1)} \otimes Z,
\end{aligned}
\end{equation}
then direct calculation gives
\begin{equation}\label{eq:norm of [H,[H,L]]}
    \left\| \left[H, [H, L] \right] \right\|
    = 2 \gamma_2^2 \sqrt{\gamma_1^2 \sin^2(\lambda_1-\lambda_2) + s^2}
    \leq 2 \gamma_2^2 \sqrt{\gamma_1^2 + s^2}.
\end{equation}

The first commutator for computing the Trotter error of $\exp(-iH\tau)$ is given by
\begin{equation}
\begin{aligned}
    \left[ \mathcal{H}_2(\lambda_2, n), [\mathcal{H}_2(\lambda_2, n), \mathcal{H}_1(\lambda_2, n)] \right]
    =& -2 e^{i3\lambda_2} \sum_{j=1}^{n-2} I^{\otimes(n-2-j)} \otimes \sigma_{01} \otimes \sigma_{10}^{\otimes(j-1)} \otimes I \otimes \sigma_{10} \\
    & + e^{-i\lambda_2} \left( 2I^{\otimes (n-1)} - \sigma_{00}^{\otimes (n-1)} -  \sigma_{11}^{\otimes (n-1)} \right) \otimes \sigma_{10} \\
    & + e^{i\lambda_2} \left( 2I^{\otimes (n-1)} - \sigma_{00}^{\otimes (n-1)} -  \sigma_{11}^{\otimes (n-1)} \right) \otimes \sigma_{01} \\
    & - 2 e^{-i3\lambda_2} \sum_{j=1}^{n-2} I^{\otimes(n-2-j)} \otimes \sigma_{10} \otimes \sigma_{01}^{\otimes(j-1)} \otimes I \otimes \sigma_{01},
\end{aligned}
\end{equation}
and through a similar analysis to that in Appendix~\ref{sec:eigenvalue_of_B}, we obtain
\begin{equation}\label{eq:norm of [H2,[H2,H1]]}
    \left\| \left[ \mathcal{H}_2(\lambda_2, n), [\mathcal{H}_2(\lambda_2, n), \mathcal{H}_1(\lambda_2, n)] \right] \right\|
    \leq 4.
\end{equation}

The corresponding second commutator is expressed as
\begin{equation}\label{eq:H1_H1_H2_in_H}
\begin{aligned}
    \left[ \mathcal{H}_1(\lambda_2, n), [\mathcal{H}_1(\lambda_2, n), \mathcal{H}_2(\lambda_2, n)] \right]
    =& 2 e^{i\lambda_2} \sum_{j=2}^n I^{\otimes(n-j)} \otimes \sigma_{01} \otimes \sigma_{10}^{\otimes(j-2)} \otimes \begin{pmatrix}
        & -e^{i2\lambda_2} \\
        1 &
    \end{pmatrix} \\
    &+ 2 e^{-i\lambda_2} \sum_{j=2}^n I^{\otimes(n-j)} \otimes \sigma_{10} \otimes \sigma_{01}^{\otimes(j-2)} \otimes \begin{pmatrix}
        & 1 \\
        -e^{i2\lambda_2} &
    \end{pmatrix}.
\end{aligned}
\end{equation}
The eigenvalues of the operator in Eq.~\eqref{eq:H1_H1_H2_in_H} are
\begin{equation}
    \lambda_k = \pm 4 \cos \left( \frac{k\pi}{2^{n-1}+1} \right), \quad k=1,2,\cdots,2^{n-2},
\end{equation}
where each eigenvalue has a multiplicity of 2. We thus obtain
\begin{equation}\label{eq:norm of [H1,[H1,H2]]}
    \left\| \left[ \mathcal{H}_1(\lambda_2, n), \left[ \mathcal{H}_1(\lambda_2, n), \mathcal{H}_2(\lambda_2, n) \right] \right] \right\|
    = 4 \cos \left( \frac{\pi}{2^{n-1}+1} \right)
    \leq 4.
\end{equation}

The first commutator for computing the Trotter error of $\exp(-iL\tau)$ is given by
\begin{equation}\label{eq:5 term of robin}
\begin{aligned}
    &\left[ \gamma_1 \mathcal{H}_2(\lambda_1, n) + s_0 \sigma_{00}^{\otimes n} + s_1 \sigma_{11}^{\otimes n}, [\gamma_1 \mathcal{H}_2(\lambda_1, n) + s_0 \sigma_{00}^{\otimes n} + s_1 \sigma_{11}^{\otimes n}, \gamma_1 \mathcal{H}_1(\lambda_1, n)] \right] \\
    =& -2 \gamma_1^3 e^{i3\lambda_1} \sum_{j=1}^{n-2} I^{\otimes(n-2-j)} \otimes \sigma_{01} \otimes \sigma_{10}^{\otimes(j-1)} \otimes I \otimes \sigma_{10} + 2 \gamma_1^3 e^{-i\lambda_1}  I^{\otimes (n-1)} \otimes \sigma_{10} \\
    & - 2 \gamma_1^3 e^{-i3\lambda_1} \sum_{j=1}^{n-2} I^{\otimes(n-2-j)} \otimes \sigma_{10} \otimes \sigma_{01}^{\otimes(j-1)} \otimes I \otimes \sigma_{01}
    + 2 \gamma_1^3 e^{i\lambda_1} I^{\otimes (n-1)} \otimes \sigma_{01} \\
    &+ \gamma_1^2 s_0 \sigma_{00}^{\otimes (n-2)} \otimes\left( e^{2i\lambda_1} \sigma_{01} + e^{-2i\lambda_1} \sigma_{10} \right) \otimes \sigma_{00}
    + \gamma_1 (s_0^2 - \gamma_1^2) \sigma_{00}^{\otimes (n-1)} \otimes\left( e^{i\lambda_1} \sigma_{01} + e^{-i\lambda_1} \sigma_{10} \right) \\
    &+ \gamma_1^2 s_1 \sigma_{11}^{\otimes (n-2)} \otimes\left( e^{2i\lambda_1} \sigma_{01} + e^{-2i\lambda_1} \sigma_{10} \right) \otimes \sigma_{11}
    + \gamma_1 (s_1^2 - \gamma_1^2) \sigma_{11}^{\otimes (n-1)} \otimes\left( e^{i\lambda_1} \sigma_{01} + e^{-i\lambda_1} \sigma_{10} \right).
\end{aligned}
\end{equation}
Through a similar analysis to that presented in Appendix~\ref{sec:eigenvalue_of_B}, we find that if $|\lambda|_{\max}(B) \leq 2$, it holds
\begin{equation}\label{eq:norm of [H2,[H2,H1]] in L when abs_lam <= 2}
    \left\| \left[ \gamma_1 \mathcal{H}_2(\lambda_1, n) + s_0 \sigma_{00}^{\otimes n} + s_1 \sigma_{11}^{\otimes n}, [\gamma_1 \mathcal{H}_2(\lambda_1, n) + s_0 \sigma_{00}^{\otimes n} + s_1 \sigma_{11}^{\otimes n}, \gamma_1 \mathcal{H}_1(\lambda_1, n)] \right] \right\|
    \leq 4 \gamma_1^3.
\end{equation}
Conversely, if $|\lambda|_{\max}(B) > 2$, then the sum $\mathcal{H}$ of the first four terms in Eq.~\eqref{eq:5 term of robin} satisfies $\left\| \mathcal{H}\right\| \leq 4\gamma_1^3$, and the sum $\mathcal{H}'$ of the last four terms in Eq.~\eqref{eq:5 term of robin} satisfies $\left\| \mathcal{H}'\right\| = \gamma_1 \sqrt{\gamma_1^4 + s^4 - \gamma_1^2 s^2}$, which implies
\begin{equation}\label{eq:norm of [H2,[H2,H1]] in L when abs_lam > 2}
    \left\| \left[ \gamma_1 \mathcal{H}_2(\lambda_1, n) + s_0 \sigma_{00}^{\otimes n} + s_1 \sigma_{11}^{\otimes n}, [\gamma_1 \mathcal{H}_2(\lambda_1, n) + s_0 \sigma_{00}^{\otimes n} + s_1 \sigma_{11}^{\otimes n}, \gamma_1 \mathcal{H}_1(\lambda_1, n)] \right] \right\|
    \leq 4 \gamma_1^3 + \gamma_1 \sqrt{\gamma_1^4 + s^4 - \gamma_1^2 s^2}.
\end{equation}

The corresponding second commutator is expressed as
\begin{equation}
\begin{aligned}
    &\left[ \gamma_1 \mathcal{H}_1(\lambda_1, n), [\gamma_1 \mathcal{H}_1(\lambda_1, n), \gamma_1 \mathcal{H}_2(\lambda_1, n) + s_0 \sigma_{00}^{\otimes n} + s_1 \sigma_{11}^{\otimes n}] \right] \\
    =& 2\gamma_1^3 e^{i\lambda_1} \sum_{j=2}^n I^{\otimes(n-j)} \otimes \sigma_{01} \otimes \sigma_{10}^{\otimes(j-2)} \otimes \begin{pmatrix}
        & -e^{i2\lambda_1} \\
        1 &
    \end{pmatrix} \\
    &+ 2\gamma_1^3 e^{-i\lambda_1} \sum_{j=2}^n I^{\otimes(n-j)} \otimes \sigma_{10} \otimes \sigma_{01}^{\otimes(j-2)} \otimes \begin{pmatrix}
        & 1 \\
        -e^{-i2\lambda_1} &
    \end{pmatrix} \\
    &+ 2 \gamma_1^2 s_0 \sigma_{00}^{\otimes (n-1)} \otimes Z - 2 \gamma_1^2 s_1 \sigma_{11}^{\otimes (n-1)} \otimes Z.
\end{aligned}
\end{equation}
Following a similar analysis to that for the first commutator, if $|\lambda|_{\max}(B) \leq 2$, it holds
\begin{equation}\label{eq:norm of [H1,[H1,H2]] in L when abs_lam <= 2}
    \left\| \left[ \gamma_1 \mathcal{H}_1(\lambda_1, n), [\gamma_1 \mathcal{H}_1(\lambda_1, n), \gamma_1 \mathcal{H}_2(\lambda_1, n) + s_0 \sigma_{00}^{\otimes n} + s_1 \sigma_{11}^{\otimes n}] \right] \right\|
    \leq 4 \gamma_1^3,
\end{equation}
else if $|\lambda|_{\max}(B) > 2$ we obtain
\begin{equation}\label{eq:norm of [H1,[H1,H2]] in L when abs_lam > 2}
    \left\| \left[ \gamma_1 \mathcal{H}_1(\lambda_1, n), [\gamma_1 \mathcal{H}_1(\lambda_1, n), \gamma_1 \mathcal{H}_2(\lambda_1, n) + s_0 \sigma_{00}^{\otimes n} + s_1 \sigma_{11}^{\otimes n}] \right] \right\|
    \leq 2 \gamma_1^3 \left( \frac{s}{\gamma_1} + \frac{\gamma_1}{s} \right).
\end{equation}

\subsection{Periodic boundary conditions}
Denote $L = \gamma_1 \left( \mathcal{H}_1(\lambda_1, n) + \mathcal{H}_2(\lambda_1, n) + \mathcal{H}_3(\lambda_1, n) \right), H = \gamma_2 \left( \mathcal{H}_1(\lambda_2, n) + \mathcal{H}_2(\lambda_2, n) + \mathcal{H}_3(\lambda_2, n) \right)$ and assume $\gamma_1,\gamma_2\geq0$; direct calculation gives $\left[ L,H \right] = 0$.
Applying the second-order Trotter formula to $\exp(-iLt)$ and $\exp(-iHt)$ gives the uniform result as
\begin{equation}
\begin{aligned}
    &\exp \left( i\gamma \left( \mathcal{H}_1(\lambda, n) + \mathcal{H}_2(\lambda, n) + \mathcal{H}_3(\lambda, n) \right)t \right) \\
    \approx& \exp \left( i \gamma \mathcal{H}_1(\lambda, n) \frac{t}{2} \right)
    \exp \left( i \gamma \mathcal{H}_2(\lambda, n) t \right)
    \exp \left( i \gamma \mathcal{H}_3(\lambda, n) t \right)
    \exp \left( i \gamma \mathcal{H}_1(\lambda, n) \frac{t}{2} \right).
\end{aligned}
\end{equation}

The first commutator for computing the Trotter error is given by
\begin{equation}\label{eq:1st_term_in_Periodic}
\begin{aligned}
    \left[ \mathcal{H}_2 + \mathcal{H}_3, [\mathcal{H}_2 + \mathcal{H}_3, \mathcal{H}_1] \right]
    =& -2 e^{3i\lambda} \left( \sigma_{10}^{\otimes(n-2)} + \sum_{j=1}^{n-2} I^{\otimes(n-2-j)} \otimes \sigma_{01} \otimes \sigma_{10}^{\otimes(j-1)} \right) \otimes I \otimes \sigma_{10} \\
    & +2 e^{-i\lambda} I^{\otimes(n-1)} \otimes \sigma_{10}
    +2 e^{i\lambda} I^{\otimes(n-1)} \otimes \sigma_{01} \\
    & -2 e^{-3i\lambda} \left( \sigma_{01}^{\otimes(n-2)} + \sum_{j=1}^{n-2} I^{\otimes(n-2-j)} \otimes \sigma_{10} \otimes \sigma_{01}^{\otimes(j-1)} \right) \otimes I \otimes \sigma_{01}.
\end{aligned}
\end{equation}
The eigenvalues of Eq.~\eqref{eq:1st_term_in_Periodic} are
\begin{equation}
    \lambda_k = 4 \sin \left( \frac{(k-1)\pi}{2^{n-1}} + 2\lambda \right), \quad k=1,2,\cdots,2^n,
\end{equation}
then we have
\begin{equation}\label{eq:norm of [H2+H3,[H2+H3,H1]]}
    \left\| \left[ \mathcal{H}_2 + \mathcal{H}_3, [\mathcal{H}_2 + \mathcal{H}_3, \mathcal{H}_1] \right] \right\|
    \leq 4.
\end{equation}

The second term takes the form
\begin{equation}\label{eq:2nd_term_in_Periodic}
\begin{aligned}
    \left[ \mathcal{H}_1, [\mathcal{H}_1, \mathcal{H}_2 + \mathcal{H}_3] \right]
    =& 2 e^{i \lambda} X \otimes \sigma_{10}^{\otimes (n-2)} \otimes \begin{pmatrix}
        & -e^{i 2\lambda} \\
        1 & \\
    \end{pmatrix}
    + 2 e^{-i \lambda} X \otimes \sigma_{01}^{\otimes (n-2)} \otimes \begin{pmatrix}
        & 1 \\
        -e^{-i 2\lambda} & \\
    \end{pmatrix} \\
    &+ 2 e^{i \lambda} \sum_{j=2}^{n-1} I^{\otimes(n-j)} \otimes \sigma_{01} \otimes \sigma_{10}^{\otimes(j-2)} \otimes \begin{pmatrix}
        & -e^{i 2\lambda} \\
        1 & \\
    \end{pmatrix} \\
    &+ 2 e^{-i \lambda} \sum_{j=2}^{n-1} I^{\otimes(n-j)} \otimes \sigma_{10} \otimes \sigma_{01}^{\otimes(j-2)} \otimes \begin{pmatrix}
        & 1 \\
        -e^{-i 2\lambda} & \\
    \end{pmatrix}.
\end{aligned}
\end{equation}
Similarly the eigenvalues of Eq.~\eqref{eq:2nd_term_in_Periodic} are
\begin{equation}
    \lambda_k = 4 \sin \left( \frac{(k-1)\pi}{2^{n-1}} + 2\lambda \right), \quad k=1,2,\cdots,2^n,
\end{equation}
then we have
\begin{equation}\label{eq:norm of [H1,[H1,H2+H3]]}
    \left\| \left[ \mathcal{H}_1, [\mathcal{H}_1, \mathcal{H}_2 + \mathcal{H}_3] \right] \right\|
    \leq 4.
\end{equation}

\section{Matrices with $\alpha=0$}\label{sec:BCs with alpha=0}

\subsection{Robin boundary conditions}\label{subsec:robin with alpha=0}

Under this condition, $L = -s_0 r_j \sigma_{00}^{\otimes n} - s_1 r_j \sigma_{11}^{\otimes n}$, then for $U_j(\tau) = \exp\left(-i(H + r_j L)\tau \right)$ we have
\begin{equation}\label{eq:Uj of robin with alpha=0}
\begin{aligned}
    U_j(\tau)
    =& \exp \left( -i \left( \beta \mathcal{H}_1 \left( \frac{\pi}{2}, n \right) + \beta \mathcal{H}_2 \left( \frac{\pi}{2}, n  \right) -s_0 r_j \sigma_{00}^{\otimes n} - s_1 r_j \sigma_{11}^{\otimes n} \right) \right) \\
    \approx& \exp\left(-i \beta \mathcal{H}_1\left(\frac{\pi}{2},n\right) \frac{\tau}{2} \right)
    \exp\left(-i \beta \mathcal{H}_2\left(\frac{\pi}{2},n\right)\tau \right) \\
    &\exp\left(i s_0 r_j \sigma_{00}^{\otimes n} \tau \right)
    \exp\left(i s_1 r_j \sigma_{11}^{\otimes n} \tau \right)
    \exp\left(-i \beta \mathcal{H}_1\left(\frac{\pi}{2},n\right) \frac{\tau}{2} \right) \\
    =& \exp\left(-i \beta \mathcal{H}_1\left(\frac{\pi}{2},n\right) \frac{\tau}{2} \right)
    \exp\left(-i \beta \mathcal{H}_2\left(\frac{\pi}{2},n\right)\tau \right)
    \exp\left(-i s_0 \widetilde{R} \sigma_{00}^{\otimes n} \tau \right) \\
    &\exp\left(-i (-s_0\Delta r) \sigma_{00}^{\otimes n} \tau j \right)
    \exp\left(-i s_1 \widetilde{R} \sigma_{11}^{\otimes n} \tau \right) \\
    &\exp\left(-i (-s_1\Delta r) \sigma_{11}^{\otimes n} \tau j \right)
    \exp\left(-i \beta \mathcal{H}_1\left(\frac{\pi}{2},n\right) \frac{\tau}{2} \right),
\end{aligned}
\end{equation}
where the associated Trotter error is bounded by
\begin{equation}\label{eq:error when alpha=0}
\begin{aligned}
    e_j
    =& \left\| U_j(\tau) - \widetilde{U}_j(\tau) \right\| \\
    \leq& \frac{\tau^3}{12} \left\| \left[ \beta \mathcal{H}_2 \left(\frac{\pi}{2}, n \right) - s_0 r_j \sigma_{00}^{\otimes n} - s_1 r_j\sigma_{11}^{\otimes n}, [ \beta \mathcal{H}_2 \left(\frac{\pi}{2}, n \right) - s_0 r_j \sigma_{00}^{\otimes n} - s_1 r_j\sigma_{11}^{\otimes n}, \beta \mathcal{H}_1 \left(\frac{\pi}{2}, n \right) ] \right] \right\| \\
    &+ \frac{\tau^3}{24} \left\| \left[ \beta \mathcal{H}_1 \left(\frac{\pi}{2}, n \right), [ \beta \mathcal{H}_1 \left(\frac{\pi}{2}, n \right), \beta \mathcal{H}_2 \left(\frac{\pi}{2}, n \right) - s_0 r_j \sigma_{00}^{\otimes n} - s_1 r_j\sigma_{11}^{\otimes n} ] \right] \right\| \\
    \leq& \begin{cases}
        \frac{|\beta|^3\tau^3}{2}, &|\lambda|_{\max}(B_j) \leq 2; \\
        \frac{\tau^3}{12} \left( 4 |\beta|^3 + |\beta| \sqrt{\beta^4 + s^4 r_j^4 - \beta^2 s^2 r_j^2} + |\beta|^3 \left( \frac{s |r_j|}{|\beta|} + \frac{|\beta|}{s |r_j|} \right) \right), &|\lambda|_{\max}(B_j) > 2.
    \end{cases}
\end{aligned}
\end{equation}
Here, we define $B_j(\pi/2, \mu_0, \mu_1) := (H + r_j L)/\beta$ with $\mu_0 = -s_0 r_j/\beta$ and $\mu_1 = -s_1 r_j/\beta$, and employ the following results:

If $|\lambda|_{\max}(B_j) \leq 2$, then by Eqs.~\eqref{eq:norm of [H2,[H2,H1]] in L when abs_lam <= 2} and \eqref{eq:norm of [H1,[H1,H2]] in L when abs_lam <= 2}, the first term $\mathcal{H}_j$ on the right-hand side of Eq.~\eqref{eq:error when alpha=0} satisfies $\|\mathcal{H}_j\| \leq 4 |\beta|^3$, and the second term $\mathcal{H}_j'$ satisfies $\|\mathcal{H}_j'\| \leq 4 |\beta|^3$.
If $|\lambda|_{\max}(B_j) > 2$, then from Eqs.~\eqref{eq:norm of [H2,[H2,H1]] in L when abs_lam > 2} and \eqref{eq:norm of [H1,[H1,H2]] in L when abs_lam > 2} we have
\begin{equation}
    \|\mathcal{H}_j\| \leq 4 |\beta|^3 + |\beta| \sqrt{\beta^4 + s^4 r_j^4 - \beta^2 s^2 r_j^2}, \quad
    \|\mathcal{H}_j'\| \leq 2 |\beta|^3 \left( \frac{s |r_j|}{|\beta|} + \frac{|\beta|}{s |r_j|} \right).
\end{equation}

By applying Lemmas~\ref{lem:select oracle}, \ref{lem:operator1}, Theorem~\ref{th:exp_sigma11} and Proposition~\ref{prop:exp_sigma00} and combining Eq.~\eqref{eq:Uj of robin with alpha=0}, we construct the quantum circuit \replytwo{with the shorthand block $\mathfrak{B}$ defined in Eq.~\eqref{eq:shorthand blocks}} as
\begin{equation}\label{eq:select under robin, alpha=0}
\begin{aligned}
    \mathrm{SEL_R}(\tau)
    =& \sum_{j=0}^{M - 1}|j\rangle\langle j| \otimes U_j(\tau) \\
    \approx& W_1\left(\frac{\beta\tau}{2},\frac{\pi}{2}\right)
    \left[ \prod_{l=2}^{n} W_l\left( \beta\tau,\frac{\pi}{2}\right) \right]
    \replytwo{\mathfrak{B}_{s_0\tau}\left[S_n^{(0)}\right]
    \mathfrak{B}_{s_1\tau}\left[S_n^{(1)}\right]}
    W_1\left(\frac{\beta\tau}{2},\frac{\pi}{2}\right).
\end{aligned}
\end{equation}

Considering the extreme situation that $|\lambda|_{\max}(B_j) > 2$ for all $j$, we find that
\begin{equation}
\begin{aligned}
    e_j
    &\leq \frac{\tau^3}{12} \left( 4 |\beta|^3 + |\beta| \sqrt{\beta^4 + s^4 r_j^4 - \beta^2 s^2 r_j^2} + |\beta|^3 \left( \frac{s |r_j|}{|\beta|} + \frac{|\beta|}{s |r_j|} \right) \right) \\
    &\leq \frac{\tau^3}{12} \left( 4 |\beta|^3 + |\beta| (\beta^2 + s^2r_j^2) + \beta^2 s r_j + |\beta|^3 \right) \\
    &= \frac{|\beta|^3 \tau^3}{2} + \frac{\beta^2 s \tau^3}{12} r_j + \frac{|\beta| s^2 \tau^3}{12} r_j^2,
\end{aligned}
\end{equation}
where the relaxation $|\beta|^3 \cdot |\beta|/(s|r_j|) \leq |\beta|^3$ in the second inequality relies on $s|r_j| > |\beta|$: by the critical-curve analysis in Appendix~\ref{subsec:poly case} (Eqs.~\eqref{eq:lam=2} and~\eqref{eq:lam=-2}), $|\lambda|_{\max}(B_j) > 2$ implies $\max\{|\mu_0|, |\mu_1|\} > 1$, i.e., $s|r_j|/|\beta| > 1$.
Writing the above bound as $e_j \leq (w_0 + w_1 r_j + w_2 r_j^2)\tau^3$ with $w_0 = |\beta|^3/2$, $w_1 = \beta^2 s/12$ and $w_2 = |\beta|s^2/12$, then similarly to the analysis of Theorem~\ref{th:circuit error under robin}, the circuit error between $\sum_{j=0}^{M-1} c_j U_j(\tau)$ and $\sum_{j=0}^{M-1} c_j \widetilde{U}_j(\tau)$ satisfies
\begin{equation}
\begin{aligned}
    \left\| \sum_{j=0}^{M-1} c_j U_j(\tau) - \sum_{j=0}^{M-1} c_j \widetilde{U}_j(\tau) \right\|
    &\leq \sum_{j=0}^{M-1} |c_j| e_j
    \replythr{\leq \tau^3 \sum_{k=0}^{2} I_k^{(M)}(\gamma,\delta)\, w_k
    \leq \tau^3 \sum_{k=0}^{2} \left( I_k(\gamma,\delta) + \varepsilon_{k,M} \right) w_k}\\
    &\replythr{\leq 2 \tau^3 \sum_{k=0}^{2} I_k(\gamma,\delta)\, w_k
    = |\beta| \tau^3 \left( \beta^2 I_0(\gamma, \delta) + \frac{|\beta| s}{6} I_1(\gamma, \delta) + \frac{s^2}{6} I_2(\gamma, \delta) \right).}
\end{aligned}
\end{equation}

\subsection{Periodic boundary conditions}\label{subsec:periodic with alpha=0}

In this case, $L \equiv O$, meaning the system reduces to a Schr\"{o}dinger equation; this further yields $A = iH$ and $\exp(-A\tau) = \exp\left(-iH\tau\right)$.

Based on the analysis in Section~\ref{subsubsec:quantum circuit under the periodic boundary conditions} and Lemmas~\ref{lem:operator1} and \ref{lem:operator2}, we construct the quantum circuit for the "select oracle" of the Schr\"{o}dinger equation (referred to as "select oracle" for consistency with prior circuit implementations) \replytwo{with the shorthand $\mathcal{U}_H^{\mathrm{P}}$ defined in Eq.~\eqref{eq:shorthand periodic block}} as

\begin{equation}\label{eq:select under periodic with alpha=0}
    \exp\left(-iH\tau\right)
    \approx \replytwo{\mathcal{U}_H^{\mathrm{P}}\left(\beta\tau, \frac{\pi}{2}\right)}
\end{equation}
with the Trotter error bounded by $\frac{|\beta|^3 \tau^3}{2} = \frac{|a|^3\tau^3}{16h^3}$.

\bibliographystyle{unsrt}
\bibliography{reference}

\end{document}